%% file: A_packet_of_unitary_group.tex
\documentclass[10pt]{article}
\usepackage[
	paper=a4paper,
	top=3cm,
	inner= 2.54cm,
	outer= 2.54cm,
	bottom= 3cm,
	headheight=5ex,
	headsep=5ex]{geometry}
\usepackage{enumitem}
\usepackage{amsmath, amssymb}
	\numberwithin{equation}{section} 
\usepackage{tikz-cd}
	\usetikzlibrary{decorations.pathreplacing}
	\newcommand{\tikznode}[3][inner sep=0pt]{\tikz[remember picture,baseline=(#2.base)]{\node(#2)[#1]{$#3$};}}
\usepackage{ytableau}
\usepackage{mathdots}
\usepackage{natbib}
	\setcitestyle{numbers,square}
\usepackage[bookmarks,colorlinks ]{hyperref}
	\hypersetup{linkcolor =blue,
	filecolor=red}

\usepackage{amsthm}
\newtheorem{thm}{Theorem}[section]
\newtheorem{prop}[thm]{Proposition}
\newtheorem{lem}[thm]{Lemma}

\theoremstyle{definition}
\newtheorem{din}[thm]{Definition}
\newtheorem{hyp}[thm]{Hypothesis}

\theoremstyle{remark}
\newtheorem{rmk}{Remark}[subsection]
\newtheorem{eg}{Example}

\newcommand{\N}{\mathbb N}
\newcommand{\Z}{\mathbb Z}
\newcommand{\R}{{\mathbb R}}
\newcommand{\C}{{\mathbb C}}
\newcommand{\Id}{{\operatorname{Id}}}
\newcommand{\GL}{{\mathrm {GL}}}
\newcommand{\U}{{\mathrm U}}
\newcommand{\SL}{{\mathrm {SL}}}
\newcommand{\Sp}{{\mathrm {Sp}}}
\newcommand{\SO}{{\mathrm {SO}}}
\newcommand{\g}{{\mathfrak g}}
\newcommand{\h}{{\mathfrak h}}
\newcommand{\q}{{\mathfrak q}}
\newcommand{\p}{{\mathfrak p}}

\renewcommand{\l}{{\mathfrak l}}
\renewcommand{\k}{{\mathfrak k}}
\newcommand{\m}{{\mathfrak m}}
\newcommand{\n}{{\mathfrak n}}
\renewcommand{\t}{{\mathfrak t}}

\newcommand{\gl}{{\mathfrak {gl}}}
\renewcommand{\u}{{\mathfrak u}}
\newcommand{\la}{\left \langle}
\newcommand{\ra}{\right \rangle}

\begin{document}
\title{Non-zero condition on M{\oe}glin-Renard's parametrization for Arthur packets of $\U(p,q)$}
\author{Chang Huang}
\maketitle

\newcommand{\keywords}[1]{\textbf{Keywords:} #1}
\begin{abstract}
	M{\oe}glin-Renard parametrized A-packet of unitary group through cohomological induction in good parity case.
	Each parameter gives rise to an $A_\q(\lambda)$ which is either $0$ or irreducible.
	Trapa proposed an algorithm to determine whether a ``mediocre'' $A_\q(\lambda)$ of $\U(p, q)$ is non-zero. 
	Based on his result, we present a further understanding of the non-zero condition on M{\oe}glin-Renard's parametrization.
	Our criterion comes out to be a system of linear constraints, and has the same formulation as $p$-adic case.
	This suggests a map from A-packets of real unitary group to A-packets of $p$-adic symplectic group or special orthogonal group.
	
\end{abstract}

\tableofcontents
\pagestyle{plain}

\input{Introduction}


\newcommand{\FDR}{{\mathcal S}}
\newcommand{\gp}{{\mathrm g}}
\newcommand{\bp}{{\mathrm b}}
\newcommand{\Herm}{{\mathrm H}}
\newcommand{\bg}{{\operatorname b}}
\newcommand{\ed}{{\operatorname e}}
\newcommand{\Seg}{{\mathrm{Seg}}}
\newcommand{\MR}{{\mathrm {MR}}}
\input{Notations}


\newcommand{\AV}{{\mathrm{AV}}}
\newcommand{\an}{{\mathrm{an}}}
\newcommand{\MRV}{{\mathcal D}}
\newcommand{\ES}{{\mathcal {E}}}
\newcommand{\sgn}{{\mathrm{sgn}}}
\input{Comments}


\newcommand{\Speh}{{\mathrm{Speh}}}
\newcommand{\Std}{{\mathrm {Std}}}
\newcommand{\Ind}{{\mathrm{Ind}}}
\renewcommand{\Re}{{\mathrm{Re}}}
\newpage

\newcommand{\ind}{{\mathrm{ind}}}
\input{Proof-3}

\newcommand{\overlap}{{\mathrm {overlap}}}
\newcommand{\sing}{{\mathrm {sing}}}
\newcommand{\anti}{{\mathrm A}}
\newcommand{\sign}{\pm}
\input{Method}


\input{Proof-1}


\input{Proof-2}

\bibliographystyle{alpha}
\bibliography{ref}

\vspace{1em}
\begin{flushleft} \small
	Chang Huang: Yau Mathematical Sciences Center, Tsinghua University, Haidian District, Beijing 100084, China. \\
	E-mail address: \href{mailto:hc21@mails.tsinghua.edu.cn}{\texttt{hc21@mails.tsinghua.edu.cn}}
\end{flushleft}

\end{document}

%% file: Introduction.tex
\section{Introduction}
	Local A-packet plays an important role in the endoscopic classification for automorphic representations.
	It is determined by endoscopic character identities (\cite{Arthur13, Mok15}),
	thus difficult to further understand.
	For example, an A-packet by definition could have multiplicity,
	and it is an important problem to determine whether this multi-set is in fact multiplicity-free.
	
	To make the A-packet more explicit,	
	M{\oe}glin (\cite{Moeglin06-summary, Moeglin06, Moeglin09, Moeglin11}) and M{\oe}glin-Renard (\cite{MR18, MR, MR20}) has consecutively worked on its construction.
	For real classical groups, they obtained the following result:
	\begin{itemize}
	\item by parabolic induction, one can construct a general A-packet from an A-packet of good parity,
	where the multiplicity one property is preserved;
	\item by cohomological induction, one can construct an A-packet of good parity from a unipotent A-packet.
	\end{itemize}
	Note that an A-packet is associated with an A-parameter,  to which the conditions ``of good parity'' and ``unipotent'' actually apply.
	
	For real unitary groups, their result in \cite{MR} is simpler.
	Its A-packet of good parity is constructed directly by cohomological induction. 
	Generally, the multiplicity one property could be lost in cohomological induction, unless the A-parameter of good parity is ``very regular''.
	However, for unitary group it is obtained without the assumption of regularity.
	In this case, the induction process starts from characters of a $\theta$-stable Levi lying in the ``weakly fair range'';
	due to an irreducibility result special for unitary group, it produces either irreducible representations or $0$.
	This is the origin of multiplicity one property.
	
	Nevertheless, it's difficult to detemine whether the cohomological induction from ``weakly fair range'' does not vanish.
	Trapa (\cite{Trapa}) proposed a criterion algorithm,
	but highly involved with combinatorics over Young diagrams.
	In this article, we would sduty his criterion in detail, 
	and deduce a new one (Theorem \ref{non-vanishing}) for M{\oe}glin-Renard's parametrization.

	At first glance, our new criterion is not simpler;
	however, its value lies in the similarity with what has happened on the $p$-adic side.
	For $p$-adic symplectic groups and orthogonal groups, 
	\cite{Atobe22} summarizes the parametrization and non-zero criterion of their A-packets.
	Here we sketch that for a non-negative A-parameter (teminology introduced in \cite{Atobe22}) as follows:
\begin{itemize}
	\item this A-parameter can be decomposed into irreducible components, each of which corresponds to a segement and provides a component to the parametrization of A-packet
	\item each pair of adjacent segements in the decomposition provides a necessary condition to the criterion;
	\item decomposition of A-parameter could be different, resulting in different admissible arrangements/orders of segements, 
	but equivalent parametrizations (by a transition map);
	\item those necessary conditions,
	 provided by all pairs of segements that are adjacent in some admissible arrangements/orders,
	 compose into the sufficient condition in the criterion. 
\end{itemize}
	Our criterion (Theorem \ref{non-vanishing}) is formally the same as above. 
	We can even write out a correspondence between the combinatorial parameters on two sides, under which those two criteria are equivalent!
	
	The correspondence on the combinatoric level suggests a correspondence on the representation level.
	We may expect a representation theoretical map from A-packets of real unitary groups to A-packets of $p$-adic symplectic groups or special orthogonal groups;
	under this map the corresponded representations would have similar Langlands quotient/submodule parameters.
	A recent project of the author with Taiwang Deng, Bin Xu and Qixian Zhao would support this from the geometric perspective.
	
	Here is an outline of this article.
	In section \ref{result} we breifly review M{\oe}glin-Renard's construction for A-packets of unitary groups, 
	and prepare necessary notions about segements and their admissible arrangements for our non-vanishing criterion.
	In section \ref{compare}, we will compare our criterion with that formulated in \cite[Section 4]{Atobe22}.
	This will lead to a map of Arthur-Vogan packets from real side to $p$-adic side.
	Section \ref{sec: remained}, \ref{tableau}, \ref{swap-tableau}, \ref{sufficiency} are devoted to the proof of Theorem \ref{non-vanishing}.

%% file: Notations.tex
\section{Notations and results}\label{result}
\subsection{A-parameter}
	Let $G^*$ be the quasi-split real unitary group of rank $n\geqslant 1$, more explicitly, $\U(\frac n 2, \frac n 2)$ if $n$ is even and $\U(\frac {n-1}2, \frac{n+1} 2)$ if $n$ is odd.
	In this subsection we will review relevant notions about A-parameters of $G^*$.
	
	The Langlands dual group of $G^*$ is $\hat G= \GL(n, \C)$, with an action of $W_\R$ (after fixing a $\R$-splitting of $G^*$).
	An A-parameter of $G^*$ is a $\hat G$-conjugacy class of admissible homomorphism $\psi: W_\R \times \SL(2, \C) \to \hat G \rtimes W_\R$.
	By abuse of notation we may denote by $\psi$ either an admissible homomorphism or its $\hat G$ conjugacy class.
	\cite[lemma 2.2.1]{Mok15} shows that $\psi$ can be determined by its base change to $W_\C$:
	\[\begin{tikzcd}
	{W_\C\times \SL(2, \C)} \arrow[r] \arrow[d, hook] \arrow[rr, "\psi_\C", bend left] & {\GL(n,\C) \times W_\C} \arrow[d, hook] \arrow[r] & {\GL(n, \C)} \\
	{W_\R \times \SL(2, \C)} \arrow[r, "\psi"]                                          & {\GL(n, \C) \rtimes W_\R}                         &             
\end{tikzcd},\]
	where the square is a pull-back.
	Moreover, this base change $\psi_{\C}$ can be regarded as a semisimple representation, which is conjugate self-dual with parity $(-1)^{n-1}$. 
	
	Decompose the representation $\psi_\C$ into irreducible components as
	\[ \psi_\C = \bigoplus_{i=1}^r \rho_i \boxtimes \FDR_{m_i},\]
	where each $\rho_i$ is a character of $W_\C= \C^\times$ and $\FDR_{m_i}$ is the $m_i$-dimensional representation of $\SL(2, \C)$.
	Following \cite{MR} we say an irreducible representation $\rho \boxtimes \FDR_{m}$ of $W_\C \times \SL(2, \C)$ has good parity if it is conjugate self-dual.
	In this case $\rho$ takes the form $z^{\frac{a}2} \bar z^{-\frac{a}2}$, and then $\rho\boxtimes \FDR_m$ has parity $(-1)^{a+m -1}$.
	Consequently, if $\psi_\C$ has a component $(\frac z {\bar z})^{\frac {a_i} 2} \boxtimes \FDR_{m_i}$ of good parity, then due to the parity consistency, $a_i + m_i \equiv n \mod 2$.
	On the other hand, if $\psi_\C$ has a component $\rho_i \boxtimes \FDR_{m_i}$ of bad parity, then there must be another component $\rho_j \boxtimes \FDR_{m_j}$ as its conjugate dual, 
	such that $\rho_i \boxtimes \FDR_i \oplus \rho_j \boxtimes \FDR_j$ is conjugate self-dual, and can be equipped with parity consistent with $\psi_\C$.
	In conclusion, there is a (unique) decomposition $\psi = \psi_\gp \oplus \psi_\bp$ such that $\psi_{\mathrm g, \C}$ consists of all components of $\psi_\C$ having good parity.

\subsection{Construction and parametrization of A-packet}\label{subsec: parametrize-A-packet}
	Let $G= \U(p,q)$ be a real unitary group of rank $n$.
	In this subsection we will review M{\oe}glin-Renard's construction for A-packets of $G$.
	The A-packet theory for non-split group $G$ is developed in \cite{MR-non-split}.
	
	Since $G$ lies in the same inner class as $G^*$, they share the same A-parameters.
	Let $\psi$ be an A-parameter of $G$, with the decomposition $\psi = \psi_\gp \oplus \psi_\bp$ according to parity.
	The base change of $\psi_\bp$ decompose further into $\psi_\bp' \oplus \psi_\bp'^\vee$,
	where $\psi_\bp'^\vee$ is the conjugate dual of $\psi_\bp'$.
	Moreover, $\psi_\bp': W_\C \times \SL(2, \C) \to \GL(n_\bp', \C)$ can be regarded as an A-parameter of $\GL(n_\bp', \C)$, and 
	uniquely determines a representation $\rho(\psi_\bp')$ via the local Langlands correspondence for $\GL(n_\bp', \C)$.
\begin{thm}[{\cite[Theorem 5.3]{MR}}]
	Denote the A-packet of $G$ attached to $\psi$ by $\Pi_\psi(G)$.
	There is a bijection:
	\[\Pi_{\psi_\gp}(G_\gp) \to \Pi_\psi(G), \quad 
	\pi \mapsto \rho(\psi_\bp') \rtimes \pi,  \]
	where $G_\gp = \U(p- n_\bp', q- n_\bp')$.
	In particular, if $\Pi_{\psi_\gp}(G_\gp)$ is multiplicity free, then so is $\Pi_\psi(G)$.
\end{thm}

	The above result reduces the construction of A-packet to the case of good parity.
	From now on we may always denote by $\psi$ an A-parameter of $G$ with base change
	\[\psi_{\C} = \bigoplus_{i=1}^r (\frac z{\bar z} )^{\frac {a_i}2} \boxtimes \FDR_{m_i}.\]
	To parametrize $\Pi_\psi(G)$, \cite{MR} introduces the set
	\[\mathcal D(\psi): = \{\underline p:=(p_1, \cdots, p_r )\in \N^r \mid 0 \leqslant p_i \leqslant m_i, \, p_1 + \cdots + p_r=p \}.\]
	For $\underline p \in \mathcal D(\psi)$, there is a $\theta$-stable parabolic subalgebra $\q_{\underline p} \subseteq \g = \gl(n, \C)$ with stabilizer
	\[N_G(\q_{\underline p}) = : L_{\underline p} \cong  \prod_{i=1}^r \U(p_i, m_i - p_i);\]
	define integers
\begin{equation}\label{lambda-nu}
	\lambda_i = \frac{a_i + m_i - n} 2 + \sum_{j< i} m_j,
\end{equation}
	such that $\boxtimes_{i=1}^r \det^{\lambda_i}$ becomes a character $\lambda_{\underline p}$ of $L_{\underline p}$ via the above isomorphism;
	then we can take the cohomologically induced module $A_{\q_{\underline p}}(\lambda_{\underline p})$ following the notation in \cite[(5.6)]{KV}.
	
	Here is a detailed definition for $\q_{\underline p}$.
	The unitary group $G= \U(p, q)$ is a real form of $\GL(n, \C)$ determined by the Cartan involution 
	\[\theta(g) := \begin{pmatrix}
	I_p 	& \\
		& -I_q
	\end{pmatrix}
	g 
	 \begin{pmatrix}
	I_p 	& \\
		& -I_q
	\end{pmatrix}.\]
	Let $\t \subseteq \g = \gl(n, \C)$ consist of diagonal matrices.
	Its root system $\Delta(\g, \t)$ consists of $e_i - e_j$, $1\leqslant i \not =j \leqslant n$, 
	where $e_i$ denotes the $i$-th coordinate funcition of $\t \cong \C^n$.
	Let $t_{\underline p} \in \C^n \cong \t^*$ be the following element:
	\[(\underbrace{r, \cdots, r}_{p_1};
	\cdots;
	\underbrace{1, \cdots, 1}_{p_r};
	\overbrace{1, \cdots, 1}^{q_r};
	\overbrace{2, \cdots, 2}^{q_{r-1}};
	\cdots;
	\overbrace{r \cdots, r}^{q_1} ),\]
	where $q_i = m_i - p_i$.
	Then take
	\[\l_{\underline p} = \t \oplus \bigoplus_{\alpha \in \Delta(\g, \t), \alpha( t_{\underline p}) = 0} \g_\alpha, \quad
	\q_{\underline p} = \t \oplus \bigoplus_{\alpha \in \Delta(\g, \t), \alpha( t_{\underline p})\geqslant 0} \g_\alpha, \quad
	L_{\underline p}= N_G(\q_{\underline p})\]
	as \cite[(3.1)]{MR}.
	It's easy to see $\l_{\underline p}$ is exactly the complex Lie algebra of $L_{\underline p}$, 
	and check $L_{\underline p} \cong \prod_{i=1}^r \U(p_i, q_i)$. 
	
	There are serval range conditions on $(\q, \lambda)$ for the irreducibility and non-vanishing of $A_\q(\lambda)$.
\begin{din}
	Suppose $\q = \l \oplus \u$.
	Choose a Cartan $\t \subseteq \l$, and a positive root system $\Delta^+(\g, \t) \subseteq \Delta(\g, \t)$ containing $\Delta(\u, \t)$.
	Let $\lambda \in \t^*$ represents the character $\lambda$ of $L$, $\delta_G$ and $\delta(\u)$ be half the sum of $\Delta^+(\g, \t)$ and $\Delta(\u, \t)$. 
	The pair $(\q, \lambda)$ is said to be in the
\begin{itemize}
	\item ``good range'' if $\la \lambda + \delta_G, \alpha \ra >0$ for all $\alpha \in \Delta(\u, \t)$;
	\item ``weakly fair range'' if $\la \lambda + \delta(\u), \alpha \ra \geqslant 0$ for all $\alpha \in \Delta(\u, \t)$.
\end{itemize}
\end{din}
\begin{rmk}
	The above definition is independent of the choice of $\t$ and $\Delta^+(\g, \t)$, since both $\lambda$ and $\delta(\u)$ is stable under the action of $W(\l, \t)$.
\end{rmk}

	Specialize to the group $G= \U(p, q)$, then the range conditions for $(\q_{\underline p}, \lambda_{\underline p})$ are described as follows:
	$(\q_{\underline p}, \lambda_{\underline p})$ lies in the
	\begin{itemize}
	\item ``good range'' iff $\lambda_i - \lambda_{i+1} \geqslant 0$ for $i =1, \cdots, r-1$;
	
	\item ``weakly fair range'' iff  $\lambda_i - \lambda_{i+1} \geqslant -\frac{m_i + m_{i+1}}2$ for $i =1, \cdots, r-1$.
	\end{itemize}
	Together with \eqref{lambda-nu}, we see each condition depends only on the decomposition of $\psi_\C$, rather than $\underline p$.
	 \cite[Section 4]{MR} fixes the order of indices $k$ such that 
\begin{equation}\label{MR-order}
	a_1 \geqslant \cdots \geqslant a_r, \textrm{ and }
	m_i \geqslant m_{i+1} \textrm{ whenever } a_i= a_{i+1}.
\end{equation}  
	Under this choice, $(\q_{\underline p}, \lambda_{\underline p})$ always lies in the ``weakly fair range''.
	Then $A_\q(\lambda)$ is either zero or irreducible according to \cite[Theorem 3.1 and Lemma 3.5]{Trapa}; 
	this irreducibility result is special for unitary group. 
\begin{thm}[{\cite[Theorem 4.1]{MR}}]
	The map $\mathcal D(\psi)\to \Pi_\psi(G) \sqcup \{0\}$, $\underline p \mapsto A_{\q_{\underline p}}(\lambda_{\underline p})$ is bijective over $\Pi_\psi(G)$.
	In particular, $\Pi_\psi(G)$ is multiplicity free.
\end{thm}

	Here is a remark on the choice \eqref{MR-order}.
	It is not necessary for the irreducibility of $A_\q(\lambda)$.
	According to \cite[Theorem 3.1(b) and Definition 3.4]{Trapa}, we may relax the condition by requiring those $(\q, \lambda)$ lying in the ``mediocre range'':
	\begin{itemize}
	\item $\lambda_i - \lambda_j \geqslant - \max\{ m_i, m_j\} - \sum_{i< k < j} m_k$ for $i < j$. 
	\end{itemize}
	Then more decompositions of $\psi_\C$ appear into consideration.
	To make it concise, we will introduce the subsequent combinatorial notions about segements and their admissible orders.

\subsection{Admissible representatives of infinitesimal character}
	Recall $\g=\gl(n, \C)$ is the complexfication of Lie algebra of $G= \U(p, q)$,
	and $\t \subseteq \g$ consists of diagonal matrices.
	According to the Harish-Chandra isomorphism, an infinitesimal character of $\gl(n, \C)$ corresponds uniquely to a Weyl group-orbit in $\t^*$.
	By the identification $\t^* \cong \C^n$, we may represent such an infinitesimal character by an $n$-tuple of complex numbers.
	
	For the A-parameter $\psi$ of good parity with base change
	\[\psi_{\C} = \bigoplus_{i=1}^r (\frac z{\bar z} )^{\frac {a_i}2} \boxtimes \FDR_{m_i},\]
	its infinitesimal character $\nu$ has the form $(\nu_1, \cdots, \nu_r)$, where each $\nu_i = (\frac{a_i+m_i -1} 2, \cdots, \frac{a_i - m_i +1}2)$ is a decreasing sequence with center $a_i$, length $m_i$ and step one.	
\begin{din}
	By \textbf{segment}, we mean a decreasing sequence of complex numbers with step one. 
	For a segement $\nu$, we denote its beginning (the maximum) by $\bg(\nu)$, and its end (the minimum) by $\ed(\nu)$. 
	Then we may also denote this segement by $[\bg(\nu), \ed(\nu)]$.
\end{din}

	It's clear that the irreducible component $(\frac z{\bar z} )^{\frac {a_i}2} \boxtimes \FDR_{m_i}$ and segement $\nu_i = (\frac{a_i+m_i -1} 2, \cdots \frac{a_i - m_i +1}2)$ determine each other.
	Then the (multi-)set of segments $\Seg_\psi := \{ \nu_1, \cdots, \nu_r \}$ records irreducible components of $\psi_\C$ and hence determines $\psi$.
	A permutation in the decomposition of $\psi_\C$ is equivalent to a rearrangement of $\Seg_\psi$, 
	and a permutation of $\nu=(\nu_1, \cdots, \nu_r)$.
	It also leads to a reprametrization of $\mathcal D(\psi) \to \Pi_\psi(G) \sqcup \{0\}$.
	For our non-zero criterion, we may consider ``admissible'' permutations of $\nu= (\nu_1, \cdots, \nu_r)$ that all $(\q, \lambda)$ involved lie in the ``mediocre range''.
	The cardinality of admissible permutations depends on the relative positions of segments in $\mathrm{Seg}_\psi$.
	Here follows the definition.
	
\begin{din}
	A segment $\nu$ is said to \textbf {precede} $\mu$ if $\bg(\nu) > \bg (\mu)$ and $\ed(\nu) > \ed(\mu)$, and denoted by $\nu > \mu$.  
	If we replace all ``>'' by ``$\geqslant$'', then we say $\nu$ weakly precedes $\mu$.
	
	If two segements $\nu$, $\mu$ don't precede each other, then we say they are in \textbf{containment} relation.
\end{din}
\begin{rmk}
	The containment relation consists of three possibilities: $\nu \subset \mu$, $\nu \supset \mu$ or $\nu = \mu.$
	When considering segments in $\mathrm{Seg}_\psi$, it will be more convenient to ``exclude'' the possibility $\nu = \mu$;
	we will discuss this issue in the next subsection, see remark \ref{rmk: deduplicate}.
\end{rmk}

\begin{din}\label{adm-arrangement}
	An arrangement $(\mu_1, \cdots, \mu_r)$ of $\mathrm{Seg}_\psi$ is said to be \textbf{admissible} if there is no such a pair $i<j$ that $\mu_i < \mu_j$. 
	The set of admissible arrangements of $\mathrm{Seg}_\psi$ would be denoted by $\Sigma_\psi$.
    
	A permutation $\sigma \in S_r$ is said to be \textbf{admissible} at $\nu \in \Sigma_\psi$ if the arrangement $\nu^\sigma:= (\nu_{\sigma(1)} , \cdots, \nu_{\sigma(r)} )$ is admissible. 
	The set of admissible permutations (at $\nu$) would be denoted by $\Sigma_r$.
\end{din}
\begin{rmk}
	We always fix a choice of $\nu$, and consider $\Sigma_r$ instead of $\Sigma_\psi$.
	While readers will notice that our choice varies according to situations,
	this variability does not affect the validity of our argument.
\end{rmk}

	In M{\oe}glin-Renard's parametrization of $\Pi_\psi(G)$, the character $\lambda_{\underline p}$ (at least as a tuple of numbers) is determined by $\nu \in \C^n \cong \h^*$ via \eqref{lambda-nu}, 
	so the range condition on $(\q_{\underline p}, \lambda_{\underline p})$ can also be put on the arrangement of segements $\nu= (\nu_1, \cdots, \nu_r)$.
	According to the explicit description in \cite[Lemma 3.5 and Definition 6.10]{Trapa}, one can check easily that 
	$(\q_{\underline p}, \lambda_{\underline p})$ lies in the 
	\begin{itemize}
	\item ``good range'' iff $\ed(\nu_i) > \bg(\nu_{i+1})$, $i = 1, \cdots, r-1$,
	\item ``weakly fair range'' iff $\frac{\bg(\nu_i)+\ed(\nu_i)}2 \geqslant 
	\frac{\bg(\nu_{i+1}) + \ed(\nu_{i+1}) }2$, $i=1, \cdots, r-1$,
	\item ``nice range'' iff  if $\nu_1 \geqslant \cdots \geqslant \nu_r$, and
	\item ``mediocre range'' iff $\nu= (\nu_1, \cdots, \nu)$ is an admissible arrangement.
	\end{itemize}
	
	Consequently, for any admissible permutation $\sigma \in \Sigma_r$ at $\nu$, take the representative $\psi^\sigma$ of A-parameter $\psi$ with base change
	\[ \psi^\sigma_\C = \bigoplus_{i=1}^r (\frac z {\bar z})^{\frac {a_{\sigma(i)}} 2} \boxtimes \FDR_{\sigma(i)}, \]
	and we can attach a cohomologically induced representation $A_{\q_{\underline p^\sigma}}(\lambda_{\underline p^\sigma})$ to each $\underline p^\sigma \in \mathcal D(\psi^\sigma)$
	through M{\oe}glin-Renard's construction, 
	Since these $(\q_{\underline p^\sigma}, \lambda_{\underline p^\sigma})$ lie in ``mediocre range'' due to the admissibility of $\sigma$, 
	their cohomological induction $A_{\q_{\underline p^\sigma}}(\lambda_{\underline p^\sigma})$ would also be either zero or irreducible, 
	due to \cite[Theorem 3.1 and Definition 3.4]{Trapa}.
	The set of non-zero representations $A_{\q_{\underline p^\sigma}}(\lambda_{\underline p^\sigma})$ turns out to be independent of $\sigma \in \Sigma_r$, as shown in section \ref{swap-tableau}.
	Hence, there is a new parametrization $\mathcal D(\psi^\sigma) \to \Pi_{\psi}(G) \sqcup \{0\}$ for each $\sigma \in \Sigma_r$, and
	we will take them into considerations.

	In the last of this subsection, we study the structure of $\Sigma_r \subseteq S_r$. 
	As a Coxeter group, $S_r$ is generated by the simple reflections $s_i:=(i, i+1), i =1, \cdots, r-1$. 
	They will also be called transpositions. 
	Admissible permutations have a samilar structure in the following sense.
\begin{lem}\label{decompose-of-perm}
	$\forall \sigma \in \Sigma_r$, there is a sequence $1= \sigma_0, \cdots, \sigma_l =\sigma$ in $\Sigma_r$ such that $\forall i =1, \cdots, l$, $\sigma_{i-1}^{-1} \sigma_i$ is a transposition. 
	We can require moreover that $l$ is exactly the length of $\sigma$ (in the Coxeter group $S_r$).
\end{lem}
\begin{rmk}
	The admissible permutations at $\nu^\sigma$ are exactly given by $\Sigma_r^\sigma:= \sigma^{-1} \Sigma_r$.
\end{rmk}
\begin{proof} 
	It follows from an easy induction on the length $l(\sigma)$ of $\sigma$;
	note that $l(\sigma)$ is also the inversion number of $\sigma$.
	
	If $l(\sigma)\geqslant 1$, then there is a maximum in $\{ k \mid \sigma(k)>k\}$. 
	Denote it by $h$, and we must have $\sigma(h) > \sigma(h+1)$,
	since $\sigma(h+1) \leqslant h+1$ and $ h+1 \leqslant \sigma(h)$ by the definition of $h$.
	Now in $\nu$, $\nu_{\sigma(h)}$ dose not precede $\nu_{\sigma(h+1)}$.
	On the other hand, in $\nu^\sigma$, $\nu_{\sigma(h+1)} = \nu^\sigma_{h+1}$ dose not precede $\nu^\sigma_h = \nu_{\sigma(h) }$.
	Hence $\nu^\sigma_h$ and $\nu^\sigma_{h+1}$ are in containment relation, 
	and $\sigma_{l-1}:= \sigma \circ (h, h+1)$ lies in $\Sigma_r$.
	The inversion number of $\sigma_{l-1}$ is strictly smaller than $\sigma_l := \sigma$,
	so we may repeat this process to obtain $\sigma_0$ such that $l(\sigma_0)=0$. 
	The permutation without inversions must be identity. 
	Now the conclusion follows.
\end{proof}

\subsection{Admissible action on the generalized parameter space}\label{subsec: transition}
	For convenience, we enlarge $\mathcal D(\psi^\sigma)$ to the space:
	\[\MR^\sigma= \{ \underline p^\sigma \in \Z^r \mid p_1^\sigma + \cdots + p_r^\sigma =p\}.\]
	Since we have claimed $\forall \sigma \in \Sigma_r$, $\mathcal D(\psi^\sigma)$ parametrizes the same A-packet $\Pi_\psi(G)$, 
	it is natural to ask a transition map $\phi_\sigma^\tau: \MR^\sigma \to \MR^\tau$ such that 
	$\underline p^\sigma$ and $\underline p^\tau:= \phi_\sigma^\tau (\underline p^\tau)$ give rise to isomorphic representations. 
	Here we extend M{\oe}glin-Renard's parametrization $\mathcal D(\psi^\sigma) \to \Pi_{\psi}(G)\sqcup\{0\}$ to $\MR^\sigma \to \Pi_{\psi}(G)\sqcup\{0\}$ by zero extension.
	
	If $s=(i, i+1)$ is a transposition in $\Sigma_r^\sigma = \sigma^{-1} \Sigma_r,$ and $\nu^\sigma_i \not= \nu^\sigma_{i+1}$, the desired map $\phi_\sigma^{\sigma s}:\MR^\sigma \to \MR^{\sigma s}$ can be given by
\begin{equation}\label{transition}
	\phi_\sigma^{\sigma s}(\underline p^\sigma) =
	\left\{ \begin{aligned}
	&	(p_1^\sigma, \cdots,q_{i+1}^\sigma, p_i^\sigma + p_{i+1}^\sigma - q_{i+1}^\sigma, \cdots, p_r^\sigma)
	&	\textrm{if } \nu^\sigma_i \supset \nu^\sigma_{i+1},\\
	&	(p_1^\sigma, \cdots, p_i^\sigma + p_{i+1}^\sigma - q_i^\sigma, q_i^\sigma , \cdots, p_r^\sigma)
	&	\textrm{if } \nu^\sigma_i \subset \nu^\sigma_{i+1},
	\end{aligned}\right.
\end{equation}
	where $q_i^\sigma = m_{\sigma(i)}-p_i^\sigma$.
	We will prove $A_{\q_{\underline p^\sigma}}(\lambda_{\underline p^\sigma}) \cong 
	A_{\q_{\underline p^{\sigma s}}}(\lambda_{\underline p^{\sigma s}})$
	in subsubsection \ref{easy-approach} and section \ref{swap-tableau};
	this is also suggested by \cite[Lemma 9.3]{Trapa}.
	
	We have to generalize this definition to $\phi_\sigma^\tau$ for abitary $\sigma, \tau \in \Sigma_r$.
	According to Lemma \ref{decompose-of-perm}, there is a sequence $\sigma = \sigma_0, \cdots, \sigma_l = \tau$ in $\Sigma_r$ such that $\sigma_{i-1}^{-1}\sigma_i$ is a transposition (in $\Sigma_r^{\sigma_{i-1}}$).
	For convenience, \textbf{assume for a moment that $\mathrm{Seg}_\psi$ is multiplicity free}. 
	Then every $\phi_{\sigma_{i-1}}^{\sigma_i}: \MR^{\sigma_{ i - 1  } } \to \MR^{\sigma_i}$ has a definition as above, and they compose together into a map $\MR^\sigma \to \MR^\tau$. 	
	Denote it by $\phi^{\underline \sigma}$ where $\underline \sigma$ denotes the sequence $\sigma_0, \cdots, \sigma_l$.
	This map turns out to be the desired $\phi_\sigma^\tau$ according to the following lemma; it will be proved in subsubsection \ref{subsec: well-define}.
\begin{lem}\label{tran-map}
	The transition map $\phi^{\underline \sigma} = \phi_{\sigma_{l-1}}^{\sigma_l} \circ \cdots \circ \phi_{\sigma_0}^{\sigma_1}$ is independent of the choice of sequence $\underline \sigma$. 
\end{lem}
	
	Now we deal with the case that $\Sigma_\psi$ has duplicate segments.
	When defining $\phi_\sigma^{\sigma s}$ in \eqref{transition} we excluded the case $\nu^\sigma_i = \nu^\sigma_{i+1}$. 
	The fact is that in this case, either choice of $\phi_\sigma^{\sigma s}$ is acceptable; this will also be implied in section \ref{swap-tableau}.
	Fix a choice of $\phi_{\sigma}^{\sigma s}$, define $\phi^{\underline \sigma}$ for each sequence $\underline \sigma$, and check Lemma \ref{tran-map} remains available in this setting,
	then we will obtain the definition for general $\phi_\sigma^\tau$.
	
\begin{rmk}\label{rmk: deduplicate}
	For convenience, we replace the relation ``$\mu = \nu$'' by ``$\mu \supset \nu$'' or ``$\mu \subset \nu$'' by hand.
	If a segment $\mu$ repeats $t$ times in $\Sigma_\psi$, denoted as $\mu_1, \cdots, \mu_t$, then choose an order on them;
	it is harmless to denote this order by $\supset$, say $\mu_1 \supset \cdots \supset \mu_t$.
	Thus for two segments $\mu, \nu \in \Sigma_\psi$, by $\mu \supset \nu$ we mean either $\mu$ properly contains $\nu$ as its subset, 
	or they are the same segemt, and $\mu$ is larger in the chosen order.
	
	By means of this approach, \textbf{we could always treat $\mathrm{Seg}_\psi$ as a multiplicity-free set}.
\end{rmk}
	
\begin{din}\label{def: related}
	Fix an admissible arrangement $\nu=(\nu_1, \cdots, \nu_r) \in \Sigma_\psi$. 
	For a parameter $\underline p \in \MR$, call $p_i$ the component (of $\underline p$) that \textbf{comes from} $\nu_i$.
	Two parameters $\underline p^\sigma \in \MR^\sigma$ and $\underline p^\tau \in \MR^\tau$ are said \textbf{related}, if $\underline p^\tau = \phi_\sigma^\tau (\underline p^\sigma)$.
\end{din}

\subsection{Our non-zero criterion}
	Fix an admissible arrangement $\nu \in \Sigma_\psi$. The parameter $\underline p \in \MR$ gives rise to a representation $A_{\q}(\lambda)$ if $0 \leqslant p_i \leqslant m_i$. 
	This representation $A_{\q}(\lambda)$ is either irreducible or $0$, and we need to put further constraints on $\underline p$ such that $A_{\q}(\lambda)$ is non-zero.
	
	The most obivous constraints come from pairs of adjacent segments $\nu_i, \nu_{i+1}$, that
\begin{equation}\label{adj-col}
	\min\{ p_i , q_{i+1} \} + \min\{q_i, p_{i+1} \} \geqslant \#( \nu_i \cap \nu_{i+1} ).
\end{equation}
	This is due to the following result for the simplest case;
	it will be explained in the front of section \ref{sufficiency}.
\begin{lem}\label{non-vanish-2}
	Suppose $r=2$. 
	Then $\underline p= (p_1, p_2) \in \mathcal D(\psi)$ parametrizes a non-zero $A_\q(\lambda)$ if and only if
	\[\min\{p_1, q_2\} + \min \{q_1, p_2\} \geqslant \#( \nu_1 \cap \nu_2 ).\]
\end{lem}
\begin{rmk}
	If $\nu_i > \nu_{i+1}$, then condition \eqref{adj-col} can be simplified to
	\[\frac{m_i+m_{i+1}-(a_i-a_{i+1})}2 \leqslant p_i + p_{i+1} \leqslant \frac{ m_i+m_{i+1}+a_i-a_{i+1}} 2.\]
	If $\nu_i \supset \nu_{i+1}$, then condition \eqref{adj-col} can be simplified to
	\[m_{i+1} \leqslant p_i + p_{i+1} \leqslant m_i.\]
	In the special case $r=2$, $p_1+p_2=p$ so the non-vanishing condition is merely on $p$, choice of inner form $\U(p,q)$, rather than the  parameter $\underline p \in \MR$. 
\end{rmk}

	There will be similar constraints on $\underline p^\sigma$, the parameter related to $\underline p$, 
	that come from the adjacent segments $\nu_i^\sigma, \nu_{i+1}^\sigma$ in $\nu^\sigma$.
	These constraints are new for $\underline p$, 
	since  $\nu_i^\sigma = \nu_{\sigma(i)}$ need not be adjacent to $\nu_{i+1}^\sigma = \nu_{\sigma(i+1)}$ in the original arrangement $\nu$. 
	The upshot is that, constraints arising in this manner are enough to guarantee the non-vanishing of $A_{\q_{\underline p}}(\lambda_{\underline p})$.
	
\begin{thm}\label{non-vanishing}
	The parameter $\underline p \in \MR$ gives rise to a non-zero $A_\q(\lambda)$ if and only if the following two conditions are satisfied:
\begin{itemize}
	\item $\forall i$ and $\forall \sigma \in \Sigma_r$, there is 
	\[\mathrm B^\sigma(i):
	0 \leqslant p_{\sigma^{-1}(i)}^\sigma \leqslant m_{i};\]
	
	\item $\forall i, j$ and $\forall \sigma \in \Sigma_r(i,j):= \{ \sigma \in \Sigma_r \mid  |\sigma^{-1}(i) - \sigma^{-1}(j)|= 1\}$, there is
	\[\mathrm C^\sigma(i,j):
	\min\{ p_{\sigma^{-1}(i)}^\sigma , q_{\sigma^{-1}(j)}^\sigma \} + \min\{q_{\sigma^{-1}(i)}^\sigma, p_{\sigma^{-1}(j)}^\sigma \} \geqslant \#( \nu_i \cap \nu_j ).\]
\end{itemize}
	Here, $p_i^\sigma$ is the $i$-th component of $\underline p^\sigma$ related to $\underline p$ and comes from $\nu_i^\sigma = \nu_{\sigma(i)}$. 
\end{thm}
\begin{rmk}
	Here are some remarks on this criterion.
	
	\begin{enumerate}[fullwidth, itemindent=2em, label=(\arabic*)]
	\item $\forall \sigma \in \Sigma_r(i,j)$, $\nu_i, \nu_j$ are adjacent in $\nu^\sigma$, unless $\Sigma_r(i,j) = \varnothing$.

	\item Each condition $\mathrm B^\sigma(i)$, $\mathrm C^\sigma(i, j)$ cuts out a subset of $\MR$ (after converting from $\MR^\sigma$ to $\MR$ by $\phi_\sigma^{\Id}$, of course).
	It is harmless to denote such subsets by  $\mathrm B^\sigma(i)$, $\mathrm C^\sigma(i,j)$ again. 
	Then the above theorem asserts that $\underline p\in \MR$ parametrizes a non-zero $A_\q(\lambda)$ if and only if $\underline p$ lies in the intersection of all these subsets $\mathrm B^\sigma(i)$, $\mathrm C^\sigma(i,j)$.
	
	\item This condition is obivously independent of the choice of $\nu\in \Sigma_\psi$.
	If we start from another $\nu^\sigma \in \Sigma_\psi$, then those condition on $\underline p^\sigma \in \MR^\sigma$ would be equivalent to the related $\underline p \in \MR$.
	\end{enumerate}
\end{rmk}

	The criterion of Theorem \ref{non-vanishing} can be simplified.
\begin{din}
	We say $\nu_i, \nu_j \in \Seg_\psi$ are \textbf{neighbors} if either one of the following two statements holds:
\begin{itemize}
	\item $\nu_i>\nu_j$, and there is no other $k$ such that $\nu_i > \nu_k >\nu_j$.
	\item $\nu_i \supset \nu_j$, and there is no other $k$ such that $\nu_i \supset \nu_k \supset \nu_j$.
\end{itemize}
\end{din}
\begin{rmk}
	Neighbors must be adjacent in some $\sigma \in \Sigma_r$, but the converse is not true.
	One may check easily that $\Sigma_r(i,j) = \varnothing$ if and only if $\exists k$ such that $\nu_i > \nu_k > \nu_j$.
\end{rmk}

\begin{thm}\label{non-vanish}
	The parameter $\underline p \in \MR$ gives rise to a non-zero $A_\q(\lambda)$ if and only if the following two conditions are satisfied:
\begin{itemize}
	\item $\forall i$, there is
	\[0 \leqslant p_i \leqslant m_{i};\]
	\item $\forall i, j$ such that $\nu_i, \nu_j$ are neighbors, there is
	\[\exists \sigma \in \Sigma_r(i,j), 
	\min\{ p_{\sigma^{-1}(i)}^\sigma , q_{\sigma^{-1}(j)}^\sigma \} 
	+ \min\{q_{\sigma^{-1}(i)}^\sigma, p_{\sigma^{-1}(j)}^\sigma \} 
	\geqslant \#( \nu_i \cap \nu_j ).\]
\end{itemize}
\end{thm}

\subsection{Framework of proof}
	In section \ref{sec: remained}, we verify the well-definedness and validity of $\phi_\sigma^\tau$, 
	simplify the statement of Theorem \ref{non-vanishing} to Theorem \ref{non-vanish},and 
	show the necessity of the conditions in Theorem \ref{non-vanishing}.
	To establish the sufficiency in Theorem \ref{non-vanishing}, a detailed study of Trapa’s criterion is required;
	this forms the objective of section \ref{tableau}.
	Section \ref{swap-tableau} is dedicated to proving related parameters give rise to same representations, through the combinatorial approach.
	Additionally, we will derive two straightforward results.
	One (Lemma \ref{overlap-p}) could contribute to Lemma \ref{non-vanish-2} and necessity in Theorem \ref{non-vanishing}.
	The other (Lemma \ref{operation-swap}) is useful for the recursive proof of sufficiency, which is demonstrated in section \ref{sufficiency}.
	
	Here we mention \cite{Du}, which also discusses the non-zero condition for $A_\q(\lambda)$ of $\U(p,q)$,
	and has some results consistent with ours.
	Its non-vanishing criterion (\cite[Theorem 4.1]{Du}) can be deduced by combining our lemmas \ref{overlap-geq-sing}, \ref{anti-by-type}, and \ref{overlap-p}.

%% file: Comments.tex
\section{Comparison with the $p$-adic side}\label{compare}
	As mentioned in the introduction, the non-zero criterion we propose on M{\oe}glin-Renard's parametrization for A-packets of real unitary group $\U(p,q)$
	shares the same formulation with the non-zero criterion on M{\oe}glin's parametrization for A-packets of symplectic group $\Sp(2N, F)$ and quasi-split special odd orthogonal group $\SO(2N+1, F)$ over a $p$-adic field $F$.
	We shall revisit the latter criterion following \cite[section 4]{Atobe22}, and highlight the similarities between them throughout this section.

\subsection{A-parameter}
	Let $\psi_\R$ be an A-parameter of $\U(p,q)$ with base change
	\[\psi_\C = \bigoplus_{i=1}^r (\frac z {\bar z})^{\frac {a_i} 2} \boxtimes \FDR_{m_i},\]
	where each $a_i+m_i$ has the same parity with $n= p+q$.
	Record its irreducible components by the (multi-)set $\Seg_{\psi_\R} = \{\nu_1, \cdots, \nu_r\}$ of segements $\nu_i = (\frac{a_i+m_i -1} 2, \cdots, \frac{a_i - m_i +1}2)$.
	
	Suppose $a_i +1 \geqslant m_i$, and
	define a representation of $W_F \times \SL(2, \C) \times \SL(2, \C)$ by
	\[\psi_F  =  \bigoplus_{i=1}^r 1_{W_F} \boxtimes \FDR_{a_i+1} \boxtimes \FDR_{m_i},\]
	where $1_{W_F}$ denotes the trivial representation of $W_F$.
	Due to the property $a_i+m_i \equiv n \mod 2$,
\begin{itemize}
	\item if $n$ is odd, then $\psi_F$ is a orthogonal representation with dimension $\sum_i (a_i +1)m_i$ odd,
	and hence becomes an A-parameter of $H^* = \Sp(2N, F)$ with good parity;
	\item if $n$ is even, then $\psi_F$ is a symplectic representation with dimension $\sum_i (a_i +1)m_i$ even,
	and hence becomes an A-parameter of quasi-split $H^* = \SO(2N+1, F)$ with good parity.
\end{itemize}
	For the definition of good parity we refer to \cite[subsection 2.3]{Atobe22}.
	Following \cite[Theorem 1.2]{Atobe22}, record the irreducible components of $\psi_F$ by the (multi-)set $\Seg_{\psi_F}$ of segements $[A_i, B_i] = [\frac{a_i+m_i -1} 2, \frac{a_i - m_i +1}2]$.
	Note that by definition each $\nu_i$ coincides with $[A_i, B_i]$.
	
	When applying the non-zero criterion for parametrization of $\Pi_{\psi_\R}(\U(p,q))$, we shall consider all admissible arrangements of $\Seg_{\psi_\R}$, 
	or equivalently, all admissible permutations of $\nu = (\nu_1, \cdots, \nu_r)$ if we have fixed an arrangement.
	Recall in Definition \ref{adm-arrangement} that a permutation $\sigma \in S_r$ is admissible iff $\nu_{\sigma(i)} > \nu_{\sigma(j)} \Rightarrow i<j$.
	Similarly, when applying the non-zero criterion for parametrization of $\Pi_{\psi_F}(H^*)$, one should consider all admissible orders on the index set $I= \{1, \cdots, r\}$ of $\Seg_{\psi_F}$.
	Our assumption that $a_i +1 \geqslant m_i$ implies $B_i\geqslant 0$.
	Then following \cite[Definition 3.1]{Atobe22}, a total order $\succ$ on $I$ is admissible iff $\nu_i > \nu_j \Rightarrow i \succ j$.
\begin{lem}\label{lem: compare-admissible}
	There is a bijection between the set $\Sigma_r$ of admissible permutations (at $\nu$) and the set of admissible orders on $I$:
	$\sigma \mapsto \succ_\sigma,$ where $i \succ_\sigma j$ iff $\sigma^{-1}(i) < \sigma^{-1}(j)$.
\end{lem}
	The above correspondence is clear.
	\textbf{In this section, we may assume} $\nu= (\nu_1, \cdots, \nu_r)$ is admissible, and then the natural order $1 \succ \cdots \succ r$ on $I$ is admissible.

\subsection{Arthur-Vogan packet}
	It was suggested by \cite{ABV92} and \cite{Vogan93} that for the refined local Langlands correspondence, one should consider all A-packets of real forms within an inner class simultaneously.
	
	The real forms $\U(p,q)$ with $p+q =n$ constitute to an inner class, with the quasi-split form $G^* = \U(\frac n2, \frac n2)$ or $\U(\frac{n-1}2, \frac{n+1} 2)$.
	Then we define the Arthur-Vogan packets for A-parameter $\psi_\R$ by
	\[\Pi_{\psi_\R}^\AV = \bigsqcup_{p+q = n} \Pi_{\psi_\R}(\U(p,q)),\]
	and parametrize it by
	\[\MRV: = \{ \underline p \in \Z^r \mid 0 \leqslant p_1 + \cdots + p_r \leqslant n \}
	\to \Pi_{\psi_\R}^\AV \sqcup \{0\}.\]
	It is glued from the maps $\MR \to \Pi_{\psi_\R}(\U(p,q))\sqcup \{0\}$.
	
	When $H^* = \Sp(2N, F)$, its inner class consists of itself only,
	while $H^* = \SO(2N+1, F)$ quasi-split, there is another non-split form $H^\an$ in its inner class.
	Then the Arthur-Vogan packets for A-parameter $\psi_F$ is defined as
	\[\Pi_{\psi_F}^{\AV} = \left\{\begin{aligned}
	&\Pi_{\psi_F}(H^*),	&\textrm{if } H^* = \Sp(2N, F),\\
	&\Pi_{\psi_F}(H^*) \sqcup \Pi_{\psi_F}(H^\an),	&\textrm{if } H^* = \SO(2N+1, F).
	\end{aligned}\right.\]
	Following \cite{Atobe22}, we parametrize $\Pi_{\psi_F}^\AV$ by the extended multi-segement $(\underline l, \underline \eta) \in \Z^r \times \{\pm\}^r$, 
	whose component $(l_i, \eta_i)$ is a extended segement coming from $[A_i, B_i]$, subject to the condition $l_i \leqslant \frac{m_i}2$, 
	and identified with $(l_i, -\eta_i)$ if $l_i$ attains $\frac {m_i} 2$.
\begin{rmk}
	Here are serval remarks on the parametrization in \cite{Atobe22}.
	
	\begin{enumerate}[fullwidth, itemindent=2em, label=(\arabic*)]
	\item Atobe denotes an extended segment by $\cup ([A_i, B_i]_\rho, l_i, \eta_i, \succ)$.
	Here the segment $[A_i, B_i]_\rho$ is omitted since it is self-evident,
	while the admissible order $\succ$ is hidden temporarily.
	
	\item Atobe requires $l_i \geqslant 0$, and we extend to the region $l_i <0$ by sending them to zero representation.
	
	\item Atobe considers only the quasi-split group $H^*$, where each extended multi-segement should satisfy a sign condition
	\begin{equation}\label{sign-condition}
	\prod_{i=1}^r (-1)^{[\frac {m_i} 2] + l_i} \eta_i^{m_i} = 1.
	\end{equation}
	\end{enumerate}
	The explicity parametrization for $\Pi_{\psi_F}(H^\an)$ has not been studied yet.
	We shall work under the following hypothesis.
\end{rmk}
\begin{hyp}\label{hyp: anisotropic}
	The A-packet $\Pi_{\psi_F}(H^\an)$ is parametrized by extended multi-segements $(\underline l, \underline \eta)$ subject to the sign condition
	\[\prod_{i=1}^r (-1)^{[\frac {m_i} 2] + l_i} \eta_i^{m_i} = -1.\]
	Moreover, whether $(\underline l, \underline \eta)$ parametrizes a non-zero representation is determined by the criterion in \cite[Theorem 1.4 and Theorem 4.4]{Atobe22}.
\end{hyp}

	Let's denote parameter space for $\Pi_{\psi_F}^\AV$ by $\ES_F$, which consists of extended multi-segements $(\underline l, \underline \eta)$,
	and is subject to the sign condition \eqref{sign-condition} if $n$ is odd and $H^* = \Sp(2N, F)$. 
\begin{prop}\label{prop: compare-parameter}
	Under the hypothesis \ref{hyp: anisotropic}, 
	there is a map $\MRV \to \ES_F$ that descends to $\Pi_{\psi_\R}^\AV \to \Pi_{\psi_F}^\AV$.
	Moreover, when $n$ is odd (and $H^* = \Sp(2N, F)$), it is a $2$-$1$ surjection,
	while $n$ is even (and $H^* = \SO(2N+1, F)$), it is a bijection.
\end{prop}
\begin{rmk}
	Moreover, the map between signle A-packets would be injective;
	see remark \ref{rmk: fiber-of-comparison}.
\end{rmk}
	The point is that under the map $\underline p \mapsto (\underline l, \underline \eta)$, $\underline p$ parametrizes a non-zero representation iff $(\underline l, \underline \eta)$ parametrizes a non-zero representation.

\subsection{Parameter space}\label{compare-param}
	In this subsection we define the comparison map $\MRV \to \ES_F$. 
	It will be broken up into two steps.
	The first step would provide a different parametrization $\ES_\R$ for $\Pi_{\psi_\R}^\AV$ instead of $\MRV$.
	The second step would map $\ES_\R$ to $\ES_F$.
	
	The set $\ES_\R$ consists of the extended multi-segements $(\underline l, \underline \eta) \in \Z^r \times \{\pm\}^r$, 
	whose component $(l_i, \eta_i)$ is a extended segement coming from $[A_i, B_i]$, subject to the condition $l_i \leqslant \frac{m_i}2$, 
	and identified with $(l_i, -\eta_i)$ if $l_i$ attains $\frac {m_i} 2$.
\begin{din}\label{def: reparametrization}
	For $\underline p \in \MRV$, define $(\underline l, \underline \eta)$ by
	\[l_i = \min\{p_i, q_i\}, \quad
	\eta_i =  (-1)^{m_1 + \cdots + m_i+1} \sgn(p_i - q_i). \]
	It turns out to be a bijection between $\MRV$ and $\ES_\R$.
\end{din}
\begin{rmk}
	The workable bijection $\MRV \to \ES_\R$ for Proposition \ref{prop: compare-parameter} is not unique.
	For example, we may obtain another one by compose above bijection with  the automorphism over $\ES_\R$ sending $\underline \eta = (\eta_1, \cdots, \eta_r)$ to $-\underline \eta = (- \eta_1, \cdots, -\eta_r)$.
\end{rmk}

	For $n$ even and $H^* = \SO(2N+1, F)$, $\ES_\R$ is obviously in bijection with $\ES_F$; they are actually the same set!
	For $n$ odd and $H^* = \Sp(2N, F)$, $\ES_\R$ contains $\ES_F$ as a subset.
	For convenience, define 
	\[\sgn(\underline l, \underline \eta) =
	\prod_{i=1}^r (-1)^{[\frac {m_i} 2] + l_i} \eta_i^{m_i}.\]
	Since $n = m_1 + \cdots + m_r$ is odd,  the number $\# \{i \mid m_i \textrm{ odd}\}$ is also odd.
	Then $\sgn(\underline l, -\underline \eta) = - \sgn(\underline l, \underline \eta)$.
	By definition, $\ES_F$ is the subset of $\ES_\R$ cut out by $\sgn =1$.
	We thus have the following.
\begin{din}\label{def: compare-parameter}
	The map $\ES_\R \to \ES_F$ is given as follows.
	If $n$ is odd and $H^*= \Sp(2N, F)$, it sends $(\underline l, \underline \eta)$ to $(\underline l, \underline \eta \cdot \sgn(\underline l, \underline \eta) )$,
	which is a $2$-$1$ surjection.
	If $n$ is even and $H^* = \SO(2N+1, F)$, it is the identity map.
\end{din}
	
	Now we obtain the map required in Proposition \ref{prop: compare-parameter}, given by the composition $\MRV \stackrel \sim \to \ES_\R \twoheadrightarrow \ES_F$ of maps in Definition \ref{def: reparametrization} and \ref{def: compare-parameter}.

\begin{rmk}\label{rmk: fiber-of-comparison}
	Let $n$ be odd and $H^* = \Sp(2N, F)$.
	The fiber of $\ES_\R \to \ES_F$ consists of $(\underline l, \underline \eta)$ and $(\underline l, -\underline \eta)$.
	 Under $\MRV \to \ES_\R$, they correspond to the elements $\underline p$ and $\underline q := \underline m - \underline p$.
	Each of them parametrizes a representation in $\U(p, q)$ and $\U(q, p)$ respectively, which can be identified under the isomorphism $\U(p,q) \cong \U(q, p)$ of real forms.	
\end{rmk}

\subsection{Transition map}\label{compare-tran}
	There could be admissible orders on $I$ different from the natural order $\succ$,
	and hence prarametrizes $\Pi_{\psi_F}^\AV$ in a different way.
	It's necessary to consider all these parametrizations for the non-zero criterion. 
	We have seen in Lemma \ref{lem: compare-admissible} that admissible orders on $I$ takes the form $\succ_\sigma$ with $\sigma \in \Sigma_r$.
	In this subsection we will define the parametrization space $\MRV^\sigma$, $\ES_\R^\sigma$ and $\ES_F^\sigma$ for each admissible permutation $\sigma$, and the transition map between them.
	
	Firstly, let
	\[\MRV^\sigma =  \{ \underline p^\sigma \in \Z^r \mid 
	0 \leqslant p_1^\sigma + \cdots + p_r^\sigma \leqslant n \}.\]
	It parametrizes $\Pi_{\psi_\R}^\AV$ by extending $\MR^\sigma \to \Pi_{\psi_\R}(\U(p, q)) \sqcup \{0\}$ in the obivous manner.
	Let  $\ES_\R^\sigma$ consist of the extended multi-segements $(\underline l, \underline \eta, \sigma)$,
	whose component $(l_i, \eta_i)$ is a extended segement coming from $[A_i, B_i]$, subject to the condition $l_i \leqslant \frac{m_i}2$, 
	and identified with $(l_i, -\eta_i)$ if $l_i$ attains $\frac {m_i} 2$.
	The parameter space $\ES_F^\sigma$ for $\Pi_{\psi_F}^\AV$ is defined similiar to $\ES_\R^\sigma$:
	if  $n$ is odd and $H^* = \Sp(2N, F)$, it is the subset of $\ES_\R^\sigma$ cut out by $\sgn(\underline l, \underline \eta) = 1$;
	while $n$ is even and $H^* = \SO(2N+1, F)$, it is $\ES_\R^\sigma$ it self. 
	
	Then, define the maps $\MRV^\sigma \stackrel \sim \to \ES_\R^\sigma \twoheadrightarrow \ES_F^\sigma$ as in Definition \ref{def: reparametrization} and \ref{def: compare-parameter}.
	For $\MRV^\sigma \to \ES_\R^\sigma$, we have to emphasis that according to Definition \ref{def: related}, 
	the component $p_i^\sigma$ of $\underline p^\sigma$ comes from $\nu_{\sigma(i)}$, 
	while the component $(l_i, \eta_i)$ of $(\underline l, \underline \eta, \sigma)$ comes from $\nu_i$.
	Hence, $\underline p^\sigma$ should be sent to $(\underline l, \underline \eta, \sigma)$ with
	\[l_{\sigma(i)} = \min\{p_i^\sigma, q_i^\sigma\}, \quad
	\eta_{\sigma(i)} =  (-1)^{m_{\sigma(1)} + \cdots + m_{\sigma(i)}+1} \sgn(p_i^\sigma - q_i^\sigma).\]
	
	Next, define the transition map $\phi_\sigma^\tau$.
	For $\MRV^\sigma \to \MRV^\tau$, we can proceed as in subsection \ref{subsec: transition};
	then $\phi_\sigma^\tau: \ES_\R^\sigma \to \ES_\R^\tau$ is obtained via the isomorphisms $\MRV^\sigma \cong \ES_\R^\sigma$.
\begin{lem}\label{lem: compare-transition}
	Suppose $\tau = \sigma \circ (k-1, k)$, and denote $i = \sigma(k-1)$, $j= \sigma(k)$.
	Note that by definition, $i \succ_\sigma j$, $j \succ_{\tau} i$, 
	and $[A_i, B_i], [A_j, B_j]$ must be in containment relation.
	It's harmless to assume $[A_i, B_i] \supset [A_j, B_j]$.
	Then we have the formula for $\phi_\sigma^\tau: \ES_\R^\sigma \to \ES_\R^\tau, (\underline l, \underline \eta, \sigma) \mapsto (\underline l', \underline \eta', \tau)$ as follows:
\begin{itemize}
	\item $(l_k', \eta_k') = (l_k, \eta_k)$ for $k \not= i, j$;
	\item $l_j' = l_j$, $\eta_j' = (-1)^{1+ m_i} \eta_j$;
	\item if $\eta_i = (-1)^{1+m_j}\eta_j$ and $m_i - 2l_i < 2 (m_j - 2l_j)$, then $l_i'=m_i -l_i - m_j + 2l_j$, $\eta_i' = (-1)^{1+m_j} \eta_i$;
	\item otherwise $l_i' = l_i + (-1)^{1+m_j} \eta_i \eta_j (m_j - 2l_j)$, $\eta_i' = (-1)^{m_j} \eta_i$.
\end{itemize}
	In particular, $\phi_\sigma^\tau$ sends $(\underline l, -\underline \eta, \sigma)$ to $ (\underline l', -\underline \eta', \tau)$, 
	and hence descends to $\ES_F^\sigma \to \ES_F^\tau$.
\end{lem}
\begin{rmk}
	For this explicit formula, we need the property that $\eta_{\sigma(i)} = 
	(-1)^{\varepsilon^\sigma_i} \sgn(p_i^\sigma - q_i^\sigma)$
	with $\varepsilon_i^\sigma + \varepsilon_i^{\sigma\circ (i, i+1)} = m_{\sigma(i)} + m_{\sigma(i+1)}$;
	here $\varepsilon_i^\sigma(\underline m) = \varepsilon_i(m_{\sigma(1)}, \cdots, m_{\sigma(r)})$ is regarded as a $(\Z/2)$-valued function over $\Z^r$.
	One may check the formula here coincides with that given before \cite[Theorem 4.3]{Atobe22}.
\end{rmk}

	Finally, we summarize this subsection by the following commutative diagram:
	\[\begin{tikzcd}
	\MRV^\sigma \arrow[r] \arrow[d, "\phi_\sigma^\tau"] 
	&\ES_\R^\sigma \arrow[r] \arrow[d, "\phi_\sigma^\tau"] 
	&\ES^\sigma_F \arrow[d, "\phi_\sigma^\tau"] \\
	\MRV^\tau \arrow[r]                                 
	&\ES^\tau_\R \arrow[r]                                 
	& \ES_F^\sigma
	\end{tikzcd}.\]

\subsection{Basic necessary condition}\label{compare-necessary}
	Suppose $(\underline l, \underline \eta, \sigma) \in \ES_F^\sigma$ parametrizes a non-zero representation. 
	For $i, j \in I$ adjacent in the admissible order $\succ_\sigma$, 
	there are necessary conditions on the pairs $(l_{i}, \eta_{i})$ and $(l_{j}, \eta_{j})$.
	These conditions are analogous to the condition \eqref{adj-col} on real side.
	We will compare them in this subsection.
	
	For convenience, let $i = \sigma(h-1)$ and $j = \sigma(h)$.
	Following Theorem \ref{non-vanishing}, denote
	\[\mathrm C_\R^\sigma(i,j):
	\min\{ p_{h-1}^\sigma , q_{h}^\sigma \} + \min\{q_{h-1}^\sigma, p_{h}^\sigma \} \geqslant \#( \nu_i \cap \nu_j ),\]
	and regard it as a subset of $\MRV^\sigma$.
\begin{lem}\label{lem: compare-necessary}
	Under the bijection $\MRV^\sigma \to \ES_\R^\sigma$, the image of $\mathrm C_\R^\sigma(i,j )$ consists of extended multi-segements $(\underline l, \underline \eta, \sigma)$ meeting the following conditions:
	\begin{enumerate}[label=$(\arabic*)$]
	\item in the case $[A_{i}, B_{i}] \geqslant [A_{j}, B_{j}]$, if $\eta_{i} = (-1)^{1+m_{j}} \eta_{j}$, then
	\[\frac{a_{j} - a_{i} - m_{j} + m_{i}} 2 \leqslant 
	l_{i} - l_{j} \leqslant 
	\frac{a_{i} - a_{j} + m_{i} - m_{j}} 2,\]
	while $\eta_{i} = (-1)^{m_{j}} \eta_{j}$, then
	\[l_{i} + l_{j} \geqslant \frac{ a_{j} - a_{i} + m_{j} + m_{i}} 2;\]
	
	\item in the case $[A_{i}, B_{i}] \supset [A_{j}, B_{j}]$, if $\eta_{i} = (-1)^{1+m_{j}} \eta_{j}$, then
	\[0 \leqslant l_{i} - l_{j} \leqslant m_{i} - m_{j},\]
	while $\eta_{i} = (-1)^{m_{j}} \eta_{j}$, then
	\[l_{i} + l_{j} \geqslant m_{j};\]
	
	\item in the case $[A_{i}, B_{i}] \subset [A_{j}, B_{j}]$, if $\eta_{i} = (-1)^{1+m_{j}} \eta_{j}$, then
	\[0 \leqslant l_{i} - l_{j} \leqslant m_{j} - m_{i},\]
	while $\eta_{i} = (-1)^{m_{j}} \eta_{j}$, then
	\[l_{i} + l_{j} \geqslant m_{i}.\]
	\end{enumerate}
	In particular, $(\underline l, \underline \eta, \sigma)$ meets the above conditions iff $(\underline l, -\underline \eta, \sigma)$ meets them.
	We denote this image also by $\mathrm C_\R^\sigma(i, j)$, and its image under $\ES_\R^\sigma \to \ES_F^\sigma$ by $\mathrm C_F^\sigma(i,j )$.
\end{lem}
\begin{rmk}
	For this lemma we need the property that $\eta_{\sigma(h-1)} = (-1)^{\varepsilon_{h-1}^\sigma} \sgn(p_{h-1}^\sigma - q_{h-1}^\sigma)$,
	with $\varepsilon_{h-1}^\sigma + \varepsilon_{h}^\sigma = m_{\sigma(h)}$;
	here $\varepsilon_h^\sigma(\underline m) = \varepsilon_h (m_{\sigma(1)}, \cdots, m_{\sigma(r)})$ is regarded as a $(\Z/2)$-valued function over $\Z^r$..
	One may check $\mathrm C_F^\sigma(i, j)$ coincides with the necessary condition formulated in \cite[Proposition 4.1]{Atobe22}; 
	its $(k, k-1)$ should be replaced by $(i, j)$ for the comparison.
\end{rmk}

\subsection{Sufficient condition}
	In this subsection we verify Proposition \ref{prop: compare-parameter}.
	Suppose $\underline p$ is sent to $(\underline l, \underline \eta)$ under the map $\MRV \to \ES_F$ defined in subsection \ref{compare-param}.
	It suffices to check $\underline p$ parametrizes a non-zero representation iff $(\underline l, \underline \eta)$ parametrizes a non-zero representation.

	According to \cite[Definition 3.1 and Theorem 4.4]{Atobe22},
	a extended multi-segement $(\underline l, \underline \eta) \in \ES_F$ parametrizes a non-zero representation iff it lies in the following subsets: $\mathrm C_F^\sigma(i, j)$, where $i, j$ are adjacent in $\succ_\sigma$, and
	\[\mathrm B_F^\sigma(i) = \{(\underline l, \underline \eta, \sigma) \in \ES_F^\sigma \mid l_i \geqslant 0\}.\]
	Note that $\mathrm B_F^\sigma(i)$ and $\mathrm C_F^\sigma(i, j)$ are subsets of $\ES_F^\sigma$ by definition, 
	and we can regrad them as subsets of $\ES_F$ via the transition map $\phi_\sigma^\Id: \ES_F^\sigma \to \ES_F$.

	This description is the same as our Theorem \ref{non-vanishing}:
	$\underline p \in \MRV$ parametrizes a non-zero representation iff 
	it lies in all the subsets $\mathrm C_\R^\sigma(i,j)$, where $i, j$ are adjacent in $\succ_\sigma$, and
	\[\mathrm B_\R^\sigma(i)=\{ \underline p \in \MRV^\sigma \mid 
	0 \leqslant p_{\sigma^{-1}(i)}^\sigma \leqslant m_{i}\}.\] 
	Here we also regard $\mathrm B_\R^\sigma(i)$ and $\mathrm C_\R^\sigma(i, j)$ as subsets of $\MRV$ via the transition map $\phi_\sigma^\Id: \MRV^\sigma \to \MRV$.
	
	Due to the consistency in Lemma \ref{lem: compare-transition} and \ref{lem: compare-necessary}, we know $\underline p$ lies in $\mathrm C^\sigma_\R(i,j)$ iff $(\underline l, \underline \eta)$ lies in $\mathrm C_F^\sigma(i, j)$.
	It is also easy to check $\mathrm B_\R^\sigma(i)$ is the preimage of $\mathrm B_F^\sigma(i)$ under $\MRV^\sigma \to \ES_F^\sigma$.
	Then due to Lemma \ref{lem: compare-transition} again we know $\underline p$ lies in $\mathrm B^\sigma_\R(i,j)$ iff $(\underline l, \underline \eta)$ lies in $\mathrm B_F^\sigma(i, j)$.
	Consequently, we obtain the equivalence of the following $4$ statements:
	\begin{itemize}
	\item $\underline p$ parametrizes a non-zero representation;
	\item $\underline p$ lies in all $\mathrm B_\R^\sigma(i)$ and $\mathrm C_\R^\sigma(i, j)$;
	\item $(\underline l, \underline \eta)$ lies in all $\mathrm B_F^\sigma(i)$ and $\mathrm C_F^\sigma(i, j)$;
	\item $(\underline l, \underline \eta)$ parametrizes a non-zero representation.
	\end{itemize}
	Hence Proposition \ref{prop: compare-parameter} follows.

\begin{rmk}
	We may simplify the non-zero criterion for $(\underline l, \underline \eta)\in \ES_F$ as Theorem \ref{non-vanish}.
	That is, $(\underline l, \underline \eta)$ satisfies the following two conditions:
	\begin{itemize}
	\item $\forall i$, there is $\mathrm B_F(i): l_i \geqslant 0$;
	
	\item $\forall i, j$ such that $\nu_i, \nu_j$ are neighbors, there is
	$\mathrm C_F(i, j): \exists \sigma \in \Sigma_r(i, j), (\underline l, \underline \eta)$ meets $\mathrm C_F^\sigma(i, j)$.
	\end{itemize}
	
\end{rmk}

%% file: Proof-3.tex
\section{Proof of necessity}\label{sec: remained}
	In this section we deal with serval unproved claims in Section \ref{result}.

\subsection{Transition map}\label{transition-map}
	Here we will verify the validity of $\phi_\sigma^{\sigma s}$, that 
	related $\underline p^\sigma \in \MR^\sigma$ and $\underline p^\sigma \in \MR^{\sigma s}$ parametrize the same representation, and
	 prove Lemma \ref{tran-map}, that
	$\phi^{\underline \sigma} = \phi_{\sigma_{l-1}}^{\sigma_l} \circ \cdots \circ \phi_{\sigma_0}^{\sigma_1}:
	\MR^\sigma \to \MR^\tau$ is independent of choice of sequence $\underline \sigma: \sigma= \sigma_0 , \cdots, \sigma_l = \tau$.
	Note that both statements are independent of the choice of base point $\nu$, 
	so we can choose an appropriate one for our convenience.

\subsubsection{Identifying representations}\label{easy-approach}
	By replacing $\nu$ with $\nu^\sigma$, we may assume $\sigma = \Id$.
	Suppose the transposition $s= (k, k+1) \in \Sigma_r$, and $\nu_k \subset \nu_{k+1}$ in the admissible arrangement $\nu=( \nu_1, \cdots, \nu_r)$.
	Let $\underline p \in \MR$ and moreover lie in
	\[\mathrm B^\Id (i): p_i \leqslant m_i, \forall i =1, \cdots, r;\quad 
	\mathrm C^\Id (k, k+1): m_k \leqslant p_k + p_{k+1} \leqslant m_{k+1}.\]
	One may check easily $\underline p^s: = \phi^s(\underline p)$ lie in all $\mathrm B^s(i)$ and $\mathrm C^s(k, k+1)$;
	this is essentially Lemma \ref{B-independent}.
	Then there are $A_{\q_{\underline p}}(\lambda_{\underline p})$ and $A_{\q_{\underline p^s}}(\lambda_{\underline p^s})$ defined as in subsection \ref{subsec: parametrize-A-packet},
	and we have claimed they are the same representation.
	
	We will give a combinatorial proof in section \ref{swap-tableau}.
	Here is a convenient representation-theoretic proof. 
	It needs two ingredients.
	One is the conclusion in case $r=2$.
	It is stated in \cite[lemma 9.3]{Trapa}, and has an easy combinatorial proof.
	Trapa also presents an illustrative example before that lemma.
	
	The other one is transitivity of cohomological induction.
	Our notations follows \cite[(5.6)]{KV}, that
	$A_\q(\lambda) = \mathcal L_{S}(\C_\lambda)$, 
	where following \cite[(5.1)-(5.5)]{KV}, $S= \dim (\u \cap \k)$,
	\[\mathcal L_j(Z) = \left( P_{\bar \q, L\cap K}^{\g, K} \right)_j 
	\left( \mathcal F_{\l, L\cap K}^{\bar \q, L \cap K} (Z^{\#}) \right),  \]
	and 
	\[Z^{\#} = Z \otimes_\C \bigwedge^{\mathrm {top}} \u.\]
	According to \cite[Theorem 5.35]{KV}, $S= \dim \u \cap \k$ is the highest possible degree that $ \left( P_{\bar \q, L\cap K}^{\g, K} \right)_j $ does not vanish.
	We shall denote $\mathcal L_S$ by $\mathcal L_\q^\g$. 
	
\begin{prop}\label{prop: inductioin-in-stage}
	Suppose $\q \subseteq \q' \subseteq \g$ are $\theta$-stable parabolic subalgebra, with Levi decomposition $\q= \l \oplus \u$, $\q'= \m \oplus \u'$,
	then the cohomological induction via $(\l, L\cap K) \hookrightarrow (\g, K)$ breaks up into two-stage inductions via
	$(\l, L \cap K) \hookrightarrow (\m, M\cap K) \hookrightarrow (\g, K)$.
	More precisely, there are natrual isomorphism of functors
	$\mathcal L_{\q'}^\g \circ \mathcal L_\p^\m \simeq
	\mathcal L_\q^\g,$
	where $\p:= \m \cap \q \subseteq \m$ is also a $\theta$-stable parabolic subalgebra.
\end{prop}
\begin{proof}
	The Levi desomposition of $\p$ is $\l \oplus \n$, with $\n = \m \cap \u$.
	We have to show
	\[\begin{aligned}
	&\left( P_{\bar \q', M\cap K}^{\g, K} \right)_{S'}
	\left( \mathcal F_{\m, M\cap K}^{\bar \q', M \cap K} \left(
		\left( P_{\bar \p, L\cap K}^{\m, M\cap K} \right)_{T}
		\left( \mathcal F_{\l, L\cap K}^{\bar \p, L \cap K} (
			Z\otimes_\C \bigwedge^{\mathrm{top}} \n
		) \right)
		\otimes_\C \bigwedge^{\mathrm {top}} \u'
	\right) \right)\\
	\stackrel?\simeq&
	\left( P_{\bar \q, L\cap K}^{\g, K} \right)_S
	\left( \mathcal F_{\l, L\cap K}^{\bar \q, L \cap K} (
		Z \otimes_\C \bigwedge^{\mathrm{top}} \u
	) \right),
	\end{aligned}\]
	where $S'= \dim(\u' \cap \k),$ $T = \dim( \n \cap \k)$, and $S= \dim(\u \cap \k)$.
	It's clear that $\u = \n \oplus \u'$, $\u \cap \k = ( \n \cap \k) \oplus (\u' \cap \k)$, hence $S= S' + T$,
	and the degree of two sides match.
	Since tensoring by character is exact and preserves projective objects,
	\[\begin{aligned}
	&\left( P_{\bar \p, L\cap K}^{\m, M\cap K} \right)_{T}
	\left( \mathcal F_{\l, L\cap K}^{\bar \p, L \cap K} (
		Z\otimes_\C \bigwedge^{\mathrm{top}} \n
	) \right)
	\otimes_\C \bigwedge^{\mathrm {top}} \u'	
	\cong	
	\left( P_{\bar \p, L\cap K}^{\m, M\cap K} \right)_{T}
	\left( \mathcal F_{\l, L\cap K}^{\bar \p, L \cap K} (
		Z\otimes_\C \bigwedge^{\mathrm{top}} \n
	) \otimes_\C \bigwedge^{\mathrm {top}} \u' \right)	\\
	\cong	&
	\left( P_{\bar \p, L\cap K}^{\m, M\cap K} \right)_{T}
	\left( \mathcal F_{\l, L\cap K}^{\bar \p, L \cap K} (
		Z\otimes_\C \bigwedge^{\mathrm{top}} \n
		\otimes_\C \bigwedge^{\mathrm {top}} \u'
	) \right)	
	\cong	
	\left( P_{\bar \p, L\cap K}^{\m, M\cap K} \right)_{T}
	\left( \mathcal F_{\l, L\cap K}^{\bar \p, L \cap K} (
		Z\otimes_\C \bigwedge^{\mathrm{top}} \u
	) \right),
	\end{aligned}\]
	so it suffices to prove
	\[\left( P_{\bar \q', M\cap K}^{\g, K} \right)_{S'} \circ
	\mathcal F_{\m, M\cap K}^{\bar \q', M \cap K} \circ
	\left( P_{\bar \p, L\cap K}^{\m, M\cap K} \right)_{T} \circ
	\mathcal F_{\l, L\cap K}^{\bar \p, L \cap K}
	\stackrel?\simeq
	\left( P_{\bar \q, L\cap K}^{\g, K} \right)_S \circ
	\mathcal F_{\l, L\cap K}^{\bar \q, L \cap K}.\]
	Following \cite[(11.71d)]{KV}, it is also wrote as
	\[\left(^{\mathrm u} \mathcal L_{\bar \q', M \cap K}^{\g, K} \right)_{S'}
	\circ
	\left(^{\mathrm u} \mathcal L_{\bar \p, L \cap K}^{\m, M \cap K} \right)_{T}
	\stackrel?\simeq
	\left(^{\mathrm u} \mathcal L_{\bar \q, L \cap K}^{\m, M \cap K} \right)_{S}.
	\]
	Then it follows from \cite[Theorem 11.77]{KV} and the fact that both sides attain their highest degree.
\end{proof}
	
	For $G= \U(p,q)$ and parameter $\underline p = (p_1, \cdots, p_r)$ as above, let $\underline p'= (p_1, \cdots, p_k + p_{k+1}, \cdots, p_r)$, then
	there are $\theta$-stable parabolic subgroups $\q_{\underline p} \subseteq \q_{\underline p'} \subseteq \g$ defined as in subsection \ref{subsec: parametrize-A-packet}.
	Due to the definition and Proposition \ref{prop: inductioin-in-stage},
	\[A_{\q_{\underline p}}(\lambda_{\underline p}) = 
	\mathcal L_{\q_{\underline p}}^\g (\C_{\lambda_{\underline p}}) \cong
	\mathcal L_{\q_{\underline p'}}^\g \left( 
		\mathcal L_{\q_{\underline p}}^{\q_{\underline p'}}
		(\C_{\lambda_{\underline p}})
	\right).\]
	In the meantime, the $\theta$-stable parabolic $\q_{\underline p^{s}}$ defined from $\underline p^{s}$ is also contained in $\q_{\underline p'}$, so
	\[A_{\q_{\underline p^{s}}}(\lambda_{\underline p^{s}}) = 
	\mathcal L_{\q_{\underline p^{s}}}^\g 
	(\C_{\lambda_{\underline p^{s}}}) \cong
	\mathcal L_{\q_{\underline p'}}^\g \left( 
		\mathcal L_{\q_{\underline p^{s}}}^{\q_{\underline p'}}
		(\C_{\lambda_{\underline p^{s}}})
	\right).\]
	We would obtain $A_{\q_{\underline p}}(\lambda_{\underline p})
	\stackrel? \cong
	A_{\q_{\underline p^{s}}}(\lambda_{\underline p^{s}})$ from
	\[\mathcal L_{\q_{\underline p}}^{\q_{\underline p'}}
	(\C_{\lambda_{\underline p}}) \stackrel? \cong
	\mathcal L_{\q_{\underline p^{s}}}^{\q_{\underline p'}}
	(\C_{\lambda_{\underline p^{s}}}),\]
	which is implied by the following compability of cohomological induction and external tensor product:
\begin{lem}\label{lem: induction-by-component}
	Suppose $(\g, K) = (\g_1, K_1) \times (\g_2, K_2)$, and $\q= \g_1 \oplus \q_2$ with $\q_2= \l_2 \oplus \u_2 \subseteq \g_2$ a $\theta$-stable parabolic subalgbra.
	Then for any $(\g_1, K_1)$-module $V_1$ and $(\l_2, L_2\cap K_2)$-module $V_2$, there is a natrual isomorphism
	$\mathcal L_\q^\g( V_1 \boxtimes V_2) \cong V_1 \boxtimes \mathcal L_{\q_2}^{\g_2} (V_2).$
\end{lem}
\begin{proof}
	It suffices to prove the natrual isomorphism
	\[\left( P_{\q, L \cap K}^{\g, K} \right)_j \left(
		\mathcal F_{\l, L \cap K}^{\q, L\cap K}(
			( V_1 \boxtimes V_2 )\otimes_\C
			 \bigwedge^{\mathrm{top}} \u_2)	
	\right)	\stackrel?\simeq
	V_1 \boxtimes \left(P_{\q_2, L_2 \cap K_2}^{\g_2, K_2} \right)_j \left( 
		\mathcal F_{\l_2, L_2 \cap K_2}^{\q_2, L_2\cap K_2} (V_2
		\otimes_\C \bigwedge^{\mathrm{top}} \u_2)
	\right)\]
	for each $j$.
	Since
	\[\mathcal F_{\l, L \cap K}^{\q, L\cap K}(( V_1 \boxtimes V_2 ) 
		\otimes_\C \bigwedge^{\mathrm{top}} \u_2)	\cong
	V_1 \boxtimes \mathcal F_{\l_2, L_2 \cap K_2}^{\q_2, L_2\cap K_2}
	(V_2 \otimes_\C \bigwedge^{\mathrm{top}} \u_2),\]
	it suffices that
	\[\left( P_{\q, L \cap K}^{\g, K} \right)_j ( V_1 \boxtimes - )
	\stackrel?\simeq
	V_1 \boxtimes \left(P_{\q_2, L_2 \cap K_2}^{\g_2, K_2} \right)_j(-).\]
	The right hand side is derived functor for $V_1\boxtimes P_{\q_2, L_2 \cap K_2}^{\g_2, K_2}(-)$, since $V_1 \boxtimes -$ is exact.
	Then we have to verify that the left hand side calculates the derived functor for $P_{\q, L \cap K}^{\g, K}(V_1 \boxtimes -) \cong 
	V_1\boxtimes P_{\q_2, L_2 \cap K_2}^{\g_2, K_2}(-)$.
	
	Due to the Grothendieck spectral sequence, it follows once $V_1 \boxtimes -$ is proved to 
	send projective $(\q_2, L_2 \cap K_2)$-modules to $P_{\q, L\cap K}^{\g, K}$-acyclic modules, that is,
	$\left( P_{\q, L \cap K}^{\g, K} \right)_j ( V_1 \boxtimes W) 
	\stackrel?= 0$
	for each $j>0$ and projective $(\q_2, L_2 \cap K_2)$-module $W$.
	According to the properties of $P$ listed in \cite[section II.4]{KV},
	\[\left( P_{\q, L \cap K}^{\g, K} \right)_j \simeq 
	\left( \Pi_{\g, L\cap K}^{\g, K} \right)_j\circ 
	\ind_{\q, L\cap K}^{\g, L\cap K},\]
	since $P = \Pi \circ \ind$, with $\ind$ exact and preserving projectives.
	Due to \cite[Proposition 2.115]{KV},
	\[\left(\Pi_{\k, L\cap K}^{\k ,K} \right)_j \circ 
	\mathcal F_{\g, L\cap K}^{\k, L\cap K}
	\simeq
	\mathcal F_{\g, K}^{\k, K} \circ
	\left( \Pi_{\g, L\cap K}^{\g, K} \right)_j,\]
	where the forgetful functor $\mathcal F_{\g,K}^{\k, K}$ is faithfully exact.
	Then $\left( \Pi_{\g, L\cap K}^{\g, K} \right)_j\left(
	\ind_{\q, L\cap K}^{\g, L\cap K} (V_1 \boxtimes W)\right) \stackrel?=0$ would follow from
	\[\left(\Pi_{\k, L\cap K}^{\k ,K} \right)_j \left(
		\mathcal F_{\g, L\cap K}^{\k, L\cap K}\left(
			\ind_{\q, L\cap K}^{\g, L\cap K} (V_1 \boxtimes W)
		\right)
	\right) \stackrel?= 0.\]
	As a module of $(\k, L\cap K) = (\k_1, K_1)\times (\k_2, L_2 \cap K_2)$,
	\[\begin{aligned}
	\mathcal F_{\g, L\cap K}^{\k, L\cap K}\left(
		\ind_{\q, L\cap K}^{\g, L\cap K} (V_1 \boxtimes W)
	\right)	=	&
	\ind_{\q, L\cap K}^{\g, L\cap K} (V_1 \boxtimes W)	=
	U(\g) \otimes_{U(\q)}( V_1 \boxtimes W)	\\
	=	&
	\left( U(\g_1) \otimes U(\g_2) \right) \otimes_{U(\g_1) \otimes U(\q_2)}
		(V_1 \boxtimes W)	\\
	\cong	&
	V_1 \boxtimes \left( U(\g_2) \otimes_{U(\q_2)} W \right)	\cong
	V_1 \boxtimes \left( U(\u_2^-) \otimes_\C W \right).
	\end{aligned}\]
	Note that $V_1$ is a projective module of $(\k_1, K_1)$ (\cite[Lemma 2.4]{KV}), and 
	$U(\u_2^-) \otimes_\C W$ is a projective module of $(\k_2, L_2\cap K_2)$, since it is projective over $(\q_2, L_2 \cap K_2)$ and 
	$\mathcal F_{\g_2, L_2\cap K_2}^{\k_2, L_2 \cap K_2}$ preserves projectives(\cite[Lemma 2.57]{KV}).
	Then $V_1 \boxtimes \left( U(\u_2^-) \otimes_\C W \right)$ is a projective $(\k, L\cap K)$-module, 
	hence acyclic for the right exact fuctor $\Pi_{\k, L\cap K}^{\k, K}$.
\end{proof}

\subsubsection{General definition}\label{subsec: well-define}
	Let $\underline \sigma: \sigma= \sigma_0 , \cdots, \sigma_l = \tau$ be a sequence in $\Sigma_r$ which differs by transpositions,
	there is $\phi^{\underline \sigma} = \phi_{\sigma_{l-1}}^{\sigma_l} \circ \cdots \circ \phi_{\sigma_0}^{\sigma_1}:
	\MR^\sigma \to \MR^\tau$ 
	to be proved independent of choice of sequence $\underline \sigma$.
	
	\paragraph{Complete case.}
	Suppose now $\nu_1\supset \cdots \supset \nu_r$, then $\Sigma_r = S_r$. 
	It is a Coxeter group, and has a length fuction $l$ given by counting the inversion number.
	Let $\phi_\sigma^{\sigma s}: \MR^\sigma \to \MR^{\sigma s}$ be
    	\[\phi_\sigma^{\sigma s}(\underline p^\sigma) =
    	\left\{ \begin{aligned}
	&   (p_1^\sigma, \cdots,q_{i+1}^\sigma, p_i^\sigma + p_{i+1}^\sigma - q_{i+1}^\sigma, \cdots, p_r^\sigma)
    	&   \textrm{if }  l(\sigma )< l(\sigma s),\\
    	&   (p_1^\sigma, \cdots, p_i^\sigma + p_{i+1}^\sigma - q_i^\sigma, q_i^\sigma , \cdots, p_r^\sigma)
    	&   \textrm{if } l(\sigma)> l(\sigma s),
    	\end{aligned}\right.\]
    	which is compatible with \eqref{transition}.
    	
	Take a big space
    \[\widetilde{\MR}= \prod_{\sigma \in S_n} \MR^\sigma,\]
and denote its element by $P=(\underline p^\sigma)$.
	A transposition $s$ acts on it by $\Phi_s: P \mapsto 
P^s= ( \phi_\sigma^{\sigma s}(\underline p^{\sigma})) $. 
	More precisely, the $\sigma$-factor $\underline p^\sigma$ of $P$ are send to the $\sigma s$-factor $\phi_\sigma^{\sigma s}(\underline p^{\sigma})$ of $P^s$. 
	This action can be extended to the whole group $S_r$ once we prove $\Phi_s$ satisfies the generator relations:
	\begin{enumerate}[label=(\alph*)]
	\item $\Phi_s^2 = \Id$, which means
	$\forall \sigma \in S_r$, $\phi^\sigma_{\sigma s} \phi^{\sigma s}_\sigma = \Id$;
	
	\item if $s=(i, i+1)$, $t=(j, j+1)$ with $|i-j|>1$, then $\Phi_t \Phi_s = \Phi_s \Phi_t$, which means
	$\phi_{\sigma s}^{\sigma st} \phi_\sigma^{\sigma s} = 
	\phi_{\sigma t}^{\sigma ts} \phi_\sigma^{\sigma t}$;
	
	\item if $s=(i, i+1)$, $t=(i+1, i+2)$, then $\Phi_t \Phi_s \Phi_t = \Phi_s \Phi_t \Phi_s$, which means
	$\phi_{\sigma ts}^{\sigma tst} \phi_{\sigma t}^{\sigma ts} \phi_\sigma^{\sigma t} = 
	\phi_{\sigma st}^{\sigma sts} \phi_{\sigma s}^{\sigma st} \phi_\sigma^{\sigma s}$.
	\end{enumerate}

	The relations (a),(b) are easy. 
	The relations (c) concerns only about the $i, i+1, (i+2)$-th components,
	 so it suffices to prove in the case $r=3$.
	If $\sigma = \Id$, we can done by direct computation as the following diagram:
	\[\begin{tikzpicture}
	\begin{scope}
	\foreach \i in {1, 2, 3}{
	        	\pgfmathsetmacro{\theta}{\i*60-15};
	        	\pgfmathsetmacro{\phi}{\i*60-45};
	        	\pgfmathsetmacro{\sigma}{\theta-90};
	        	\pgfmathsetmacro{\tau}{\phi+90};
		\draw[-stealth] (2*cos \theta, 2*sin \theta) -- (2*cos \phi , 2* sin \phi );
		\draw[-latex] (2.5*cos \theta, 2.5*sin \theta) to[out=\sigma, in=\tau] (2.5*cos \phi , 2.5* sin \phi );
		\draw (2.4*cos \theta, 2.4*sin \theta)--(2.6*cos \theta, 2.6*sin \theta);
	}
	\foreach \i in {3, 4, 5}{
	        	\pgfmathsetmacro{\theta}{\i*60+15};
	        	\pgfmathsetmacro{\phi}{\i*60+45};
	        	\pgfmathsetmacro{\sigma}{\theta+90};
	        	\pgfmathsetmacro{\tau}{\phi-90};
		\draw[-stealth] (2* cos \theta, 2* sin \theta) -- (2* cos \phi , 2* sin \phi );
		\draw[-stealth] (2.5*cos \theta, 2.5*sin \theta) to[out= \sigma, in=\tau]  (2.5*cos \phi , 2.5* sin \phi );
		\draw (2.4*cos \theta, 2.4*sin \theta)--(2.6*cos \theta, 2.6*sin \theta);
	}
	\node at (2*cos 180, 2*sin 180) {$\MR$};
	\node at (3.5*cos 180, 2.5*sin 180) {$(p_1, p_2, p_3)$};
	\node at (2*cos 120, 2*sin 120) {$\MR^s$};
	\node at (3.5*cos 120-0.5, 2.5*sin 120) {$(q_2, p_1+p_2-q_2, p_3)$};
	\node at (2*cos 60, 2*sin 60) {$\MR^{st}$};
	\node at (3.5*cos 60+1.5, 2.5*sin 60) {$(q_2, q_3, p_1+p_2+p_3-q_2-q_3)$};
	\node at (2*cos 0, 2*sin 0) {$\MR^{sts=tst}$};
	\node at (3.5*cos 0+2.5, 2.5*sin 0) {$(p_3, q_2+q_3-p_3, p_1+p_2+p_3-q_2-q_3)$};
	\node at (2*cos 240, 2*sin 240) {$\MR^t$};
	\node at (3.5*cos 240-0.5, 2.5*sin 240) {$(p_1, q_3, p_2+p_3-q_3)$};
	\node at (2*cos 300, 2*sin 300) {$\MR^{ts}$};
	\node at (3.5*cos 300+1.5, 2.5*sin 300) {$(p_3, p_1+q_3-p_3, p_2+p_3-q_3)$};
	\end{scope}.
	\end{tikzpicture}\]
	For other $\sigma \in S_3$, the conclusion would also follow from the above commutative diagram, together with the relation (a).
	Now we obtain a well defined right action $\Phi$ of $S_r$ on $\widetilde\MR$.
	
	In this case, $\phi^{\underline \sigma} =  \phi_{\sigma_{l-1}}^{\sigma_l} \circ \cdots \circ \phi_{\sigma_0}^{\sigma_1}: \MR^\sigma \to \MR^\tau$ coincides with the component of $\Phi_{\sigma_{l-1}^{-1}  \sigma_l} \circ \cdots \circ \Phi_{\sigma_{0}^{-1} \sigma_1} = \Phi_{\sigma_0^{-1} \sigma_l} = \Phi_{\sigma^{-1} \tau}.$ 
	Then we get the independence between $\phi^{\underline \sigma}$ and  $\underline \sigma$, and hence a well-defined $\phi_\sigma^{\tau}$ immediately.
	
\paragraph{Partial case.}
	For the general case, we may take the ``appropriate'' arrangement $\nu = (\nu_1, \cdots, \nu_r) \in \Sigma_\psi$ such that $\forall i<j,$ either $\nu_i> \nu_j$ or $\nu_i \supset \nu_j$.
	
\begin{lem}
	$\phi^{\underline \sigma}$ is affine linear in $\underline p^\sigma$, with  coefficients affine linear in $\underline m:=(m_1, \cdots, m_r)$. 
	More explicitly, if regard $\underline p^\sigma \in \MR^\sigma$ as column vectors, 
	then there exist a  matrix $A^{\underline \sigma}$ and a vector $v^{\underline \sigma}$ such that 
	\[\phi^{\underline \sigma}(\underline p^\sigma) 
	= A^{\underline \sigma} \underline p^\sigma + v^{\underline \sigma},\]
	and moreover, $A^{\underline \sigma}$ is independent of $\underline m$, and entries of $v^{\underline \sigma}$ are linear functions of $\underline m$. 
\end{lem}
\begin{proof}
	It follows from the definition of $\phi_\sigma^{\sigma s}$, and a trivial induction on the lenth $l$ of $\underline \sigma$.
\end{proof}
	
	We have to prove the matrix $A^{\underline \sigma}$ and vector $v^{\underline \sigma}(\underline m)$ coincide with $A^{\underline \sigma'}$ and $v^{\underline \sigma'}(\underline m)$ for 
	different choices of sequences $\underline \sigma$ and $\underline \sigma'$ from $\sigma$ to $\tau$.
	This is true if we enlarge $m_1, \cdots, m_{r-1}$ in turn, to achieve the case $\nu_1 \supset \nu_2 \supset \cdots \supset \nu_r$ considered previously.
	In this enlarging process the $A^{\underline \sigma}$, $v^{\underline \sigma}$ (and $A^{\underline \sigma'}$, $v^{\underline \sigma'}$) are preserved  (if regarded as fuctions with variable $\underline m$), 
	since every factor in $\phi^{\underline \sigma} = \phi_{\sigma_{r-1}}^{\sigma_r} \circ \cdots \circ \phi_{\sigma_0}^{\sigma_1}$ are preserved.
	After enlarging, we see $A^{\underline \sigma}= A^{\underline \sigma'}$ and $v^{\underline \sigma}(\underline m) = v^{\underline \sigma'}( \underline m)$ easily since 
	they both come from the coefficients of $\phi_\sigma^\tau$ (it has been defined in the enlarged case). 
	The region $m_1 \gg m_2 \gg \cdots \gg m_r$ in $\Z^r$ is big enough to contain a geometrically independent set, so the consistency of $v^{\underline \sigma}(\underline m)$ and $v^{\underline \sigma'}(\underline m)$ over it implies they are the same linear functions in $\underline m$, and then coincide everywhere.
	In particular,  $A^{\underline \sigma}, v^{\underline \sigma}(\underline m)$ coincide with $ A^{\underline \sigma'}, v^{\underline \sigma'}( \underline m)$ before enlarging $\underline m$. 
	Now we obtain Lemma \ref{tran-map}, and the well-defined $\phi_\sigma^\tau := \phi^{\underline \sigma}$.

\subsection{Simplification of the criterion}\label{simplification-main-thm}
	Let's deduce Theorem \ref{non-vanish} from Theorem \ref{non-vanishing} in this subsection. We will regard the conditions
	\[\begin{aligned}
	\mathrm B^\sigma(i):	&
	0 \leqslant p_{\sigma^{-1}(i)}^\sigma \leqslant m_{i},	\\
	\mathrm C^\sigma(i, j):	&
	\min\{ p_{\sigma^{-1}(i)}^\sigma , q_{\sigma^{-1}(j)}^\sigma \}
	 + \min\{q_{\sigma^{-1}(i)}^\sigma, p_{\sigma^{-1}(j)}^\sigma \} 
	 \geqslant \#( \nu_i \cap \nu_j ),	\\
	 \mathrm B(i):	&
	 0 \leqslant p_i \leqslant m_i,	\\
	 \mathrm C(i,j):	&
	 \exists \sigma \in \Sigma_r(i,j), 
	\min\{ p_{\sigma^{-1}(i)}^\sigma , q_{\sigma^{-1}(j)}^\sigma \}
	 + \min\{q_{\sigma^{-1}(i)}^\sigma, p_{\sigma^{-1}(j)}^\sigma \} 
	 \geqslant \#( \nu_i \cap \nu_j ),
	\end{aligned}\]
as subsets of $\MR$, and prove
	\begin{equation}\label{eq: simplify}
	\bigcap_{i=1}^{r} 
	\bigcap_{\sigma \in \Sigma_r}
	 \mathrm B^\sigma (i) 
	\cap
	\bigcap_{ i < j } 
	\bigcap_{\sigma \in \Sigma_r(i,j)}
	\mathrm C^\sigma(i,j)
	=
	\bigcap_{i=1}^r \mathrm B(i) 
	\cap
	\bigcap_{\mbox{\tiny{
		$\begin{array}{c}
		i<j \\
		\textrm{neighbors}
		\end{array}$}
	}} \mathrm C(i,j).
	\end{equation}

\subsubsection{Condition $\mathrm B$}\label{simplify-B}
	We first simplify the condition on B:
	\begin{equation}\label{eq: simplify-B}
	\bigcap_{i=1}^{r} 
	\bigcap_{\sigma \in \Sigma_r}
	 \mathrm B^\sigma (i) 
	\cap
	\bigcap_{ i < j } 
	\bigcap_{\sigma \in \Sigma_r(i,j)}
	\mathrm C^\sigma(i,j)
	=
	\bigcap_{i=1}^{r} 
	 \mathrm B (i) 
	\cap
	\bigcap_{ i < j } 
	\bigcap_{\sigma \in \Sigma_r(i,j)}
	\mathrm C^\sigma(i,j).
	\end{equation}
	It would follow from next lemma and Lemma \ref{decompose-of-perm}.

\begin{lem}\label{B-independent}
	$\mathrm B^\sigma(i) \cap \mathrm B^\sigma(j) \cap \mathrm C^\sigma(i,j) = 
	\mathrm B^{\sigma s}(i) \cap \mathrm B^{\sigma s}(j) \cap \mathrm C^{\sigma s}(i,j)$
	for $s=(h, h+1)$ and $\{i,j\} = \{\sigma(h), \sigma(h+1)\}$ or $\{i,j\} \cap \{\sigma(h), \sigma(h+1)\}=\varnothing$.
\end{lem}
\begin{proof}
	If $\{i,j\} \cap \{\sigma(h), \sigma(h+1)\}=\varnothing$, then it following trivially  from $\mathrm B^\sigma(i) = \mathrm B^{\sigma s} (i)$, $\mathrm C^\sigma(i, j) = \mathrm C^{\sigma s}(i,j)$.
	
	If $\{i,j\} = \{\sigma(h), \sigma(h+1)\}$, then we may assume moreover $\nu_i = \nu_h^\sigma \supset \nu_{h+1}^\sigma = \nu_j$.
	The transition map beteween $\underline p^\sigma$ and $\underline p^{\sigma s}$ is 
	\[\begin{tikzpicture}
	\begin{scope}
	\draw[->] (0, 0) -- (1, 0);
	\draw (0,0.1)--(0,-0.1);
	\draw[->] (1, -0.5) -- (0, -0.5);
	\draw (1,-0.4)--(1,-0.6);
	\node at (-1, -0.02) {$(p_h^\sigma, p_{h+1}^\sigma)$};
	\node at (3, -0.02) {$(q_{h+1}^\sigma, p_h^\sigma + p_{h+1}^\sigma - q_{h+1}^\sigma)$,};
	\node at (2, -0.52) {$(p_h^{\sigma s}, p_{h+1}^{\sigma s})$.};
	\node at (-1.9, -0.52) {$(p_h^{\sigma s} + p_{h+1}^{\sigma s} - q_{h}^{\sigma s}, q_{h}^{\sigma s})$};
	\end{scope}.
	\end{tikzpicture}\]
	For $\underline p^\sigma \in \mathrm B^\sigma(i) \cap \mathrm B^\sigma(j) \cap \mathrm C^\sigma(i, j)$, we have easily $0 \leqslant p^{\sigma s}_h = q^\sigma_{h+1} \leqslant m_j$, moreover
	\[\begin{aligned}
	p^{\sigma s}_{h+1} &= p^\sigma_h + p^\sigma_{h+1} - q^\sigma_{h+1} \geqslant 
	m_{\sigma(h+1)} - q^\sigma_{h+1} \geqslant 0,	\\
	p^{\sigma s}_{h+1} &= p^\sigma_h + p^\sigma_{h+1} - q^\sigma_{h+1} \leqslant 
	m_{\sigma(h)} - q_{h+1}^\sigma \leqslant m_i.
	\end{aligned}\]
	Hence $\underline p^{\sigma s} \in \mathrm B^{\sigma s}(i) \cap \mathrm B^{\sigma s}(j)$; 
	also in $\mathrm C^{\sigma s}(i,j)$ according to $p^\sigma_h + p^\sigma_{h+1} = p^{\sigma s}_h + p^{\sigma s}_{h+1}$.
	
	For $\underline p^{\sigma s} \in  \mathrm B^{\sigma s}(i) \cap \mathrm B^{\sigma s}(j) \cap \mathrm C^{\sigma s}(i,j)$, we have 
	$0 \leqslant p^\sigma_{h+1}= q^{\sigma s}_{h+1} \leqslant m_j$, moreover
	\[\begin{aligned}
	p^\sigma_h &= p_h^{\sigma s} + p_{h+1}^{\sigma s} - q_{h}^{\sigma s} \geqslant 
	m_j - q_h^{\sigma s} \geqslant 0,	\\
	p^\sigma_h &= p_h^{\sigma s} + p_{h+1}^{\sigma s} - q_{h}^{\sigma s} \leqslant 
	m_i - q_{h}^{\sigma s} \leqslant m_i.
	\end{aligned}\]
	Hence $\underline p^\sigma \in \mathrm B^\sigma(i) \cap \mathrm B^\sigma(j)$; 
	also in $\mathrm C^\sigma(i,j)$ according to $p^\sigma_h + p^\sigma_{h+1} = p^{\sigma s}_h + p^{\sigma s}_{h+1}$.
\end{proof}
	
	Let's deduce equation \eqref{eq: simplify-B}.
	If $\sigma, \sigma s\in \Sigma_r$ with $s=(h, h+1)$ a transposition, chose $j$ such that either $\{i, j\} = \{\sigma(h), \sigma(h+1)\}$ or they have no common elements, then
	\[\mathrm B^\sigma(i) \cap 
	\mathrm B^\sigma(j) \cap
	\bigcap_{k< l} 
	\bigcap_{\tau \in \Sigma_r(k, l)} 
	\mathrm C^\tau (k, l) 
	= 
	\mathrm B^{\sigma s}(i) \cap 
	\mathrm B^{\sigma s}(j) \cap
	\bigcap_{k< l} 
	\bigcap_{\tau \in \Sigma_r(k, l)} 
	\mathrm C^\tau (k, l) .
	 \]
	 Run over $i$ from $1$ to $r$, then
	 \[\bigcap_{i=1}^r \mathrm B^\sigma(i) \cap 
	\bigcap_{i< j} 
	\bigcap_{\tau \in \Sigma_r(i,j)} 
	\mathrm C^\tau (i,j) 
	= 
	\bigcap_{i=1}^r \mathrm B^{\sigma s}(i) \cap 
	\bigcap_{i< j} 
	\bigcap_{\tau \in \Sigma_r(i,j)} 
	\mathrm C^\tau (i,j) .\]
	Thus, the set
	\[\bigcap_{i=1}^r \mathrm B^\sigma(i) \cap 
	\bigcap_{i< j} 
	\bigcap_{\tau \in \Sigma_r(i,j)} 
	\mathrm C^\tau (i,j) \]
	is independent of the choice of $\sigma \in \Sigma_r$ according Lemma \ref{decompose-of-perm}.
	Now equation \eqref{eq: simplify-B} follows.

\subsubsection{Structure of $\Sigma_r(i,j)$}
	Then we have to simplify the condition on C:
	\begin{equation}\label{eq: simplify-C}
	\bigcap_{i<j} 
	\bigcap_{\sigma \in \Sigma_r(i,j)} 
	\mathrm C^\sigma (i,j) 
	=
	\bigcap_{\mbox{\tiny{
		$\begin{array}{c}
		i<j \\
		\textrm{neighbors}
		\end{array}$
	}}}
	\mathrm C (i,j) .
	\end{equation}
	Here we study the set $\Sigma_r(i, j)$.
	
\begin{lem}\label{decom-of-neigh}
	$\forall \sigma, \tau \in \Sigma_r(i,j)$, there is a sequence $\tau = \sigma_0, \cdots, \sigma_l = \sigma$ in $\Sigma_r(i,j)$ such that each $\sigma_{i-1}^{-1} \sigma_i$ is in one of the following forms:
	\begin{enumerate}[label=$(\arabic*)$]
	\item transposition $(h, h+1)$, with $\{\sigma_i(h), \sigma_i(h+1 )\} = \{i, j\}$;
	\item transposition $(h, h+1)$, with $\{\sigma_i(h), \sigma_i(h+1 )\} \cap \{i, j\} =\varnothing$;
	\item $3$-circle $(h-1, h, h+1)$ with $\{\sigma_i(h-1), \sigma_i(h) \} = \{i, j\}$, or its inverse $(h+1, h, h-1)$ with $\{ \sigma_i (h), \sigma_i(h+1)\} = \{i, j\}$).
	\end{enumerate}
\end{lem}
\begin{rmk}
	Case (1) appears if and only if $\nu_i, \nu_j$ are in containment relation.
\end{rmk}
\begin{proof}
	We may assume $j = i+1$ such that $\Id \in \Sigma_r(i, j)$.
	Then we assume $\tau = \Id$ and prove by induction on the length $l(\sigma)$.
	If $l(\sigma) = 0$, then $\sigma = \Id$, and there is nothing to prove.
	If $l(\sigma) \geqslant 1$, according to Lemma \ref{decompose-of-perm}, there is $h$ such that $\rho= \sigma \circ (h, h+1) \in \Sigma_r$, and $l(\rho) < l(\sigma)$.
	
	If $\{\sigma(h), \sigma(h+1)\} = \{i, i+1\}$, or they have no common elements, then $\rho$ is also in $\Sigma_r(i,j)$, and we may apply induction hypothesis to it.
	If $\{\sigma(h), \sigma(h+1)\}$ and $\{i, i+1\}$ have exactly one common element, there are four possibilities:
	\begin{enumerate}[label=(\alph*)]
	\item $i = \sigma(h+1)$, $i+1 = \sigma(h+2)$;
	\item $i= \sigma(h)$, $i+1 = \sigma(h-1)$;
	\item $i+1 =\sigma(h+1)$, $i= \sigma(h+2)$;
	\item $i+1 = \sigma(h)$, $i=\sigma(h-1)$.
	\end{enumerate}
	In the possibility (b), $\sigma \circ (h-1, h)$ lies in $\Sigma_r(i, i+1)$, and is shorter than $\sigma$, so we can apply induction hypothesis.
	The possibility (c) can be dealt with in the same manner.
	In the possibility (a), $\rho = \sigma\circ (h, h+1)$ is shorter than $\sigma$, so $\rho(h+1) > \rho(h) = i$.
	Since $\rho(h+1) \not = \rho(h+2) = \sigma(h+2) = i+1$, it must be larger that $i+1$.
	Consequently, $\rho \circ (h+1, h+2) \in \Sigma_r(i,j)$ is shorter that $\rho$.
	Now we can apply induction hypothesis to $\sigma \circ (h, h+1) \circ (h+1, h+2) = \sigma \circ (h, h+1, h+2)$.
	The possibility (d) can be dealt with in the same manner.
\end{proof}

\subsubsection{Condition $\mathrm C$}
	Now we can verify \eqref{eq: simplify-C}.
	It's trivial that 
	\[\bigcap_{i<j} 
	\bigcap_{\sigma \in \Sigma_r(i,j)} 
	\mathrm C^\sigma (i,j) 
	= \bigcap_{\mbox{\tiny{
		$\begin{array}{c}
		i<j \\
		\Sigma_r(i,j) \not = \varnothing
		\end{array}$
	}}}
	\bigcap_{\sigma \in \Sigma_r(i,j)} 
	\mathrm C^\sigma (i,j). \]
	
\begin{lem}\label{C-independent}
	If $\nu_i, \nu_j$ are neighbors, then $\mathrm C^\sigma(i,j)$ is independent of the choice of $\sigma \in \Sigma_r(i,j)$.
\end{lem}
\begin{proof}
	It suffices to prove $\mathrm C^\sigma(i,j) = \mathrm C^{\sigma \tau}(i,j)$ with $\tau$ in one of the conceret forms in Lemma \ref{decom-of-neigh}.
	
	In case (1), $\tau = (h, h+1)$ and $\{ \sigma (h), \sigma(h+1)\} = \{i, j\}$, and we may assume moreover $\nu_h^\sigma \supset \nu_{h+1}^\sigma$.
	Then $\mathrm C^\sigma(i,j) = \mathrm C^{\sigma \tau}(i,j)$  follows from the fact that $p^\sigma_h + p^\sigma_{h+1} = p^{\sigma \tau}_h + p^{\sigma \tau}_{h+1}$.
	
	In case (2), $\tau = (h, h+1)$ and $\{ \sigma (h), \sigma(h+1)\} \cap \{i, j\} = \varnothing$.
	Then $\mathrm C^\sigma(i,j) = \mathrm C^{\sigma \tau}(i,j)$  follows from the fact that $p^\sigma_h = p^{\sigma \tau}_h, p^\sigma_{h+1} = p^{\sigma \tau}_{h+1}$.
	
	In case (3), we may assume $\tau = (h+1, h, h-1)$, and  $\{\sigma(h-1), \sigma(h) \} = \{i, j\}$.
	Since $i,j$ are neighbours, either $\nu_{h-1}^\sigma \subseteq \nu_h^\sigma \cap \nu_{h+1}^\sigma$, or $\nu_{h-1}^\sigma \supseteq \nu_h^\sigma \cup \nu_{h+1}^\sigma$.
	If $\nu_{h-1}^\sigma \subseteq \nu_h^\sigma \cap \nu_{h+1}^\sigma$, then the transition map between $\underline p^\sigma$, $\underline p^{\sigma \tau}$ would be 
	\[\begin{aligned}
	(p^\sigma_{h-1}, p^\sigma_h, p^\sigma_{h+1} )
	\stackrel{\phi_\sigma^{\sigma\circ (h, h+1)}}\mapsto	&
	(p^\sigma_{h-1}, q^\sigma_{h+1}, p^\sigma_h + p^\sigma_{h+1} - q^\sigma_{h+1} )	\\
	\stackrel{\phi_{\sigma \circ (h, h+1)}^{\sigma \tau}}
	\mapsto	&
	(p_{h+1}^\sigma, p_{h-1}^\sigma + q_{h+1}^\sigma - p^\sigma_{h+1}, p^\sigma_h + p^\sigma_{h+1} - q^\sigma_{h+1}).
	\end{aligned}\]
	Hence, we have 
	$p_h^{\sigma \tau} + p_{h+1}^{\sigma \tau} =
	p_{h-1}^\sigma + p_h^\sigma$,
	and then $\mathrm C^{\sigma\tau}(i, j) = \mathrm C^{\sigma}(i,j)$.
	If $\nu_{h-1}^\sigma \supseteq \nu_h^\sigma \cap \nu_{h+1}^\sigma$, then the transition map between $\underline p^\sigma$, $\underline p^{\sigma \tau}$ would be 
	\[\begin{aligned}
	(p^\sigma_{h-1}, p^\sigma_h, p^\sigma_{h+1} )
	\stackrel{\phi_\sigma^{\sigma\circ (h, h+1)}}\mapsto	&
	(p^\sigma_{h-1}, p^\sigma_h + p^\sigma_{h+1} - q^\sigma_{h} , q^\sigma_{h} )	\\
	\stackrel{\phi_{\sigma \circ (h, h+1)}^{\sigma \tau}} \mapsto	&
	(p_{h-1}^\sigma + p_h^\sigma + p_{h+1}^\sigma - q_{h-1}^\sigma - q_h^\sigma, q_{h-1}^\sigma , q^\sigma_{h} ).
	\end{aligned}\]
	Hence, we have 
	$p_h^{\sigma \tau} + p_{h+1}^{\sigma \tau} =
	q_{h-1}^\sigma + q_h^\sigma$,
	and then $\mathrm C^{\sigma\tau}(i, j) = \mathrm C^{\sigma}(i,j)$.	
\end{proof}

	Consequently, for neighbors $\nu_i, \nu_j$, each $\mathrm C^\sigma(i,j)$, and hence their intersection, is equal to $\mathrm C(i,j)$, and we may deduce that 
	\[\bigcap_{\mbox{\tiny{
		$\begin{array}{c}
		i<j \\
		\Sigma_r(i,j) \not = \varnothing
		\end{array}$}
	}}
	\bigcap_{\sigma \in \Sigma_r(i,j)} 
	\mathrm C^\sigma (i,j)
	=
	\left( \bigcap_{\mbox{\tiny{
		$\begin{array}{c}
		i<j \\
		\textrm{neighbors}
		\end{array}$}
	}}
	\mathrm C (i,j) \right)
	\cap
	\left( \bigcap_{\mbox{\tiny{
		$\begin{array}{c}
		i<j \\
		\textrm{not neighbors}
		\end{array}$
	}}}
	\bigcap_{\sigma \in \Sigma_r(i,j)} 
	\mathrm C^\sigma (i,j) \right).
	\]

\begin{lem}
	If $\nu_i \supset \nu_j$ but are not neighbors, 
	then $\forall \sigma \in \Sigma_r(i,j)$, there is a $\nu_k$ between $\nu_i, \nu_j$, and $\tau \in \Sigma_r(i,k) \cap \Sigma_r(k, j)$, 
	such that $\mathrm C^\sigma(i,j) \supseteq \mathrm C^{\tau}(i,k) \cap \mathrm C^{\tau}(k,j)$.
\end{lem}
\begin{proof}
	Since $\nu_i \supset \nu_j$ are not neighbors, there is a $\nu_k$ between them.
	For $\sigma \in \Sigma_r(i,j)$, denote $h$ such that $\{ \sigma(h), \sigma(h+1)\} = \{i, j\}$, 
	and $g= \sigma^{-1}(k)$.
	
	Consider the case $g < h$.
	If $g<h-1$, we may enlarge $g$ but retain the property $\nu_i \supset \nu^\sigma_g \supset \nu_j$ by following methods:
	\begin{enumerate}[label=(\alph*)]
	\item if $\nu^\sigma_{g+1}$ are in containment relation with $\nu_k = \nu^\sigma_g$, then 
	 we may replace $\sigma$ by $\sigma \circ (g, g+1)$; 
	in other words, swap $\nu_g^\sigma$ and $\nu_{g+1}^\sigma$ in $\nu^\sigma$;
	\item if not the above case, so $\nu^\sigma_{g+1} < \nu_k =  \nu^\sigma_g$ according to the admissibility of $\sigma$,
	then $\nu^\sigma_{g+1}\subset \nu_i$ since it begins after $\nu_i$ but can not be preceded by $\nu_i$, 
	and $\nu^\sigma_{g+1} \supset \nu_j$ since it ends after $\nu_j$ but can not be preceded by $\nu_j$;
	we may simply replace $k=\sigma(g)$ by $\sigma(g+1)$.
	\end{enumerate}
	Note that for non-neighbors $\nu_i, \nu_j$, $\mathrm C^\sigma(i,j) = \mathrm C^\tau(i,j)$ if
	there is a sequence in $\Sigma_r(i,j)$, from $\sigma$ to $\tau$, with difference belong to case (1) or (2) in Lemma \ref{decom-of-neigh}.
	Then in the possibility (a), $\sigma$ is changed but $\mathrm C^\sigma(i,j)$ is preserved, 
	while in the possibility (b), the segment $\nu_k$ are changed.
	By an easy induction we would reach the case $\sigma^{-1}(k)=g = h-1$.
	
	If it is the case $g> h+1$, then we may assume $g=h+2$ after a similar argument.
	
	We can prove for such a $\nu_k$ and $\sigma$ that $\mathrm C^\sigma(i,j) \supseteq \mathrm C^{\tau}(i,k) \cap \mathrm C^{\tau}(k,j)$ for some $\tau$.
	It's harmless to assume $\sigma(h)=i$, $\sigma(h+1)=j$.
	In the case $\sigma^{-1}(k) = h-1$, take $\tau = \sigma \circ (h-1, h)$. 
	The transition map between $\underline p^\tau$, $\underline p^{\sigma}$ is
	\[(p_{h-1}^\tau, p_h^\tau, p_{h+1}^\tau) \mapsto 
	(q_h^\tau, p_{h-1}^\tau+ p_h^\tau - q_{h}^\tau, p_{h+1}^\tau),\]
	so
	\[ p^\sigma_h + p^\sigma_{h+1} = 
	p_{h-1}^\tau+ p_h^\tau - q_{h}^\tau+  p_{h+1}^\tau= 
	(p_{h-1}^\tau+ p_h^\tau) + (p_h^\tau + p_{h+1}^\tau) - m_k .\]
	Then by $ m_k \leqslant  p_{h-1}^\tau + p_h^\tau \leqslant m_i$ and $m_j \leqslant  p_h^\tau + p_{h+1}^\tau \leqslant m_k$ we obtain
	\[m_j \leqslant p^\sigma_h + p^\sigma_{h+1} \leqslant m_i.\]
	In the case $\sigma^{-1}(k) = h+2$, take $\tau = \sigma \circ (h+1, h+2)$. 
	The transition map between $\underline p^\tau$, $\underline p^{\sigma}$ is
	\[(p_h^\tau, p_{h+1}^\tau, p_{h+2}^\tau) \mapsto
	(p_h^\tau, q_{h+2}^\tau, p_{h+1}^\tau+ p_{h+2}^\tau - q_{h+2}^\tau),\]
	so
	\[ p^\sigma_h + p^\sigma_{h+1} = 
	p_{h}^\tau+ q_{h+2}^\tau = 
	(p_h^\tau + p_{h+1}^\tau) - (p_{h+1}^\tau + p_{h+2}^\tau) + m_j .\]
	Then by $ m_k \leqslant  p_{h}^\tau + p_{h+1}^\tau \leqslant m_i$ and $m_j \leqslant  p_{h+1}^\tau + p_{h+2}^\tau \leqslant m_k$ we obtain
	\[m_j \leqslant p^\sigma_h + p^\sigma_{h+1} \leqslant m_i.\]
	Now the conclusion follows.
\end{proof}

	Consequently, 
	\[\bigcap_{\mbox{\tiny{
		$\begin{array}{c}
		i<j \\
		\textrm{neighbors}
		\end{array}$
	}}}
	\mathrm C (i,j) 
	\subseteq
	\bigcap_{\mbox{\tiny{
		$\begin{array}{c}
		i<j \\
		\textrm{not neighbors}
		\end{array}$
	}}}
	\bigcap_{\sigma \in \Sigma_r(i,j)} 
	\mathrm C^\sigma (i,j),\]
	and we may conclude \eqref{eq: simplify-C}.
	
	Then \eqref{eq: simplify} follows by combining \eqref{eq: simplify-B} and \eqref{eq: simplify-C}.

\subsection{Necessity in the criterion}\label{necessity}
	In this subsection we prove if $\underline p \in \MR$ parametrize a non-zero $A_\q(\lambda)$, then $\underline p$ lies in all $\mathrm B^\sigma(i)$ and $\mathrm C^\sigma(i,j)$.
	We will assume Lemma \ref{non-vanish-2}, which follows easily from Trapa's criterion, as explained in the front of section \ref{sufficiency}.

	If $\underline p$ parametrizes a non-zero $A_\q(\lambda)$, then as a priori it lies in all $\mathrm B(i)$.
	It also lies in all $\mathrm C^{\Id}(k, k+1)$ due to Lemma \ref{non-vanish-2}, 
	Proposition \ref{prop: inductioin-in-stage}, and Lemma \ref{lem: induction-by-component}:
	let $\underline p'=(p_1, \cdots, p_k + p_{k+1}, \cdots, p_r)$, then there are
	$\q_{\underline p} \subseteq \q_{\underline p'} \subseteq \g$,
	and $\p= \q_{\underline p} \cap \l_{\underline p'}$, such that
	\[A_{\q_{\underline p}}(\lambda_{\underline p}) = 
	\mathcal L_{\q_{\underline p}}^\g( \C_{\lambda_{\underline p}})\cong
	\mathcal L_{\q_{\underline p'}}^\g\left( 
		\mathcal L_\p^{\l_{\underline p'}}(\C_{\lambda_{\underline p}})
	\right), \]
	so $A_{\q_{\underline p}}(\lambda_{\underline p}) \not =0$ implies 
	$\mathcal L_\p^{\l_{\underline p'}}(\C_{\lambda_{\underline p}}) \not=0$;
	$L_{\underline p'}$ decomposes as $L \times \U(p_k+p_{k+1}, q_k+ q_{k+1})$ with $\p \cong \l \oplus \q_{(p_k, p_{k+1})}$, 
	$\C_{\lambda_{\underline p}} \cong 
	\C_{\lambda_{\underline p}|_L} \boxtimes 
	\C_{\lambda_{(p_k, p_{k+1})}}$, then
	\[\mathcal L_\p^{\l_{\underline p'}}(\C_{\lambda_{\underline p}}) \cong
	\C_{\lambda_{\underline p}|_L} \boxtimes 
	\mathcal L_{\q_{(p_k, p_{k+1})}}^{\gl(m_k+ m_{k+1}, \C)}(
		\C_{\lambda_{(p_k, p_{k+1})}}),
	\]
	so $\mathcal L_\p^{\l_{\underline p'}}(\C_{\lambda_{\underline p}}) \not =0$ implies 
	$\mathcal L_{\q_{(p_k, p_{k+1})}}^{\gl(m_k+ m_{k+1}, \C)}
	(\C_{\lambda_{(p_k, p_{k+1})}}) \not =0$;
	then $\min\{p_k, q_{k+1}\} + \min\{q_k, p_{k+1}\} \geqslant \#(\nu_k \cap \nu_{k+1})$ due to Lemma \ref{non-vanish-2},
	which is exactly $\mathrm C^\Id(k, k+1)$.
\begin{rmk}
	An alternative approach is the purely combinatorial Trapa's algorithm and Lemma \ref{overlap-p}.
\end{rmk}
	
	It follows that $\underline p$ lies in all $\mathrm B^s(s(i))$ (regarded as subsets of $\MR$), for any transposition $s\in \Sigma_r$, 
	due to Lemma \ref{B-independent}.
	Then  $\underline p^s\in \MR^s$ parametrizes the module
	 $A_{\q_{\underline p^s}}(\lambda_{\underline p^s})$, 
	 and we have proved it to be isomorphic with  $A_{\q_{\underline p}}(\lambda_{\underline p}) \not =0$.
	Hence, $\underline p^s$ lies in all $\mathrm C^s(s(i), s(i+1))$ (as subsets of $\MR^s$), and equivalently, $\underline p$ lies in all $\mathrm C^s( s(i), s(i+1))$ (as subsets of $\MR$).
	Since $\Sigma_r$ is ``generated by transpositions'' (in the sense of Lemma \ref{decompose-of-perm}), $\underline p$ lies in all $\mathrm B^\sigma( \sigma(i))$ and $\mathrm C^\sigma( \sigma(i), \sigma(i+1))$.
	
	It's clear that
	\[\bigcap_{i=1}^{r} 
	\bigcap_{\sigma \in \Sigma_r}
	 \mathrm B^\sigma (i) 
	\cap
	\bigcap_{ i < j } 
	\bigcap_{\sigma \in \Sigma_r(i,j)}
	\mathrm C^\sigma(i,j)
	=
	\bigcap_{\sigma \in \Sigma_r}
	 \left( \bigcap_{h=1}^r
	 \mathrm B^\sigma( \sigma(h)) 
	 \cap
	  \left (\bigcap_{h=1}^{r-1}
	   \mathrm C^\sigma( \sigma(h), \sigma(h+1)) \right) \right),\]
	   then the necessity of the critierion in Theorem \ref{non-vanishing} follows.

%% file: Method.tex
\section{Trapa's algorithm}\label{tableau}
	In this section we will review Trapa's algorithm on determing whether an $A_\q(\lambda)$ is non-zero. 
	Many combinatorial notations are involved.
	Our definitions follow \cite{Trapa}.
	
\subsection{$\nu$-antitableau}
	Here we introduce the following objects.
\begin{itemize}
	\item \textbf{$\nu$-antitableau}, a Young diagram filled by elements of the real vector $\nu$ in a certain way.
	They are combinatorial counterpart of primitive ideals (with infinitesimal character $\nu$).
	\item \textbf{Skew column}, certain part of a Young diagram.
	It records the cohomological induction process on the combinatorial side.
\end{itemize}

	Given a partition $\underline n=(n_1 \geqslant \cdots \geqslant n_k)$ of $n$, we may attach a left justifed arrangement of $n$ boxes with $n_i$ boxes in the $i$th row. 
	Call such an arrangement a Young diagram of size $n$, and denote it by $S_{\underline n}$, or simply $S$ if $\underline n$ is clear. 
	For $\nu=(\nu_1, \cdots, \nu_n) \in \R^n$, a $\nu$-quasitableau is defined to be any arrangement of $\nu_1, \cdots, \nu_n$ in a Young diagram $S$ of size $n$.
	$S$ is called shape of this $\nu$-quasitableau.
	If a $\nu$-quasitableau satisfes the condition that the entries weakly decrease across rows and strictly decrease down columns, it is said to be a $\nu$-antitableau. 

	A skew diagram (resp. skew tableau) is obtained by removing a smaller Young diagram (resp. quasitableau) in a larger one.
	 A skew column is a skew diagram which  has at most one box per row.\footnote{
	This definition is slight different with \cite{Trapa}, the skew column there is a part of $\nu$-antitableau; here we emphasis on the shape.
}
	By $i$th box of a skew column we mean the $i$th one counted from top to the bottom.
	
	A partition $S = \bigsqcup S_k$ of Young diagram $S$ into skew columns should satisfy the following conditions:
\begin{itemize}
	\item each $S_k$ is a skew column;
	\item each $S^k:= \bigsqcup_{j\leqslant k} S_j$ is a Young diagram.
\end{itemize}
	Then $k$th skew column $S_k$ has at most $k$ columns.
	By $i$th \textbf{component} of $S_k$ we mean its boxes on the $(k+1-i)$th column of $S$.
	We may associate a (weakly) increasing sequence 
	$\{L_{k,i} \mid i=1,\cdots, k \}$
	to $S_k$, by defining $L_{k, i}$ to be the number of boxes in its first $i$ components.
	This sequence, together with 
	$S^{k-1} = \bigsqcup_{j \leqslant k-1} S_j$, 
	detemines $S_k$ completely.
	We call it the \textbf{type} of $S_k$.
	It is harmless to extend the sequence by 
	\[L_{k, i } = \left\{\begin{aligned}
		&0,	&i \leqslant 0,\\
		&L_{k, k},	&i \geqslant k+1.
	\end{aligned}\right.\]
	Thess definitions can be easily extended to a skew diagram.
	
	For a $\nu$-quasitableau $S_\anti$ with shape $S$, the partition $S= \bigsqcup S_k$ into skew column will be said \textbf{compatible} with $S_\anti$ if the following condition holds:
	denote the $i$th entry of $S_k$ by $\nu_{k,i}$, then each sequence $\nu_{k, 1}, \nu_{k, 2}, \cdots$ constitutes to a segment (decreasing with step one).
	We may understand the $\nu$-filling on $S_k$ by the (weakly) decreasing sequence 
	$\nu_{k; i}:= \nu_{k, L_{k, i}} = \nu_{k, 1} - L_{k, i} +1, i \in \Z,$
	and call it \textbf{type} of this $\nu$-filling.

\begin{eg}
	Here is a compatible partition $S = S_1 \sqcup S_2 \sqcup S_3$ of a $\nu$-quasitableau.
\[\begin{ytableau}
	7	&\textcolor{purple} 5	&\textcolor{blue} 5	\\
	6	&\textcolor{purple} 4	&\textcolor{blue} 4	\\
	5	&\textcolor{purple} 3	&\none	\\
	4	&\textcolor{blue} 3	&\none	\\
	3	&\textcolor{blue} 2	&\none	\\
	\textcolor{purple} 2	&\none	\\
	\textcolor{blue} 1	&\none	\\
	\none[\tikznode{S}{~}]
	\end{ytableau}
	=\quad
	\begin{ytableau}
	7	\\
	6	\\
	5	\\
	4	\\ 
	3	\\ 
	\none	\\
	\none	\\
	\none[\tikznode{S_1}{~}]
	\end{ytableau}
	\quad\bigsqcup\quad
	\begin{ytableau}
	\none	&\textcolor{purple} 5	\\
	\none	&\textcolor{purple} 4	\\
	\none	&\textcolor{purple} 3	\\
	\none	&\none	\\
	\none	&\none	\\
	\textcolor{purple} 2	&\none	\\
	\none\\
	\none[\tikznode{S_2}{~}]
	\end{ytableau}
	\quad \bigsqcup\quad
	\begin{ytableau}
	\none	&\none	&\textcolor{blue} 5	&\none[\tikznode{L_30}{~}]\\
	\none	&\none	&\textcolor{blue} 4	&\none[\tikznode{L_31}{~}]\\
	\none	&\none	&\none	\\
	\none	&\textcolor{blue} 3	&\none	\\
	\none	&\textcolor{blue} 2	&\none	&\none[\tikznode{L_32}{~}]\\
	\none	&\none	\\
	\textcolor{blue} 1	&\none	&\none	&\none[\tikznode{L_33}{~}]\\
	\none[\tikznode{S_3}{~}]
	\end{ytableau}\]
	In this example, the skew column $S_3$ has type $L_{3,1} = 2$, $L_{3, 2} = 4$, $L_{3, 3} = 5$.
	Its $\nu$-filling has type $\nu_{3;0} = 6$, $\nu_{3; 1} = 4$, $\nu_{3; 2} = 2$, $\nu_{3; 3}=1$.

\tikz[overlay,remember picture]{
	\draw[decorate,decoration={brace},thick]
	([yshift=1em, xshift=0.5em] L_30.north west) -- 
	([yshift=-0.5em, xshift=0.5em]L_31.south west)
	node[midway, right]{$L_{3,1}$};
	\draw[decorate,decoration={brace},thick]
	([yshift=1em, xshift=2.75em] L_30.north west) -- 
	([yshift=-0.5em, xshift=2.75em] L_32.south west)
	node[midway, right]{$L_{3,2}$};
	\draw[decorate,decoration={brace},thick]
	([yshift=1em, xshift=5em] L_30.north west) -- 
	([yshift=-0.5em, xshift=5em] L_33.south west)
	node[midway, right]{$L_{3,3}$};
	\node at ([xshift = 2em] S) {$S$};
	\node at ([xshift = 0em] S_1) {$S_1$};
	\node at ([xshift = 1em] S_2) {$S_2$};
	\node at ([xshift = 2em] S_3) {$S_3$}
}
\end{eg}

\begin{lem}\label{anti-by-type}
	Suppose $S= \bigsqcup S_k$ is a compatible partition into skew columns with a $\nu$-quasitableau $S_\anti$.
	Denote the type of $\nu$-filling on $S_k$ by $\nu_{k; i}$.
	Then $S_\anti$ is a $\nu$-antitableau if and only if 
	$\forall k, i,$ $\nu_{k; i} \geqslant \nu_{k+1; i}$.
\end{lem}
\begin{proof}
	Look at the last $(i+1)$th column of $S_k \sqcup S_{k+1}$,
	which consists of the $i$th component of $S_k$ and $(i+1)$th component of $S_{k+1}$, 
	and hence filled by $\nu_{k; i-1} -1, \cdots, \nu_{k; i} , \nu_{k+1; i} -1, \cdots, \nu_{k+1; i+1}$.
	If $S_\anti$ is an antitableau, then by its column-decreasing property, 
	we know $\nu_{k; i} > \nu_{k+1; i}-1$, and hence $\nu_{k; i} \geqslant \nu_{k+1; i}$.
	
	Conversely, if all $\nu_{k;i} \geqslant \nu_{k+1; i}$, then we know $S_\anti$ has entries decreasing down columns immediately.
	Let's deduce further that it has entries weakly decreasing across columns.
	Adjacent boxes on a row come from different skew columns, say $S_k$, $S_l$, and $k<l$.
	The left one is farther to ends of the component it belonging to.
	More precisely, suppose these two boxes are the $\Delta_k$th and $\Delta_l$th of $S_k$, $S_{l}$ respectively,
	with $L_{k; i-1} < \Delta_k \leqslant L_{k; i}$,
	$L_{l; j-1} < \Delta_l \leqslant L_{l; j}$,
	then we would have 
	$L_{k; i} - \Delta_k \geqslant L_{l; j} - \Delta_l$ 
	due to the shape of $S^l =  \bigsqcup_{j \leqslant l} S_j$.
	Moreover, since the two boxes belong to adjacent columns, we know $l - j = k - i +1$, and hence $j = (l- k -1) +i \geqslant i$.
	It follows $\nu_{k; i } \geqslant \nu_{l; j}$ then, and  
	\[\nu_{k, \Delta_k} = \nu_{k; i} + L_{k, i} - \Delta_k \geqslant
	\nu_{l; j} + L_{l; j} - \Delta_{k+1} = \nu_{l, \Delta_l}.\]
	Consequently, entries of $S_\anti$ also weakly decrease across rows.	
\end{proof}

\subsection{Signed tableau}
	Here we introduce \textbf{signed tableaus}.
	They are equivalence classes of certain sign-filled Young diagrams, 
	and are combinatorial counterpart of associated varieties.
	
	The sign-filled Young diagram we consider should have signs alternate across rows.
	If the sign-filling has $p$ pluses and $q$ minuses, call it has signature $(p, q)$.
	Two sign-filled Young diagrams are equivalent, if they can be made to coincide by interchanging rows of equal length.
	An equivalence class of sign-filled Young diagrams is called a \textbf{signed tableau}.
	It has well-defined shapes and signature. 
	
	A sign-filled skew column could be any arrangement of pluses and minuses in this skew column.
	Two sign-filled skew columns are equivalent, if the can be made to coincide by interchanging boxes in same components.
	An equivalence class of sign-filled skew columns is called a \textbf{signed skew column}.

	The reader should notice that the above two equivalence relations are not always compatible:
	for a partition $S= \bigsqcup S_k$ of Young diagram into skew columns,  two equivalent sign-fillings in $S$ could give rise to non-equivalent sign-fillings in $S_k$.
	To avoid this, there is a sufficient condition on the sign-filling and partition of $S$:
	if a row is skipped in $S_k$, then all rows in $S^k= \bigsqcup_{j \leqslant k} S_j$ below (and including) the frst skipped row and above (and including) the last row of $S_k$ end in the same sign.
	In this case we say the sign-filling and partition are \textbf{compatible}.
	One may check the following:
\begin{itemize}
	\item if a sign-filling is compatible with a partition into skew columns, then its equivalent sign-fillings are also compatible with this partition; 
	hence we can talk about the compatibility of a signed tableau and a partition;
	\item equivalent sign-fillings in $S$ give rise to equivalent sign-fillings in $S_k$ if these signed-fillings are compatible with the partition $S=\bigsqcup S_k$.
\end{itemize}

	Given a signed tableau $S_\pm$ and a compatible partition $S=\bigsqcup S_k$ of its shape, we may associate two (weakly) increasing sequence $\{L^\pm_{k, i} \mid i = 1, \cdots, k\}$ to the sign-fillings in $S_k$, by defining $L_{k, i}^+$ (or $L^-_{k, i}$) to be the number of pluses (or minuses) in its first $k$ components.
	They are well-defined up to equivalence relation, and determine this signed skew column completely.
	We call them \textbf{type} of this signed skew column.
	$L^\pm_{k,i}$ can be extend to $i \in \Z$ similarly as $L_{k,i}$, and by definition $L_{k, i}= L_{k, i}^+ + L_{k, i}^-$.
\begin{eg}
	Here is a compatible partition $S = S_1 \sqcup S_2 \sqcup S_3$ of a sign-filled Young diagram.
\[\begin{ytableau}
	+	&\textcolor{purple} -	&\textcolor{blue} +	\\
	+	&\textcolor{purple} -	&\textcolor{blue} +	\\
	-	&\textcolor{purple} +	&\textcolor{blue} -	\\
	+	&\textcolor{blue} -	&\none	\\
	\textcolor{purple} +	&\none	\\
	\textcolor{blue} +	&\none	\\
	\none[\tikznode{S_pm}{~}]
	\end{ytableau}
	=\quad
	\begin{ytableau}
	+	\\
	+	\\
	-	\\
	+	\\
	\none	\\
	\none	\\
	\none[\tikznode{S_pm_1}{~}]
	\end{ytableau}
	\quad\bigsqcup\quad
	\begin{ytableau}
	\none	&\textcolor{purple} -	\\
	\none	&\textcolor{purple} -	\\
	\none	&\textcolor{purple} +	\\
	\none	&\none	\\
	\textcolor{purple} +	&\none	\\
	\none	\\
	\none[\tikznode{S_pm_2}{~}]
	\end{ytableau}
	\quad \bigsqcup\quad
	\begin{ytableau}
	\none	&\none	&\textcolor{blue} +	&\none	\\
	\none	&\none	&\textcolor{blue} +	&\none	\\
	\none	&\none	&\textcolor{blue} -	\\
	\none	&\textcolor{blue} -	&\none	\\
	\none	&\none	\\
	\textcolor{blue} +	&\none	\\
	\none[\tikznode{S_pm_3}{~}]
	\end{ytableau}\]
	In this exapmle, the skew column $S_3$ has type $L_{3,1} = 3$, $L_{3, 2} =4$, $L_{3, 3}= 5$.
	Its sign-filling has type 
	$L_{3, 1}^+ = 2$, $L_{3,1}^- = 1$, 
	$L_{3, 2}^+ = 2$, $L_{3, 2}^- = 2$, 
	$L_{3, 3}^+ = 3$, $L_{3, 3}^- = 2$.
	
\tikz[overlay,remember picture]{
	\node at ([xshift = 2em] S_pm) {$S$};
	\node at ([xshift = 0em] S_pm_1) {$S_1$};
	\node at ([xshift = 1em] S_pm_2) {$S_2$};
	\node at ([xshift = 2em] S_pm_3) {$S_3$}
}	
\end{eg}

\subsection{Understand $A_\q(\lambda)$ by tableaus}\label{construct-tableau}
	Suppose $\nu$ is a weight lattice translation of half the sum of positive roots of $\U(p, q)$. 
	Then an irreducible Harish-Chandra module for $\U(p,q)$ with infinitesimal character $\nu$ is determined completely by its annihilator (which is a primitive ideal in $\gl(n, \C)$) and associated varieties (\cite{BV1983}).
	We do not bother to explain the definition of these two objects (\cite[Section 4, 5]{Trapa} are recommended for interested readers), and switch to the combinatoric side as quick as possible:
\begin{itemize}
	\item the primitive ideals in $\gl(n, \C)$ with infinitesimal character $\nu$ are in bijection to $\nu$-antitableaus (\cite[Theorem 4.1]{Trapa});
	\item the nilpotent orbits (consitituent of associated varieties) in $\u(p,q)$ are in bijection to signed tableaus of signature $(p, q)$ (\cite[Lemma 5.5]{Trapa}).
\end{itemize}
	The associated variety of an irreducible $A_\q(\lambda)$ always consists of a single nilpotent orbit.
	Hence, we may associate a $\nu$-antitableau and a signed tableau to each $A_\q(\lambda)$ when $\lambda$ lies in the mediocre range for $\q$.
	These two tableaus compose into a complete invariant of representations of this type; see Theorem 6.1 in \cite{Trapa}.
	Here we will sketch this construction for the $A_\q(\lambda)$ parametrized by $\underline p \in \mathcal D(\psi) \subset \MR$.
	
	The first step is to construct a signed tableau $S_{\pm}$ with shape $S$ and a compatible partition $S= \bigsqcup S_k$.
	The process would be inductive.
	$S_1$ (which has only $1$ component, alternative speaking, is a one-column Young diagram) is filled with $p_1$ pluses and $q_1$ minuses.
	Now suppose the signed tableau $S^{k-1} = \bigsqcup_{j<k} S_j$ is defined, then 
	define the sign-filled skew column $S_k$ by adding $p_k$ boxes filled with $+$, and $q_k$ boxes filled with $-$, from top to bottom, to row-ends of $S^{k-1}$, such that
\begin{itemize}
	\item at most one box is added to each row-end;
	\item the signs of resulting diagram alternate across rows.
\end{itemize}
	An appropriate sign-filling on $S^{k-1}$ could be chosen in order that  $S^{k-1} \sqcup S_k$ is still a Young diagram (having decreasing rows).
	
	The next step is to fill $S$ by the infinitesimal character $\nu$.
	Regard $\nu$ as admissible arrangement of segments $\nu_1, \cdots, \nu_r$.
	Each $\nu_k = (\nu_{k, 1}, \cdots, \nu_{k, m_k})$ is a segment of length $m_k$, consistent with the size of $S_k$, so its entry $\nu_{k,i}$ can be filled into the $i$th box of $S_k$.
	Then we obtain a $\nu$-quasitableau $S_{\mathrm A}$ with shape $S$.
	If $\lambda$ is in the good range for $\q$, such that $\nu$ is a decerasing sequence, then $S_{\mathrm A}$ is already a $\nu$-antitableau.

\begin{eg}\label{tableau-data}
	Take $(m_1, m_2, m_3) = (3, 5, 6)$, and $\underline p = (2, 2, 2)$.
	Then $\underline q= (1, 3, 4)$.
	Go through the first step, and we would obtain the partition and signed filling as:
	\[\begin{ytableau}
	+	\\
	+	\\
	-	\\
	\none	\\
	\none	\\
	\none	\\
	\none	\\
	\none[\tikznode{S^1}{~}]
	\end{ytableau}
	\quad \Longrightarrow \quad
	\begin{ytableau}
	\tikznode{L_10_1}{-}	&\textcolor{purple} +	\\
	+	&\textcolor{purple} -	\\
	\tikznode{L_11_1}{+}	&\textcolor{purple} -	\\
	\textcolor{purple} +	&\none	\\
	\textcolor{purple} -	\\
	\none	\\
	\none	\\
	\none[\tikznode{S^2}{~}]
	\end{ytableau}
	\quad \Longrightarrow \quad
	\begin{ytableau}
	\tikznode{L_10_2}{+}	&
	\tikznode{L_20_2}{\textcolor{purple} -}	&
	\textcolor{blue} +	\\
	+	&
	\textcolor{purple} -	&
	\textcolor{blue} +	\\
	\tikznode{L_11_2}{-}	&
	\tikznode{L_21_2}{\textcolor{purple} +}&
	\textcolor{blue} -	\\
	\textcolor{purple} +	&
	\textcolor{blue} -	&
	\none	\\
	\tikznode{L_22_2}{\textcolor{purple} -}&
	\none	\\
	\textcolor{blue} -	&
	\none	\\
	\textcolor{blue} -\\
	\none[\tikznode{S^3}{~}]
	\end{ytableau}
	\]
\tikz[overlay,remember picture]{
	\draw[very thick, color=purple]
	([yshift=0.375em, xshift=0.375em] L_10_1.north east) -- 
	([yshift=-0.5em, xshift=0.375em]L_11_1.south east) --
	([yshift=-0.5em, xshift=-0.375em]L_11_1.south west);
	
	\draw[very thick, color=purple]
	([yshift=0.375em, xshift=0.375em] L_10_2.north east) -- 
	([yshift=-0.5em, xshift=0.375em]L_11_2.south east) --
	([yshift=-0.5em, xshift=-0.375em]L_11_2.south west);

	\draw[very thick, color=blue]
	([yshift=0.375em, xshift=0.375em] L_20_2.north east) -- 
	([yshift=-0.5em, xshift=0.375em]L_21_2.south east) --
	([yshift=-0.5em, xshift=-0.375em]L_21_2.south west) --
	([yshift=-0.5em, xshift=0.375em]L_22_2.south east) --
	([yshift=-0.5em, xshift=-0.375em]L_22_2.south west);
	
	\node at ([xshift = 0mm] S^1) {$S^1= S_1$};
	\node at ([xshift = 4mm] S^2) {$S^2 = S_1 \sqcup \textcolor{purple}{S_2}$};
	\node at ([xshift = 8mm] S^3) {$S^3 = S_1 \sqcup \textcolor{purple}{S_2} \sqcup \textcolor{blue}{S_3}$}
}	
	
	In the second step, if take $\nu = (\nu_1, \nu_2, \nu_3)$ as 
	$\nu_1 = (7, 6, 5)$,
	$\nu_2 = (7, 6, 5, 4, 3)$,
	$\nu_3 = (6, 5, 4, 3, 2, 1)$,
	then we would obtain 
	\[\begin{ytableau}
	\tikznode{v_10}7	&
	\tikznode{v_20}{\textcolor{purple} 7}	&
	\textcolor{blue} 6	\\
	6	&\textcolor{purple} 6	&\textcolor{blue} 5	\\
	\tikznode{v_11}5	&
	\tikznode{v_21}{\textcolor{purple} 5}	&
	\textcolor{blue} 4	\\
	\textcolor{purple} 4	&\textcolor{blue} 3	&\none	\\
	\tikznode{v_22}{\textcolor{purple} 3}	&
	\none	\\
	\textcolor{blue} 2	&
	\none	\\
	\textcolor{blue} 1
	\end{ytableau},\]
\tikz[overlay,remember picture]{
	\draw[very thick, color=purple]
	([yshift=0.375em, xshift=0.5em] v_10.north east) -- 
	([yshift=-0.5em, xshift=0.5em]v_11.south east) --
	([yshift=-0.5em, xshift=-0.5em]v_11.south west);

	\draw[very thick, color=blue]
	([yshift=0.375em, xshift=0.5em] v_20.north east) -- 
	([yshift=-0.5em, xshift=0.5em]v_21.south east) --
	([yshift=-0.5em, xshift=-0.525em]v_21.south west) --
	([yshift=-0.5em, xshift=0.525em]v_22.south east) --
	([yshift=-0.5em, xshift=-0.5em]v_22.south west);
}
	which is already a $\nu$-antitableau.
	
	If take  $\nu = (\nu_1, \nu_2, \nu_3)$ as 
	$\nu_1 = (7, 6, 5)$,
	$\nu_2 = (6, 5, 4, 3, 2)$,
	$\nu_3 = (6, 5, 4, 3, 2, 1)$,
	then we would obtain 	
	\[\begin{ytableau}
	\tikznode{v_10_0}7	&
	\tikznode{v_20_0}{\textcolor{purple} 6}	&
	\textcolor{blue} 6	\\
	6	&\textcolor{purple} 5	&\textcolor{blue} 5	\\
	\tikznode{v_11_0}5	&
	\tikznode{v_21_0}{\textcolor{purple} 4}	&
	\textcolor{blue} 4	\\
	\textcolor{purple} 3	&\textcolor{blue} 3	&\none	\\
	\tikznode{v_22_0}{\textcolor{purple} 2}	&
	\none	\\
	\textcolor{blue} 2	&
	\none	\\
	\textcolor{blue} 1
	\end{ytableau}\]
\tikz[overlay,remember picture]{
	\draw[very thick, color=purple]
	([yshift=0.375em, xshift=0.5em] v_10_0.north east) -- 
	([yshift=-0.5em, xshift=0.5em]v_11_0.south east) --
	([yshift=-0.5em, xshift=-0.5em]v_11_0.south west);

	\draw[very thick, color=blue]
	([yshift=0.375em, xshift=0.5em] v_20_0.north east) -- 
	([yshift=-0.5em, xshift=0.5em]v_21_0.south east) --
	([yshift=-0.5em, xshift=-0.525em]v_21_0.south west) --
	([yshift=-0.5em, xshift=0.525em]v_22_0.south east) --
	([yshift=-0.5em, xshift=-0.5em]v_22_0.south west);
}
	and we have to go thourgh one more step, adjusting it to a $\nu$-antitableau.	
\end{eg}

	This third step is for the case that $S_\anti$ is not an antitableau.
	There is an algorithm to adjust it, which ends up in two possibilities:
\begin{itemize}
	\item a $\nu$-antitableau $S_{\mathrm A}'$, and it follows $A_{\q_{\underline p}}(\lambda_{\underline p})$ is non-zero, with annihilator given by this tableau;
	\item a formal symbol $0$, and it follows $A_{\q_{\underline p}}(\lambda_{\underline p})$ is zero.
\end{itemize}
	In the subsequents we will study it in detail, and formulate it in the last of subsection \ref{subtle}.
	Consequently, we can determine whether $A_{\q_{\underline p}}(\lambda_{\underline p})$ is zero by means of this algorithm.

\subsection{Overlap}
	In Trapa's algorithm, a recurring step is comparing ``overlap'' and ``singularity'' of two adjacent skew columns.
	The latter concerns only about the segments filled into them, while the former concerns only about their shapes.
	Here we will explore a formula for overlap  in terms of types (of skew colums). 
	
	We always assume a partition $S = \bigsqcup S_k$ into skew columns.
	Denote the type of $S_k$ by $\{L_{k, i} \mid i = 1, \cdots, k\}$, and its size ($= L_{k, k}$) also by $m_k$.

\begin{din}[{\cite[Definition 7.1]{Trapa}}]
	The overlap of $S_k$ and $S_{k+1}$ is defined to be the largest integer $m \leqslant \min\{m_k, m_{k+1}\}$ such that the following condition for $m$ holds: 
	$\forall i = 1, \cdots, m$, the $(m_k-m+i)$th box of $S_k$ is strictly left to the $i$th box of $S_{k+1}$ (in $S$).
	Denote it by $\overlap(S_{k},S_{k+1})$. 
	
	If the condition for $m\geqslant 1$ never holds, define $\overlap(S_k, S_{k+1})= 0$.
\end{din}

\begin{lem}\label{overlap}
	$\overlap(S_k, S_{k+1}) = \min \{ L_{k+1, i} - L_{k, i} + m_k \mid i \leqslant k\}$.
\end{lem}
\begin{proof}
	Denote the first $i$ components of $S_k$ by $S_k^i$, we will prove by induction that 
	\[ \overlap(S_k^i, S_{k+1}^i) \stackrel ?= \min\{ L_{k+1, j} - L_{k, j} + L_{k, i} \mid j \leqslant i\}.\]
	Case $i=k$ is the desired formula, since the $(k+1)$th component of $S_{k+1}$ contributes nothing to $\overlap(S_k, S_{k+1})$.
	
	Case $i=1$ is easy; $\overlap(S_k^1, S_{k+1}^1) = L_{k+1, 1}$ since they both have only one component.
	Case $i>1$ follows from the recursive formula
\begin{equation}\label{reccursive-overlap}
	\overlap(S_k^{i+1}, S_{k+1}^{i+1}) \stackrel ?= 
	\overlap(S_k^i, S_{k+1}^i) + \min\{L_{k, i+1} - L_{k, i}, L_{k+1, i+1} - \overlap(S_k^i, S_{k+1}^i) \}.
\end{equation}
	Let's explain how to obtain this formula.
	Denote $m^i = \overlap(S_k^i, S_{k+1}^i)$, $\Delta^i = \min\{L_{k, i+1} - L_{k, i}, L_{k+1, i+1} - m^i \}$ for convenience.
	By definition, the $L_{k, i} - m^i + 1, \cdots, L_{k, i}$th boxes of $S_k^{i+1}$ (same as $S_k^i$) are strictly left to the $1, \cdots, m^i$th boxes of $S_{k+1}^{i+1}$ (same as $S_{k+1}^i$) respectively.
	Moreover, the $(i+1)$th component of $S_k^{i+1}$ are strictly left to any box of $S_{k+1}^{i+1}$. 
	Then we know the $L_{k, i+1} - (m^i+ \Delta^i) + 1, \cdots, L_{k, i+1}$th boxes of $S_k^i$ are strictly left to the $1, \cdots, (m^i + \Delta^i)$th boxes of $S_{k+1}^{i+1}$ respectively, so $m^{i+1} =  \overlap(S_k^{i+1}, S_{k+1}^{i+1})$ would be no less than $m^i+ \Delta^i$.
	
	It remains to show $m^{i+1} \leqslant m^i + \Delta^i = \min\{m^i + L_{k+1, i} - L_{k, i}, L_{k+1, i+1} \}$.
	By the definition of $m^{i+1}$ we know $m^{i+1} \leqslant L_{k+1, i+1}$ immediately.
	Moreover, the $L_{k, i+1} - m^{i+1} +1, \cdots, L_{k, i}$th boxes of $S_k^i$ (same as $S_k^{i+1}$) would be strictly left to the $1, \cdots, (L_{k, i} - L_{k, i+1} + m^{i+1})$th boxes of $S_{k+1}^i$ (same as $S_{k+1}^{i+1}$) respectively,
	so by definition of $m^i$ we have $L_{k, i} - L_{k, i+1} + m^{i+1} \leqslant m^i$.
	Alternatively speaking, $m^{i+1} \leqslant m^i + L_{k+1, i} - L_{k, i}$.
	Now equation \eqref{reccursive-overlap} and the conclusion follows.
\end{proof}

	If $S_k$, $S_{k+1}$ are filled by segments $\nu_k, \nu_{k+1}$, then the above lemma can also be formulated as 
	\[\overlap(S_k, S_{k+1}) = \min \{ \nu_{k; i} - \nu_{k+1; i} \mid i \leqslant k\} + \nu_{k+1; 0} - \nu_{k;k}.\]
	\cite[Definition 7.1]{Trapa} defines $\sing(\nu_k, \nu_{k+1}) : = \#(\nu_k \cap \nu_{k+1})$, so we may deduce following lemma easily.
\begin{lem}\label{overlap-geq-sing}
	$\overlap(S_k, S_{k+1}) \geqslant \sing(\nu_k, \nu_{k+1})$ is equivalent to the following:
	\begin{enumerate}[label=$(\arabic*)$]
	\item if $\nu_k \geqslant \nu_{k+1}$, then $\forall i$, $\nu_{k; i} \geqslant \nu_{k+1; i}$; 
	\item if $\nu_k \subset \nu_{k+1}$, then $\forall i$, $ \nu_{k; i} - \nu_{k; 0} \geqslant \nu_{k+1; i} - \nu_{k+1; 0}$, 
	while $\nu_k \supset \nu_{k+1}$, then $\nu_{k; i} - \nu_{k;k} \geqslant \nu_{k+1; i} - \nu_{k+1; k+1}$.
	\end{enumerate}
\end{lem}
\begin{rmk}
	In case (1), the $\nu$-filling on $S_k \sqcup S_{k+1}$ is already a $\nu$-antitableau (entries weakly decrease across rows and strictly decrease down columns) due to Lemma \ref{anti-by-type}.
\end{rmk}

\begin{eg}
	Recall the $\nu$-antitableau in example \ref{tableau-data}, with $\nu = (\nu_1, \nu_2, \nu_3)$, and
	$\nu_1 = (7, 6, 5)$,
	$\nu_2 = (7, 6, 5, 4, 3)$,
	$\nu_3 = (6, 5, 4, 3, 2, 1)$.
	One may calculate by definition that
	$\overlap(S_1, S_2) = 3 = \sing(\nu_1, \nu_2),$
	$\overlap(S_2, S_3) = 4 = \sing(\nu_2, \nu_3).$
	In this example $\nu_1 \geqslant \nu_2 \geqslant \nu_3$, 
	and it really is an $\nu$-antitableau. 
\end{eg}

\subsection{Understand Trapa's operation}
	In his algorithm converting a $\nu$-quasitableau into a $\nu$-antitableau (or 0), Trapa defines an equivalence relation on 
	the set of \textbf{compatible partitions into skew columns of $\nu$-quasitableau} (and a formal symbol $0$).
	The relation is generated by an operation adjusting partitions and $\nu$-fillings of adjacent skew columns $S_k, S_{k+1}$.
	Here is its definition.
	
\begin{din}[{\cite[Procedure 7.5]{Trapa}}]\label{T-operation}
	Suppose that $S_k$, $S_{k+1}$ are filled by segments $\nu_k$, $\nu_{k+1}$ respectively, and $\nu_{k+1}$ does not precede $\nu_k$. 
	Denote this $\nu$-filling by $R_{\mathrm A}$.
	
	If $\overlap(S_k, S_{k+1}) < \sing(\nu_k, \nu_{k+1})$, then  define $R_{\mathrm A}$ equivalent to $0$ (the formal zero tableau).
	
	If $\overlap(S_k, S_{k+1}) \geqslant \sing(\nu_k, \nu_{k+1})$, define a new $\nu$-antitableau $R_{\mathrm A}'$ on the skew diagram $R:= S_k \sqcup S_{k+1}$ as follows.
	Denote $\nu_k = (\nu_{k, 1}, \cdots, \nu_{k, m_k})$, $\nu_{k+1} = (\nu_{k+1, 1}, \cdots, \nu_{k+1, m_{k+1}})$.
	\begin{enumerate}[fullwidth, itemindent=2em, label=\textbf{Case \arabic*.}]
	\item If $\nu_k> \nu_{k+1}$, then $R_{\mathrm A}' = R_{\mathrm A}$; we do nothing.
	
	\item If $\nu_k \supset \nu_{k+1}$, then define by induction on $\Delta:=\nu_{k+1, m_{k+1}} - \nu_{k, m_k}$.
	For $\Delta =0$, we do nothing.
	Now suppose $\Delta >0$. 
	\begin{enumerate}[label=(\arabic*)]
	\item Define a filling $R_{\mathrm A}(-)$ with shape $R$, $S_k$-part filled by  $\nu_k$, and 
	$S_{k+1}$-part filled by $\nu_{k+1}(-): = (\nu_{k+1, 1} -1, \cdots, \nu_{k+1, m_{k+1}} -1)$.
	\item By induction, adjust $R_{\mathrm A}(-)$ to a new antitableau $[R_{\mathrm A}(-)]'$, with a compatible partition into two skew columns.
	\item Obtain $R_{\mathrm A}'$ by adding one to entries $\nu_{k+1, 1} -1, \cdots, \nu_{k+1, m_{k+1}} -1$ in $[R_{\mathrm A}(-)]'$.
	Each of them appears twice  in $[R_{\mathrm A}(-)]'$, so the problem is to specify which one should be added. Begin by considering the unique entry $\nu_{k+1, 1}$ in $[R_{\mathrm A}(-)]'$. 
	\textbf{Since $[R_{\mathrm A}(-)]'$ is an antitableau}, 
	there is at most one box filled by $\nu_{k+1, 1} -1$ and strictly to the right of $\nu_{k+1, 1}$ in $[R_{\mathrm A}(-)]'$.  
	If such a box exists, then add one to its entry. 
	If no such box exists, then add one to the entry in the left-most box filled by $\nu_{k+1, 1} -1$ in $[R_{\mathrm A}(-)]'$.
	In either case, denote the resulting filling by $[R_{\mathrm A}(-)]'_1$;
	it is also an antitableau.
	Next, construct $[R_{\mathrm A}(-)]'_2$ by the same procedure applied to $[R_{\mathrm A}(-)]'_1$, but instead considering the entries $\nu_{k+1 ,2}$ and $\nu_{k+1, 2} -1$. 
	Continue in this way, and define $R_\anti'=[R_{\mathrm A}(-)]'_{m_{k+1}}$.
	\end{enumerate}

	\item  If $\nu_k \subset \nu_{k+1}$. then define by induction on $\Delta:=\nu_{k+1, 1} - \nu_{k, 1}$.
	For $\Delta =0$, we do nothing.
	Now suppose $\Delta >0$. 
	\begin{enumerate}[label=(\arabic*)]
	\item Define a filling $R_{\mathrm A}(+)$ with shape $R$, $S_k$-part filled by  $\nu_k(+): = (\nu_{k, 1} +1, \cdots, \nu_{k, m_k} +1)$, and $S_{k+1}$-part filled by $\nu_{k+1}$.
	\item By induction, adjust $R_{\mathrm A}(+)$ to a new antitableau $[R_{\mathrm A}(+)]'$, with a compatible partition into two skew columns.
	\item
	Consider the unique entry $\nu_{k, m_k}$ in $[R_{\mathrm A}(+)]'$. 
	\textbf{Since $[R_{\mathrm A}(+)]'$ is an antitableau}, there is at most one box filled by $\nu_{k, m_k} +1$ and strictly to the left of $\nu_{k, m_k}$ in $[R_{\mathrm A}(+)]'$.  
	If such a box exists, then substract one to its entry. 
	If no such box exists, then substract one to the entry in the right-most box filled by $\nu_{k, m_k} +1$ in $[R_{\mathrm A}(+)]'$.
	In either case, denote the resulting antitableau by $[R_{\mathrm A}(+)]'_1$.
	Next, iterate this process as in \textbf{Case 2}, and define $R_\anti '=[R_{\mathrm A}(+)]'_{m_{k}}$.
	\end{enumerate}
	\end{enumerate}
	It remains to give a new partition $R= S_k' \sqcup S_{k+1}'$ compatible with $R_{\mathrm A}'$.
	$S_{k+1}'$ is defined as follows.
	Its last box is the right most one containing $\nu_{k+1, m_{k+1}' }':= \min \{ \nu_{k, m_k}, \nu_{k+1, m_{k+1} } \}$.
	Its next to last box is the right most one containing $\nu_{k+1, m_{k+1}'}'+1$.
	Continue in this way until we reach the right most box containing $\min \{ \nu_{k, 1}, \nu_{k+1, 1}\}$.
	$S_k'$ is defned to be what remains. 
	To see its well-defineness, the readers are encouraged to check an example.
\end{din}
\begin{rmk}
	The operation in \textbf{Case 2} and \textbf{Case 3} is very similiar.
	Informally, they are related by  an automorphism of $\g$ coming from the Dykin diagram.
\end{rmk}

\begin{eg}
	Recall the $\nu$-antitableau in example \ref{tableau-data}, with $\nu = (\nu_1, \nu_2, \nu_3)$, and
	$\nu_1 = (7, 6, 5)$,
	$\nu_2 = (6, 5, 4, 3, 2)$,
	$\nu_3 = (6, 5, 4, 3, 2, 1)$.
	One may calculate by definition that
	$\overlap(S_2, S_3) = 4 < 5 = \sing(\nu_2, \nu_3),$
	so it would become $0$ under Trapa's operation.
	Consequently, for $A$-parameter $\psi$ of $\U(6, 8)$ with base change
	\[\psi_\C = (\frac z{\bar z})^{6} \boxtimes \FDR_3 \oplus 
	(\frac z{\bar z})^4 \boxtimes \FDR_5 \oplus 
	(\frac z{\bar z})^{\frac 72} \boxtimes \FDR_6:
	W_\C \times \SL(2, \C) \to \GL(14, \C),\]
	the $A_\q(\lambda)$ given by $\underline p = (2, 2, 2) \in \MR$ vanishes.
\end{eg}

	Trapa's definition is descriptive. 
	In the rest of this subsection, we provide a formula in terms of types $\{\nu_{k; i}\}$, $\{\nu_{k+1; i}\}$, and $\{\nu_{k; i}'\}$, $\{\nu_{k+1; i}'\}$.
	
\subsubsection{Case $\nu_k \supset \nu_{k+1}$}
	Denote $\Delta = \nu_{k+1, m_{k+1}} - \nu_{k, m_k}$, the same as in Definition \ref{T-operation}, 
	and define a decreasing sequence
	\[\Delta_i = \min \{ \nu_{k; j} - \nu_{k+1; j} \mid j < i\}.\]
	Of course we may assume $\overlap(S_k, S_{k+1}) = \sing(\nu_k, \nu_{k+1}) = m_{k+1}$.
	According to the inequality $\nu_{k;i} - \nu_{k; k} \geqslant \nu_{k+1; i} - \nu_{k+1; k+1}$ in Lemma \ref{overlap-geq-sing}, we know $\nu_{k+1; k} = \nu_{k+1; k+1}$ (by taking $i=k$),
	and moreover, $\Delta_i$ ends with $\Delta_{k+1} = -\Delta$. 
\begin{lem}
	In the case $\nu_k \supset \nu_{k+1}$,
\begin{subequations}\label{operation-sup}
	\begin{align}
	\nu_{k+1; i}' &= \nu_{k+1; i} + \min\{ 0, \Delta_i \}, 
	\label{eq: operation-sup-k+1}\\
	\nu_{k; i}' &= \nu_{k; i} - \min\{ 0, \Delta_{i+1} \}.
	\label{eq: operation-sup-k}
	\end{align}
\end{subequations}
	Consequently, $\nu_{k; i}' \geqslant \nu_{k+1; i} \geqslant \nu_{k+1; i}'$, and the resulted filling is an antitableau.
\end{lem}
\begin{proof}
	Prove by induction on $\Delta$.
	If $\Delta = 0$, then all $\Delta_i \geqslant -\Delta =0$, and the formula \eqref{operation-sup} says $\nu_{k+1; i}' = \nu_{k+1; i}$, $\nu_{k; i}' = \nu_{k; i}$.
	This is exactly our definition.
	
	Now suppose the formula for case $\Delta \geqslant 0$ has been verified.
	Denote the resulted $\nu$-filling in this case by $R_{\mathrm A}' $, and the partition by $R = S_k' \sqcup S_{k+1}'$.
	Then we want to reach the case $\Delta(+)$, that is, to obtain the filling $R_{\mathrm A}(+)'$,
	which is resulted via process in Definition \ref{T-operation}(2), and beginning with $R = S_k \sqcup S_{k+1}$ filled by $\nu_k$ and $\nu_{k+1}(+) :=  (\nu_{k+1, 1} +1, \cdots, \nu_{k+1, m_{k+1}} +1)$.
	Under the filling $R_{\mathrm A}'$, there are exactly two boxes of $S$ filled by $\nu_{k+1, i}$, lying on different columns.
	We have to locate them and decide which should have entry added by one.
	
	Take $i'$ to be the smallest $i\leqslant k$ such that the box in $S_k'$ filling with $\nu_{k+1; i}'$ is not strictly left to the $(i+1)$th component of $S_{k+1}$. 
	Then for each $1 \leqslant i \leqslant i'$, each box in $S_k'$ with entry $\nu_{k+1; i-1}' , \cdots, \nu_{k+1; i}' +1$ is strictly left to the box in $S_{k+1}'$ with entry $\nu_{k+1; i-1}' -1, \cdots, \nu_{k+1; i}'$ respectively.
	Consequently, the box with entry $\nu_{k+1;0}'-1, \cdots, \nu_{k+1; i}',$ and to be added by one, is the one on the right; 
	it lies in $S_{k+1}'$.
	After that, the box with entry $\nu_{k+1;i'}'$ (in $S_k'$) will be not strictly left to the box in $S_{k+1}'$ with entry $\nu_{k+1;i'}' -1$.
	Hence, neither of the two boxes in $S$ with entry $\nu_{k+1;i'}' -1$ is strictly right to the box with entry $\nu_{k+1;i'}'$, 
	so the left one should have entry  $\nu_{k+1;i'}' -1$ trun into $\nu_{k+1;i'}'$; it lies in $S_k'$.
	Then the box with entry $\nu_{k+1;i'}'-1$ (now lying in $(i'+1)$ th column of $S_{k+1}')$ faces a similar situation as above.
	We may repeat the same argument, and conclude that the box with entry $\nu_{k+1; i'}'-1, \cdots, \nu_{k; k}'$ and to be added by one, is the one on the left (and lies in $S_k'$).
	
	Let's say something about the index $i'$.
	By definition, it is the smallest $i\leqslant k$ such that 
	the box in $S_k'$ filling with $\nu_{k+1; i}'$ belongs to the first $i$ components of $S_k'$.
	Then it is the first $i \geqslant 0$ such that $\nu_{k+1; i}' \geqslant \nu_{k;i}'$, 
	and equivalently
	\[\nu_{k+1; i} + \min \{0, \Delta_i\} \geqslant \nu_{k; i} - \min \{ 0, \Delta_{i+1} \}\]
	according to inductive hypothesis.
	This is exactly the $i_0$ such that $\Delta_{i_0} > 0 \geqslant \Delta_{i_0+1}$.
	Indeed, for $i \leqslant i_0-1$ we have $\nu_{k; i} - \nu_{k+1; i} \geqslant \Delta_{i+1} > 0 = \min \{0, \Delta_{i+1}\} + \min \{0, \Delta_i\}$,
	while for $i=i_0$ we know $\nu_{k; i} - \nu_{k+1; i} = \Delta_{i+1} \leqslant  \min \{0, \Delta_{i+1}\} + \min \{0, \Delta_i\}$.
	A small corollary is that, $i'= i_0$ is the first $i$ such that $\nu_{k+1; i} \geqslant \nu_{k; i}$.
	
	Based on this observation, we may describe $R_{\mathrm A}(+)'$  (compared to $R_{\mathrm A}'$) and its compatible partition $R= S(+)_k' \sqcup S(+)_{k+1}'$ as follows.
	Denote the filling on $S(+)_k'$, $S(+)_{k+1}'$ by $\{\nu(+)_{k; i}'\}$, $\{\nu(+)_{k+1; i}'\}$ respectively.
\begin{itemize}
	\item On the last $i'$ columns of $R$, which consists of the first $i'-1$ components of $S_k'$ and first $i'$ components of $S_{k+1}'$, 
	each box of $S_{k+1}'$ has entry added by one, 
	while each box of $S_k'$ has entry preserved.
	Consequently, for $i \leqslant i'$, 
	$\nu(+)_{k; i-1}' = \nu_{k; i-1}'$, and
	$\nu(+)_{k+1; i}' = \nu_{k+1; i}' +1$.
	\item On the last $(i'+1)$th column of $R$, which consists of the $i'$th component of $S_k'$ and $(i'+1)$th component of $S_{k+1}'$, 
	the entries (in $R_{\mathrm A}'$) from top to the end are $\nu_{k; i'-1}' -1, \cdots, \nu_{k; i'}'$, $\nu_{k+1; i'}' -1, \cdots, \nu_{k+1; i'+1}'$.
	Since $\Delta_{i'} > 0 \geqslant \Delta_{i' +1} = \nu_{k; i'} - \nu_{k+1; i'}$, we have $\nu_{k; i'}' = \nu_{k+1; i'} = \nu_{k+1; i'}'$ (by inductive hypothesis).
	All these entries are preserved in $R_{\mathrm A}(+)'$, but the box with entry $\nu_{k; i'}' =  \nu_{k+1; i'}'$ will be put into $S(+)_{k+1}'$,
	since it is on the right of two boxes with entry $\nu_{k+1; i'}'$;
	the left box with entry $\nu_{k+1; i'}'$ in $R_\anti(+)'$ is the left one with entry $\nu_{k+1; i'}'-1$ in $R_\anti'$, 
	and left to the entry $\nu_{k+1; i'}'$ in $R_\anti'$ which is preserved when go to $R_\anti(+)'$.
	Consequently, for $i = i'+1$, 
	$\nu(+)_{k; i-1}' = \nu_{k; i-1}' +1$, and
	$\nu(+)_{k+1; i}' = \nu_{k+1; i}'$.
	\item On the $i$th column of $R$ where $i> i'+1$, the boxes from top to someone middle have entries added by one (in $R_{\mathrm A}(+)'$), and the rest have entry preserved.
	Moreover, the former part are put into $S(+)_k'$,
	since they are the left of two boxes having the same entry (both in $R_{\mathrm A}'$ and $R_{\mathrm A}(+)'$),
	while the latter part are put into $S(+)_{k+1}'$ respectively.
	Then for $i> i'+1$ we know
	$\nu(+)_{k; i-1}' = \nu_{k; i-1}' +1$, and
	$\nu(+)_{k+1; i}' = \nu_{k+1; i}'$.
\end{itemize}
	
	Summary up, 
	\[\nu(+)_{k; i}' = 
	\left\{\begin{aligned}
	&\nu_{k;i}' = \nu_{k;i} - \min\{0, \Delta_{i+1}\},	&i< i',\\
	&\nu_{k;i}'+1 = \nu_{k;i} - \min\{-1, \Delta_{i+1}-1\},	&i \geqslant i',
	\end{aligned}\right. \]
	and
	\[\nu(+)_{k+1; i}' = 
	\left\{\begin{aligned}
	&\nu_{k+1;i}'+1 = \nu_{k+1;i} + 1+ \min\{0, \Delta_i\},	&i\leqslant i',\\
	&\nu_{k+1;i}' = \nu_{k+1;i} + 1+\min\{-1, \Delta_{i+1}-1\},	&i > i'.
	\end{aligned}\right.\]
	One may chek
	\[\begin{aligned}
	\nu(+)_{k; i}' &= \nu_{k; i} - \min\{ 0, \Delta(+)_{i+1} \},\\
	\nu(+)_{k+1; i}' &= (\nu_{k+1; i} + 1)+ \min\{ 0, \Delta(+)_i \}
	\end{aligned}\]
	easily by property $\Delta_{i'} > 0 \geqslant \Delta_{i'+1}$, and $\Delta(+)_{i} = \Delta_i -1$.
	We obtained the case $\Delta(+)$ from case $\Delta$.
\end{proof}

\subsubsection{Case $\nu_k \subset \nu_{k+1}$}
	Denote $\Delta = \nu_{k+1, 1} - \nu_{k, 1}$, and define an increasing sequence
	\[\Delta_i = \min \{ \nu_{k; j} - \nu_{k+1; j} \mid j \geqslant i\}.\]
	Of course we may assume $\overlap(S_k, S_{k+1}) = \sing(\nu_k, \nu_{k+1}) = m_k$.
	According to the inequality $\nu_{k;i} - \nu_{k; 0} \geqslant \nu_{k+1; i} - \nu_{k+1; 0}$ in Lemma \ref{overlap-geq-sing}, we know $\Delta_i$ starts with $\Delta_0 = \Delta$. 
\begin{lem}
	In the case $\nu_k \subset \nu_{k+1}$,
\begin{subequations}\label{operation-sub}
	\begin{align}
	\nu_{k+1; i}' &= \nu_{k+1; i} + \min\{ 0, \Delta_i \}, 
	\label{eq: operation-sub-k+1}\\
	\nu_{k; i}' &= \nu_{k; i} - \min\{ 0, \Delta_{i+1} \}.
	\label{eq: operation-sub-k}
	\end{align}
\end{subequations}
	Consequently, $\nu_{k; i}' \geqslant \nu_{k;i} \geqslant \nu_{k+1; i}'$, and the resulted filling is an antitableau.
\end{lem}
	The proof is similiar to \eqref{operation-sup}, so we omit it.

\subsection{Remark on Trapa's algorithm}\label{subtle}
	Trapa described the algorithm to find distinguished antitableau after Theorem 7.9 in \cite{Trapa}. 
	But there is something subtle starting from his lemma 7.7.
	Let's explore it.
	
\begin{din}
	Let $(\nu, \mu)$ be an admissible arrangement of two segements.
	Define $(\nu', \mu')$ by
	\[\begin{aligned}
	&\bg(\nu') = \max\{\bg(\nu), \bg(\mu)\},
	&\ed(\nu') = \max\{ \ed(\nu), \ed(\mu)\},\\
	&\bg(\mu')= \min\{\bg(\nu), \bg(\mu)\},
	&\ed(\mu')= \min\{\ed(\nu), \ed(\mu)\},
	\end{aligned}\]
	and call it obtained from $(\nu, \mu)$  by an \textbf{elementary operation}.
\end{din}
\begin{rmk}	
	Let $R= S_k \sqcup S_{k+1}$ filled by $\nu_k, \nu_{k+1}$. 
	Apply Trapa's operation and obtain $R = S_k' \sqcup S_{k+1}'$ filled by $\nu_k', \nu_{k+1}'$.
	Then $(\nu_k', \nu_{k+1}')$ is obtained from $(\nu_k, \nu_{k+1})$ by an elementary operation.
	Of course $\nu_k' \geqslant \nu_{k+1}'$.
\end{rmk}

	On the level of arranged segments, lemma 7.7 in \cite{Trapa} argues that if $\nu= (\nu_1, \cdots, \nu_r)$ is in the ``mediocre range'' (our terminology is ``admissible''), and apply elmentary operation to some $(\nu_k, \nu_{k+1})$ in containment relation, then the resulting $\nu' = (\nu_1, \cdots, \nu_k', \nu_{k+1}', \cdots, \nu_r)$ is also ``mediocre''.
	However, there are counterexamples to it.
	
\begin{eg}
	Consider $\nu= (\nu_1, \nu_2, \nu_3)$ with $\nu_2 \supset \nu_1 \supset \nu_3$.
	If we apply elementary operation to $(\nu_2, \nu_3)$ and obtain $(\nu_2', \nu_3')$, then $(\nu_1, \nu_2', \nu_3')$ is no longer admissible, because $\nu_2'$ precedes $\nu_1$.
\end{eg}

\begin{eg}
	One might guess that by applying elementary operations to properly chosen adjacent segments, 
	we could avoid exceeding the ``mediocre range'', and eventually reach the `` nice range''.
	However, here is a counterexample to this assumption.
	Consider $\nu = (\nu_1, \nu_2, \nu_3, \nu_4, \nu_5)$ with $\nu_4 \supset \nu_2 \supset \nu_5 \supset \nu_3 \supset \nu_1$. 
	We would go outside mediocre range either apply elementary operation to $(\nu_1, \nu_2)$, $(\nu_2, \nu_3)$, $(\nu_3, \nu_4)$, or $(\nu_4, \nu_5)$.
	
	We illustrate this example as follows, where the number attached to a segment indicates its order in the admissible arrangement $\nu$.
	\[\begin{tikzpicture}[scale=0.5]
	\begin{scope}
	\draw [-stealth](6.5,0)--(-5.5,0);
	\foreach \i in {1, 2, 3, 4, 5}{
		\draw (\i, 0) arc (0:180: \i -0.5);
		\filldraw (\i,0) circle (.1);
		\draw[fill=white] (-\i+1,0) circle (.1);
	}
	\node at (1,-0.5) {1};
	\node at (2,-0.5) {3};
	\node at (3,-0.5) {5};
	\node at (4,-0.5) {2};
	\node at (5,-0.5) {4};
	\end{scope}.
	\end{tikzpicture}\]
\end{eg}

	A solution to this problem comes from the observation that $\nu^\sigma : = (\nu_1, \nu_3, \nu_2, \nu_4, \nu_5)$ and $\nu$ turn into the same arranged segments by an elementary operation to $\nu_2, \nu_3$.
	They should be equivalent under the relation generated by elementary operations (and their inverse).
	The adavantage of $\nu^\sigma$ is that iterated elementary operations to it can stay in mediocre range, and end in nice range.
\usetikzlibrary {arrows.meta}
	\[\begin{tikzpicture}[scale=0.4]
	\begin{scope}
	\draw [-stealth](6.5,0)--(-5.5,0);
	\draw [-stealth](6.5,0)--(-5.5,0);
	\draw (1,0) arc (0:180:0.5);
	\draw (2,0) arc (0:180:1.5);
	\draw (3,0) arc (0:180:2.5);
	\draw (4,0) arc (0:180:3.5);
	\draw (5,0) arc (0:180:4.5);
	\foreach \i in {1, 2, 3, 4, 5}{
		\filldraw (\i,0) circle (.1);
		\filldraw[fill=white] (-\i+1,0) circle (.1);
	}
	\node at (1,-0.5) {1};
	\node at (2,-0.5) {3};
	\node at (3,-0.5) {5};
	\node at (4,-0.5) {2};
	\node at (5,-0.5) {4};
	\node at (-5, 2) {$\nu$};
	\end{scope}
	\draw [line width=0.5pt, double distance=1pt, arrows = {-Latex[length=6pt]} ] (6.5,2) -- (8,2);
	\node at(7.25, 3) {(2, 3)};
	\begin{scope}[xshift=400]
	\draw [-stealth](6.5,0)--(-5.5,0);
	\draw (1,0) arc (0:180:0.5);
	\draw [red](2,0) arc (0:180:2.5);
	\draw (3,0) arc (0:180:2.5);
	\draw [red](4,0) arc (0:180:2.5);
	\draw (5,0) arc (0:180:4.5);
	\foreach \i in {1, 2, 3, 4, 5}{
		\filldraw (\i,0) circle (.1);
		\filldraw[fill=white] (-\i+1,0) circle (.1);
	}
	\node at (1,-0.5) {1};
	\node at (2,-0.5) {2};
	\node at (3,-0.5) {5};
	\node at (4,-0.5) {3};
	\node at (5,-0.5) {4};
	\end{scope}
	\draw [line width=0.5pt, double distance=1pt, arrows = {-Latex[length=6pt]} ] (14.5,-3) -- (14.5,-1.5);
	\node at (16, -2.25) {(2, 3)};
	\begin{scope}[xshift=400, yshift=-250]
	\draw [-stealth](6.5,0)--(-5.5,0);
	\draw (1,0) arc (0:180:0.5);
	\draw (2,0) arc (0:180:1.5);
	\draw (3,0) arc (0:180:2.5);
	\draw (4,0) arc (0:180:3.5);
	\draw (5,0) arc (0:180:4.5);
	\foreach \i in {1, 2, 3, 4, 5}{
		\filldraw (\i,0) circle (.1);
		\filldraw[fill=white] (-\i+1,0) circle (.1);
	}
	\node at (1,-0.5) {1};
	\node at (2,-0.5) {2};
	\node at (3,-0.5) {5};
	\node at (4,-0.5) {3};
	\node at (5,-0.5) {4};
	\node at (6, 2) {$\nu^\sigma$};
	\end{scope}
	\draw [line width=0.5pt, double distance=1pt, arrows = {-Latex[length=6pt]} ] (8,-7) -- (6.5,-7);
	\node at (7.25, -6) {(3, 4)};
	\begin{scope}[yshift=-250]
	\draw [-stealth](6.5,0)--(-5.5,0);
	\draw (1,0) arc (0:180:0.5);
	\draw (2,0) arc (0:180:1.5);
	\draw (3,0) arc (0:180:2.5);
	\draw [red](4,0) arc (0:180:4);
	\draw [red](5,0) arc (0:180:4);
	\foreach \i in {1, 2, 3, 4, 5}{
		\filldraw (\i,0) circle (.1);
		\filldraw[fill=white] (-\i+1,0) circle (.1);
	}
	\node at (1,-0.5) {1};
	\node at (2,-0.5) {2};
	\node at (3,-0.5) {5};
	\node at (4,-0.5) {3};
	\node at (5,-0.5) {4};
	\end{scope}
	\draw [line width=0.5pt, double distance=1pt, arrows = {-Latex[length=6pt]} ] (0,-10.5) -- (0,-12);
	\node at (1.5, -11.25) {(4,5)};
	\begin{scope}[yshift=-500]
	\draw [-stealth](6.5,0)--(-5.5,0);
	\draw (1,0) arc (0:180:0.5);
	\draw (2,0) arc (0:180:1.5);
	\draw [red](3,0) arc (0:180:3);
	\draw (4,0) arc (0:180:4);
	\draw [red](5,0) arc (0:180:3.5);
	\foreach \i in {1, 2, 3, 4, 5}{
		\filldraw (\i,0) circle (.1);
		\filldraw[fill=white] (-\i+1,0) circle (.1);
	}
	\node at (1,-0.5) {1};
	\node at (2,-0.5) {2};
	\node at (3,-0.5) {4};
	\node at (4,-0.5) {3};
	\node at (5,-0.5) {5};
	\end{scope}
	\draw [line width=0.5pt, double distance=1pt, arrows = {-Latex[length=6pt]} ] (6.5,-16) -- (8,-16);
	\node at (7.25, -15) {(3, 4)};
	\begin{scope}[xshift=400, yshift=-500]
	\draw [-stealth](6.5,0)--(-5.5,0);
	\draw (1,0) arc (0:180:0.5);
	\draw (2,0) arc (0:180:1.5);
	\draw [red](3,0) arc (0:180:3.5);
	\draw [red](4,0) arc (0:180:3.5);
	\draw (5,0) arc (0:180:3.5);
	\foreach \i in {1, 2, 3, 4, 5}{
		\filldraw (\i,0) circle (.1);
		\filldraw[fill=white] (-\i+1,0) circle (.1);
	}
	\node at (1,-0.5) {1};
	\node at (2,-0.5) {2};
	\node at (3,-0.5) {3};
	\node at (4,-0.5) {4};
	\node at (5,-0.5) {5};
	\end{scope}
	\draw [line width=0.5pt, double distance=1pt, arrows = {-Latex[length=6pt]} ] (21,-16) -- (22.5,-16);
	\node at (25, -16) {go ahead};
	\end{tikzpicture}\]

	Now let's go back to the level of tableaus.
\begin{prop}
	If $\underline p \in \mathcal D(\psi)$ and $\underline p^\sigma \in \mathcal D(\psi^\sigma)$ are related (by $\phi^\sigma: \MR \to \MR^\sigma$), then  they give rise to equivalent $\nu$-quasitableaus.
\end{prop}
	This Proposition will be proved in the next section.
	Then the algorithm on the $A_\q(\lambda)$ parametrized by $\underline p \in \MR$ may be described as follows:
	\begin{enumerate}[fullwidth, itemindent=2em, label=\textbf{Step \arabic*.}]
	\item after replacing $\underline p \in \MR$ by  $\underline p^\sigma$, we may assume the arrangement $\nu = (\nu_1, \cdots, \nu_r)$ is ``appropriate'', 
	that $\forall i<j$, either $\nu_i > \nu_j$, or $\nu_i \subset \nu_j$;
	\item then apply Trapa's operation iterately to the $\nu$-quasitableau $S_\anti$ arsing from  $\underline p$, and we would end up with either a $\nu$-antitableau $S_\anti'$ or $0$.
\end{enumerate}
\begin{rmk}
	By an ``appropriate'' arrangement we mean \textbf{Step 2} is appliable to it.
	Thus, the condition ``$\forall i<j$, either $\nu_i > \nu_j$, or $\nu_i \subset \nu_j$'' is not necessary for ``appropriate''.
	For example, we can replace $\nu_i \subset \nu_j$ by $\nu_i \supset \nu_j$.
	We will also see other examples of  ``appropriate'' condition in section \ref{sufficiency}.
\end{rmk}

%% file: Proof-1.tex
\section{Transition on the tableaus}\label{swap-tableau}
	In this section we will verify the validity of $\phi_\sigma^\tau$:
	if $\underline p \in \MR$ parametrize a non-zero representation $A_\q(\lambda)$, then 
	$\underline p^\sigma \in \MR^\sigma$ gives rise to the same cohomologically induced module.
	We have present a representation-theoretic approach in subsubsection \ref{easy-approach},
	and the following is a combinatorial approach, that shows
	$A_{\q_{\underline p^\sigma}}(\lambda_{\underline p^\sigma})$ and
	$A_{\q_{\underline p^\tau}}(\lambda_{\underline p^\tau})$ have the same primitive ideal ($\nu$-antitableau) and associated variety (signed Young tableau).


	 It suffices to deal with the basic case: $\tau= \Id$,
	$\sigma= (k, k+1)$ is a transposition, $\underline p$ satisfies $\mathrm B(k), \mathrm B(k+1)$ and $\mathrm C(k, k+1)$
	(and hence lies in $\mathrm C^\sigma(k, k+1)$; see Lemma \ref{B-independent}).
	The argument in subsubsection \ref{easy-approach} is concise, but 
	the combinatorial proof up to subsection \ref{subsec: hard}, especially Lemma \ref{swap-for-nu-fill}, is crucial to our main theorem;
	therefore, we will present it in detail.
	
	Denote the Young diagram partition arising from $\underline p^\sigma =(p_1^\sigma, \cdots, p_r^\sigma)$ by $S^\sigma=  \bigsqcup_{i=1}^r S_i^\sigma$. 
	Since $\underline p^\sigma$ has the same first $(k-1)$ components with $\underline p$, $S^{\sigma, k-1}:=\bigsqcup_{i=1}^{k-1} S_i^\sigma$ has the same type with $S^{k-1}:=\bigsqcup_{i=1}^{k-1} S_i$.
	Moreover,  the signed Young tableaus on $S^{\sigma, k-1}$ and  $S^{k-1}$ are equivalent, and the $\nu$-antitableaus on them  coincide. 
	The key is the following proposition:
\begin{prop}\label{swap}
	$S_k^\sigma \sqcup S_{k+1}^\sigma $ has the same type with $ S_k\sqcup S_{k+1}$. Moreover, the signed Young tableaus and the $\nu$-quasitableaus on them are also equivalent.
\end{prop}

	Based on it, our conclusion follows easily.
	Since the last $(r-k-1)$ components of $\underline p^\sigma$ are also same with $\underline p$, they give rise to equivalent signed Young tableaus and $\nu$-antitableaus once they are identified on the partial diagram $S^{k+1}= S^{\sigma, k+1}$.

\subsection{Core lemmas}\label{subsec: framework}
	Here we give two combinatorial lemmas for Proposition \ref{swap}.
	\textbf{For convenience, we may assume $m_k \leqslant m_{k+1}$ in the rest of this section.}
	
	The skew diagram $S_k\sqcup S_{k+1}$ can be understood by $L_{k+1, i} + L_{k, i-1}$, $1 \leqslant i \leqslant k+1$, which means length of the last $i$th column of it. 
	Then we need $\forall 1\leqslant i \leqslant k+1,$
\begin{equation}\label{type}
	L^\sigma_{k+1, i} + L^\sigma_{k, i-1} \stackrel?= L_{k+1, i} + L_{k, i-1},
\end{equation}
	to obtain $S_k^\sigma \sqcup S_{k+1}^\sigma = S_k \sqcup S_{k+1}$ 
	
\begin{lem}\label{swap-for-sign-fill}
	Denote the type of sign-filling in $S_k^\sigma$ by$L^{\sigma, \pm}_{k,i}$, then
	$L^{\sigma, \pm}_{k+1, i}+ L^{\sigma, \pm}_{k, i-1} = L^{\sigma, \pm}_{k+1, i} + L^{\sigma, \pm}_{k, i-1},$ $\forall i=1, \cdots, k+1$.
\end{lem}
	From this lemma we may deduce $S_k^\sigma \sqcup S_{k+1}^\sigma = S_k \sqcup S_{k+1}$ and conclude further they have equivalent sign-fillings.
		
	To show the $\nu$-fillings on $S_k^\sigma \sqcup S_{k+1}^\sigma = S_k \sqcup S_{k+1}$ are equivalent, it suffices to prove following lemma according to formulas \eqref{operation-sup} and \eqref{operation-sub}.
	
\begin{lem}\label{swap-for-nu-fill}
$\forall 1 \leqslant i \leqslant k,$
\begin{subequations}
	\begin{align}
	\label{fulfill-1}
	L^\sigma_{k+1,i}&= L_{k+1, i} - \min\{ L_{k+1, j} - L_{k,j} \mid j \geqslant i\},	\\
	\label{fulfill-2}
	L_{k+1, i} &= L_{k+i, i}^\sigma - \min\{ L_{k+1, j}^\sigma - L_{k, j}^\sigma \mid j <i\}.
	\end{align}
\end{subequations}
\end{lem}

	Once the lemmas \ref{swap-for-sign-fill} and \ref{swap-for-nu-fill} are verified, Proposition \ref{swap} naturally  follows.
	The conclusion of Proposition \ref{swap} regarding type and signed Young tableau is equivalent to Lemma \ref{swap-for-sign-fill}.
	For its conclusion on $\nu$-quasitableau, let $S_k' \sqcup S_{k+1}'$ be the partition obtained from $S_k \sqcup S_{k+1}$ by Trapa's algorithm, with type $\nu_{k; i}'$, $\nu_{k+1; i}'$;
	let $S_k^{\sigma, \prime} \sqcup S_{k+1}^{\sigma, \prime}$,$\nu_{k; i}^{\sigma, \prime}$, $\nu_{k+1; i}^{\sigma, \prime}$ be obtained from $S_k^\sigma  \sqcup S_{k+1}^\sigma$;
	we have to verify $\nu_{k; i}' \stackrel?= \nu_{k; i}^{\sigma, \prime}$, 
	$\nu_{k+1; i}' \stackrel?= \nu_{k+1; i}^{\sigma, \prime}$.
	Recall our assumption $\nu_k \subset \nu_{k+1}$, and denote $\Delta= \nu_{k+1; 0} - \nu_{k; 0}$.
	According to \eqref{operation-sub} and \eqref{fulfill-1},
	\[\begin{aligned}
	\nu_{k+1; i}'=	&
	\nu_{k+1; i} + \min\{0, \nu_{k;j } - \nu_{k+1; j} \mid j \geqslant i\}	\\
	=	&	
	\nu_{k+1; i} + \min\{0, -\Delta- L_{k, j} + L_{k+1, j} \mid j \geqslant i\} \\
	=	&
	\nu_{k; 0} - L_{k+1, i} + \min\{ \Delta, L_{k+1, j} - L_{k, j} \mid j \geqslant i\}	\\
	=	&
	\nu_{k; 0}+ \min\{\Delta - L_{k+1,i}, - L_{k+1, i}^\sigma \};
	\end{aligned}\]
	according to  \eqref{operation-sup} and \eqref{fulfill-2},
	\[\begin{aligned}
	\nu_{k+1; i}^{\sigma, \prime}=	&
	\nu_{k+1; i}^\sigma + \min\{0, \nu_{k;j }^\sigma - \nu_{k+1; j}^\sigma \mid j <i\}	\\
	=	&	
	\nu_{k+1; i}^\sigma + \min\{0, \Delta- L_{k, j}^\sigma + L_{k+1, j}^\sigma \mid j < i\} \\
	=	&
	\nu_{k+1; 0}^\sigma - L_{k+1, i}^\sigma + \min\{0,  \Delta+ L_{k+1, j} - L_{k, j} \mid j < i\}	\\
	=	&
	\nu_{k; 0}+ \min\{ - L_{k+1, i}^\sigma, \Delta - L_{k+1, i} \};
	\end{aligned}\]
	It follows that $\nu_{k+1; i}' = \nu_{k+1; i}^{\sigma, \prime}$.
	Then $\nu_{k; i}' = \nu_{k; i}^{\sigma, \prime}$ automatically, because
	the $\nu$-antitableau on $S_k' \sqcup S_{k+1}'$ is determined by its $S_{k+1}'$-part, and 
	so is $S_k^{\sigma, \prime} \sqcup S_{k+1}^{\sigma, \prime}$.

\subsection{Calculation of type}\label{type-for-p}
	To prove the core lemmas, we need calculate $L_{k, i}^{\sigma, \pm}, L_{k+1, i}^{\sigma, \pm}$ in detail.

	The last entry of each row of $S^{k-1} = S^{\sigma, k-1}$ make up into a skew column.
	Denote it by $T$, and call it the extreme column of $S^{k-1}= S^{\sigma, k-1}$. 
	Denote by $P_i, Q_i$ the number of $+, -$ filling into first $i$ components of $T$. 	
	Then for the sign fillings of $S_k$, we have
\begin{equation}\label{type-of-S_k}
	L_{k, i}^+=\min\{Q_i, p_k\}, \quad 
	L_{k, i}^- = \min \{P_i, q_k\}.
\end{equation}
	Moreover, the sign filling of extreme column of $S^k = S^{k-1} \sqcup S_k$ has type
\begin{equation}\label{type-of-T_k}
	P_{k,i}= L_{k,i}^+ + P_{i-1} - L_{k, i-1}^-, \quad
	Q_{k,i} = L_{k,i}^- + Q_{i-1} - L_{k, i-1}^+.
\end{equation}
	Hence, the (type of) sign filling of $S_{k+1}$ would be given by
\begin{equation}\label{type-of-S_k+1}
	L_{k+1, i}^+ = \min \{Q_{k, i}, p_{k+1} \}, \quad
	L_{k+1, i}^- = \min \{P_{k, i}, q_{k+1} \}.
\end{equation}
	Consequently, we can present	
	
\begin{proof}[Proof of Lemma \ref{swap-for-sign-fill}]
	According to above equations,
	\[\begin{aligned}
	L_{k+1, i}^+ + L_{k, i-1}^+ =&
	\min \{Q_{i-1} + L_{k,i}^- , p_{k+1} + L_{k, i-1}^+ \} \\
	= &
	\min\{Q_{i-1}+P_i, Q_{i-1}+q_k, p_{k+1} + Q_{i-1}, p_{k+1} + p_k \} \\
	= &
	\min\{Q_{i-1}+P_i, Q_{i-1}+q_k, p_{k+1} + p_k \},
	\end{aligned}\]
	where we have used the condition that $q_k \leqslant p_{k+1}$ following from $\underline p \in \mathrm C(k, k+1)$.
	
	Following a similar calculation,
	\[\begin{aligned}
	L_{k+1, i}^{\sigma, +} + L_{k, i-1}^{\sigma, +} =&
	\min\{Q_{i-1}+P_i, Q_{i-1}+q_k^\sigma, p_{k+1}^\sigma + Q_{i-1}, p_{k+1}^\sigma + p_k^\sigma \} \\
	= &
	\min\{Q_{i-1}+P_i, p_{k+1}^\sigma + Q_{i-1}, p_{k+1}^\sigma + p_k^\sigma \},
	\end{aligned}\]
	where we have used the condition that $p_{k+1}^\sigma \leqslant q_k^\sigma$ following from $\underline p \in \mathrm C^\sigma(k, k+1)$.
	Then
	\[L^{\sigma, +}_{k+1, i}+ L^{\sigma, +}_{k, i-1} = L^{\sigma, +}_{k+1, i} + L^{\sigma, +}_{k, i-1},\]
	since $p_{k+1}^\sigma = q_k$ and $p_{k+1}^\sigma + p_k^\sigma = p_{k+1} + p_k$, according to the definition of $\underline p ^\sigma = \phi^\sigma( \underline p)$.
	The proof for 
	$L^{\sigma, -}_{k+1, i}+ L^{\sigma, -}_{k, i-1} = L^{\sigma, -}_{k+1, i} + L^{\sigma, -}_{k, i-1}$
	is similar.
\end{proof}

	For Lemma \ref{swap-for-nu-fill}, we have to study the monotonicity of difference $L_{k+1, i}- L_{k, i}$ and $L_{k, i}^\sigma - L_{k+1,i }^\sigma$, based on more conceret expressions of them.
	Denote 
	\[ i^- = \min \{ i \mid P_i \geqslant q_k\}, \quad
	i^+ = \min\{ i \mid Q_i \geqslant p_k \},\]
	and we may assume $i^- \leqslant i^+$ without loss of generality.
	Then $S_k$ would break after its $i^-$th component, and end at its $i^+$th component; 
	this is also implied by the following calculation.
	According to \eqref{type-of-S_k},
	\[L_{k, i}^+= \left\{ \begin{aligned}
	&Q_i,	& i < i^+,\\
	&p_k,	&i \geqslant i^+,
	\end{aligned}\right.
	\quad
	L_{k, i}^-= \left\{ \begin{aligned}
	&P_i,	& i < i^-,\\
	&q_k,	&i \geqslant i^-,
	\end{aligned}\right.\]
	and hence
\begin{equation}\label{detail-type-of-S_k}
	L_{k, i} = \left\{ \begin{aligned}
	&Q_i + P_i,	&i <i^-,\\
	&Q_i + q_k,	&i^- \leqslant i < i^+,\\
	&m_k,	&i \geqslant i^+.
	\end{aligned} \right.
\end{equation}
	
	To calculate $L_{k+1, i}$, we need to know more about $P_{k, i}$, $Q_{k, i}$. 
	Based on \eqref{type-of-T_k} and $i^- \leqslant i^+$,
\begin{subequations}\label{detail-type-of-T_k}
	\begin{align}
	Q_{k, i}&= \left\{ \begin{aligned}
	&P_i,	& i < i^-,\\
	&q_k,	& i^- \leqslant i \leqslant i^+,
	\label{eq: detail-type-of-T_k-Q}\\
	&q_k-p_k+Q_{i-1},	& i>i^+,
	\end{aligned}\right.
	\\
	P_{k,i}&= \left\{ \begin{aligned}
	&Q_i,	& i \leqslant i^-,\\
	&Q_i+P_{i-1}-q_k,	& i^- < i < i^+,\\
	&p_k - q_k + \max\{P_{i-1}, q_k\},	&i \geqslant i^+.
	\end{aligned}\right.
	\label{eq: detail-type-of-T_k-P}
	\end{align}
\end{subequations}
	Further calculations are divided into two cases, according to whether $q_{k+1} \geqslant P_{k, i^+}$ or not.
	If it is the case, then $L_{k+1, i} - L_{k, i}$ would be increasing.
	Otherwise, more delicate work is necessary.

\subsection{Monotonicity of difference: easy case}\label{subsec: easy}
	Suppose now $q_{k+1} \geqslant P_{k, i^+}$. 
	Since $P_{k, i^+} = p_k - q_k + \max\{ P_{i^+ -1}, q_k\},$ 
	we have $q_k^\sigma = q_k + q_{k+1} - p_k \geqslant P_{i^+ -1}$. 
	\textbf{Here are some key points of the arguement in this case:}
\begin{itemize}
	\item \eqref{fulfill-1} is equivalent to $L_{k+1, i}^\sigma \stackrel ?= L_{k, i}$,
	which follows by direct computation;
	\item \eqref{fulfill-2} is equivalent to 
	$L_{k+1, i} \stackrel ?= L_{k+1, i}^\sigma + L_{k, i-1}^\sigma - L_{k+1, i-1}^\sigma$,
	which follows from $L_{k+1, i-1}^\sigma = L_{k, i-1}$ and \eqref{type}.
\end{itemize}

	Based on \eqref{type-of-S_k+1} and \eqref{detail-type-of-T_k},
	\[\begin{aligned}
	L_{k+1, i}^+&= \left\{ \begin{aligned}
	&L_{k, i}^-,	& i \leqslant i^+,\\
	&\min\{q_k - p_k + Q_{i-1}, p_{k+1}\},	& i > i^+,
	\end{aligned}\right.
	\\\
	L_{k+1,i}^-&= \left\{ \begin{aligned}
	&L_{k, i}^+,	& i \leqslant i^-,\\
	&L_{k,i}^++P_{i-1}-q_k,	& i^- < i \leqslant i^+,\\
	&\min\{ p_k - q_k +P_{i-1}, q_{k+1} \},	&i > i^+,
	\end{aligned}\right.\end{aligned}\]
	so 
	\[L_{k+1, i} = \left\{ \begin{aligned}
	&L_{k, i},	& i \leqslant i^-,\\
	&L_{k, i} + P_{i-1} - q_k,	& i^-< i \leqslant i^+,\\
	&L_{k, i-1}^\sigma,	& i > i^+.
	\end{aligned} \right.\]
	Here we have used \eqref{type-of-S_k} that
	\[L_{k, i}^\sigma = \min\{ Q_i , p_k+p_{k+1} - q_k\} + 
	\min\{ P_i, q_k + q_{k+1} - p_k\}. \]
	Then
	\[L_{k+1, i} - L_{k, i} = \left\{ \begin{aligned}
	&0,	& i \leqslant i^-,\\
	&P_{i-1} - q_k,	& i^- < i \leqslant i^+,\\
	&L_{k, i-1}^\sigma - m_k,	& i > i^+,
	\end{aligned} \right.\]
	is increasing, since at $i = i^- +1$, $P_{i^- } - q_k \geqslant 0$ by definition, and at $i= i^+ +1$, 
	\[\begin{aligned}
	L_{k, i^+}^\sigma - m_k  = &
	\left(\min\{ Q_{i^+}, p_k^\sigma \} - p_k \right) + 
	\left(\min\{ P_{i^+}, q_k^\sigma\} - q_k \right) \\
	\geqslant &
	0 + \min\{ P_{i^+}, q_k^\sigma\} - q_k 
	\geqslant P_{i^+-1} -q_k.
	\end{aligned}\]
	To obtain \eqref{fulfill-1}, it suffices to prove $L_{k+1, i}^\sigma \stackrel ?= L_{k, i}$.
	
	Now apply \eqref{type-of-S_k}-\eqref{type-of-S_k+1} to $\underline p^\sigma$.
	Since $p_{k}^\sigma \geqslant p_k, q_{k}^\sigma \geqslant q_k$,
	\[L_{k, i}^{\sigma, +}= \left\{ \begin{aligned}
	&Q_i,	& i < i^+,\\
	&\min\{ Q_i, p_k^\sigma\},	&i \geqslant i^+,
	\end{aligned}\right.
	\quad
	L_{k, i}^{\sigma, -}= \left\{ \begin{aligned}
	&P_i,	& i < i^-,\\
	&\min\{ P_i, q_k^\sigma\},	&i \geqslant i^-,
	\end{aligned}\right.\]
	and 
	\[Q_{k, i}^\sigma= \left\{ \begin{aligned}
	&L_{k, i}^{\sigma, -},	& i \leqslant i^+,\\
	&L_{k,i}^{\sigma, -} + Q_{i-1} - L_{k, i-1}^{\sigma, +},	& i>i^+,
	\end{aligned}\right.
	\quad
	P_{k,i}^\sigma= \left\{ \begin{aligned}
	&L_{k, i}^{\sigma, +},	& i \leqslant i^-,\\
	&L_{k, i}^{\sigma, +} + P_{i-1} - L_{k, i-1}^{\sigma, -},	&i > i^-.
	\end{aligned}\right.\]
	In $L_{k+1, i}^{\sigma, +} = \min \{Q_{k, i}^\sigma, q_k \}$, 
	if $i \leqslant i^+$, then it becomes $\min\{P_i, q_k^\sigma, q_k\}  = \min\{P_i, q_k\} =L_{k, i}^-$;
	while $i > i^+ (\geqslant i^-)$, then $Q_{k, i}^\sigma \geqslant L_{k,i}^{\sigma, -}= \min\{P_i, q_k^\sigma\} \geqslant q_k$, so $L_{k+1, i}^{\sigma, +} = q_k$ again coincide with $L_{k, i}^+$.
	It's also easy to deduce that $L_{k+1, i}^{\sigma, -} = L_{k, i}^+$ by considering $i \leqslant i^-$, $i^- < i \leqslant i^+$ and $i > i^+$ respectively.
	Hence, we may obtain $L_{k+1, i}^\sigma = L_{k, i}$ and \eqref{fulfill-1} follows.
	
	Moreover, since
	\[L_{k, i}^\sigma - L_{k+1, i}^\sigma = \left\{ \begin{aligned}
	&0,	& i \leqslant i^-,\\
	&\min\{P_i, q_k^\sigma\} - q_k,	& i^- \leqslant i < i^+,\\
	&L_{k, i}^\sigma - m_k,	& i \geqslant i^+,
	\end{aligned} \right.\]
	is also increasing, \eqref{fulfill-2} then becomes 
	\[L_{k+1, i} \stackrel ?= L_{k+1, i}^\sigma + L_{k, i-1}^\sigma - L_{k+1, i-1}^\sigma.\]
	It follows from $L_{k+1, i-1}^\sigma = L_{k, i-1}$ and \eqref{type}.
	Now Lemma \ref{swap-for-nu-fill} is verified in this case.

\subsection{Monotonicity of difference: remaining case}\label{subsec: hard}
	Now suppose $q_{k+1} < P_{k, i^+}$. 
	Since 
	\[P_{k, i^+} = p_k - q_k + \max\{P_{i^+ -1}, q_k\}=
	\max\{p_k - q_k + P_{i^+-1}, p_k\},\]
	and $p_k \leqslant q_{k+1}$, 
	we must have $ P_{i^+ -1} > q_k > P_{i^- -1}$.
	It follows that $i^+ > i^-$.
	\textbf{Here are some key points of the arguement in this case:}
\begin{itemize}
	\item it involves four index $i^- \leqslant i'' < i' \leqslant i^+$;
	\item \eqref{fulfill-1} is equivalent to \eqref{fulfill-1-2}, 
	which follows by comparing \eqref{detail-type-of-S_k+1^sigma} to \eqref{detail-type-of-S_k} and \eqref{detail-type-of-S_k+1} directly;
	\item when $i \leqslant i''$ and $i > i^+$, 
	 \eqref{fulfill-2} becomes $L_{k+1, i} \stackrel?= L_{k+1, i}^\sigma + L_{k, i-1}^\sigma - L_{k+1, i-1}^\sigma$,
	 which follows from $L_{k+1, i-1}^\sigma = L_{k, i-1}$ (by comparing \eqref{detail-type-of-S_k+1} and \eqref{detail-type-of-S_k+1^sigma} directly) and \eqref{type};
	\item when $i''< i \leqslant i^+$, \eqref{fulfill-2} becomes $L_{k+1, i} \stackrel?= L_{k+1, i}^\sigma + q_{k+1} -p_k$,
	which follows by comparing \eqref{detail-type-of-S_k+1} and \eqref{detail-type-of-S_k+1^sigma} directly.
\end{itemize}
	
	Take $i'= \min \{ i \mid P_{k, i}>q_{k+1}\}$.
	By $q_{k+1} < P_{k, i^+}$ we know $i' \leqslant i^+$.
	On the other hand, in \eqref{detail-type-of-T_k} we have calculated that $P_{k, i^-} = Q_{i^-} < p_k \leqslant q_{k+1}$,
	so $i' > i^-$.
	
	Based on \eqref{type-of-S_k+1} and \eqref{detail-type-of-T_k} we get
\begin{subequations}\label{detail-type-of-S_k+1}
	\begin{align}
	L_{k+1, i}^+&= \left\{ \begin{aligned}
	&L_{k, i}^-,	& i \leqslant i^+,\\
	&\min\{q_k - p_k + Q_{i-1}, p_{k+1}\},	& i > i^+,
	\end{aligned}\right.
	\label{eq: detail-type-of-S_k+1-+}\\
	L_{k+1,i}^-&= \left\{ \begin{aligned}
	&Q_i,	& i \leqslant i^-,\\
	&Q_i +P_{i-1}-q_k,	& i^- < i < i',\\
	&q_{k+1},	&i \geqslant i'.
	\end{aligned}\right.
	\label{eq: detail-type-of-S_k+1--}
	\end{align}
\end{subequations}
	Then
	\[\begin{aligned}
	L_{k+1, i}^+ - L_{k, i}^- &= \left\{ \begin{aligned}
	&0,	& i \leqslant i^+,\\
	&\min\{ Q_{i-1}-p_k, p_{k+1} - q_k\},	& i > i^+,
	\end{aligned} \right.
	\\
	L_{k+1, i}^- - L_{k, i}^+ &= \left\{ \begin{aligned}
	&0,	& i \leqslant i^-,\\
	&P_{i-1} - q_k,	& i^- < i < i',\\
	&q_{k+1} -Q_i,	& i' \leqslant i < i^+,\\
	&q_{k+1}-p_k,	& i \geqslant i^+.
	\end{aligned} \right.
	\end{aligned}\]
	From the above formula we can see that $L_{k+1, i} -L_{k, i}$ is increasing when $i \leqslant i'$ and $i \geqslant i^+$, and decreasing when $ i' \leqslant i \leqslant i^+$.
	Its minimum over $i \geqslant i'$ is $q_{k+1} -p_k$.
	To obtain \eqref{fulfill-1}, it suffices to prove 
\begin{equation}\label{fulfill-1-2}
	L_{k+1, i}^\sigma \stackrel ? = \max\{ L_{k, i}, L_{k+1, i} - q_{k+1} + p_k\} .
\end{equation}
	We observe that the right hand side is $L_{k, i}$ when $i \leqslant i^-$ and $i \geqslant i^+$, is $L_{k+1, i} - q_{k+1} + p_k$ when $i' \leqslant i <i^+$.
	
	Now apply \eqref{type-of-S_k} to $\underline p^\sigma$.
	Since $p_{k}^\sigma \geqslant p_k$,
	\[L_{k, i}^{\sigma, +}= \left\{ \begin{aligned}
	&Q_i,	& i < i^+,\\
	&\min\{ Q_i, p_k^\sigma\},	&i \geqslant i^+.
	\end{aligned}\right. \]
	For $L_{k, i}^{\sigma, -} = \min\{P_i, q_k^\sigma\}$, we need compare $P_i$ with $q^\sigma_k = q_k + q_{k+1} - p_k$.
	Introduce the fourth index $i'' = \min\{ i \mid P_i > q_k^\sigma\}$. 
	Since $q_k^\sigma \geqslant q_k > P_{i^- -1}$, we have $i'' \geqslant i^-$.
	Note $P_{k, i^+} = p_k -q_k +P_{i^+-1}$ is assumed to  $> q_{k+1}$,
	so $P_{i^+-1} > q_k^\sigma$, hence $i'' \leqslant i^+ -1$.
	If $i^- < i < i^+$ and $P_{k, i} = Q_i + P_{i-1} - q_k > q_{k+1}$, then $P_{i-1}  > q_{k+1} + q_k - Q_i \geqslant  q_k^\sigma$, 
	so $\forall i \in [i'', i^+)$, we have $i -1 \geqslant i''$.
	This shows $i'' \leqslant i'-1$.
	According to the definition of $i''$ that 
	\[L_{k, i}^{\sigma, -}= \left\{ \begin{aligned}
	&P_i,	& i < i'',\\
	&q_k^\sigma,	&i \geqslant i'',
	\end{aligned}\right.\]
	and  \eqref{type-of-T_k}, we calculate
	\[Q_{k, i}^\sigma= \left\{ \begin{aligned}
	&L_{k, i}^{\sigma, -},	& i \leqslant i^+,\\
	&L_{k,i}^{\sigma, -} + Q_{i-1} - L_{k, i-1}^{\sigma, +},	& i>i^+,
	\end{aligned}\right.
	\quad
	P_{k,i}^\sigma= \left\{ \begin{aligned}
	&L_{k, i}^{\sigma, +},	& i \leqslant i'',\\
	&L_{k, i}^{\sigma, +} + P_{i-1} - L_{k, i-1}^{\sigma, -},	&i > i''.
	\end{aligned}\right.\]
	Then apply \eqref{type-of-S_k+1},
	\[L_{k+1, i}^{\sigma, +} = \left\{ \begin{aligned}
	&P_i,	& i < i^-,\\
	&q_k,	& i \geqslant i^-,
	\end{aligned}\right.
	\quad
	L_{k+1, i}^{\sigma, -} = \left\{ \begin{aligned}
	&Q_i,	& i \leqslant i'',\\
	&\min\{Q_i + P_{i-1} - q_k^\sigma, p_k\}	& i'' < i < i^+,\\
	&p_k,	& i \geqslant i^+,
	\end{aligned}\right.\]
	and hence
\begin{equation}\label{detail-type-of-S_k+1^sigma}
	L_{k+1, i}^{\sigma} = \left\{ \begin{aligned}
	&P_i + Q_i,	& i < i^-,\\
	&q_k + Q_i	& i^- \leqslant i \leqslant i'',\\
	&Q_i + P_{i-1} + p_k -q_{k+1},	& i'' < i < i',\\
	&m_k,	& i \geqslant i'.
	\end{aligned}\right.
\end{equation}

	\eqref{fulfill-1-2} becomes $L_{k+1, i}^\sigma \stackrel ?= L_{k, i}$ when $i \leqslant i^-$ or $i \geqslant i^+$,
	and then follows by comparing \eqref{detail-type-of-S_k} and \eqref{detail-type-of-S_k+1^sigma} directly.
	It becomes  
	$L_{k+1, i}^\sigma \stackrel ?= L_{k+1, i} - q_{k+1} + p_k$ 
	when $i' \leqslant i <i^+$,
	and then follows by comparing \eqref{detail-type-of-S_k+1} and \eqref{detail-type-of-S_k+1^sigma} directly.
	If $i^- < i < i'$, the right hand side of \eqref{fulfill-1-2} becomes
	\[\max\{ L_{k, i}, L_{k+1, i} - q_{k+1} + p_k\} = 
	\max\{ Q_i + q_k, Q_i + P_{i-1} -q_{k+1} + p_k \}\]
	according to \eqref{detail-type-of-S_k+1}.
	It also coincide with $L_{k+1, i}^\sigma$ once we check the definition of $i''$.
	Hence we obtain \eqref{fulfill-1-2} and \eqref{fulfill-1}. 
	
	To prove \eqref{fulfill-2}, we need to calculate as follows:
	\[L_{k, i}^\sigma - L_{k+1, i}^\sigma = \left\{ \begin{aligned}
	&0,	& i \leqslant i^-,\\
	&P_i - q_k,	& i^- \leqslant i < i'',\\
	&q_{k+1} - p_k,	& i = i'',\\
	&q_k^\sigma - P_{i-1} + q_{k+1} - p_k,	& i'' < i < i',\\
	&\min\{Q_i, p_k^\sigma\} + q_k^\sigma - m_k,	& i \geqslant i'.
	\end{aligned} \right.\]
	This means the sequence $L_{k, i}^\sigma - L_{k+1, i}^\sigma$ is increasing when $i \leqslant i''$ and $i > i'$, and decreasing when $i'' \leqslant i <i'$.
	Its maximum over $i \leqslant i'$ is $q_{k+1} - p_k$ (taken at $i''$). 
	Moreover, if $i' \leqslant i < i^+$, then it takes value $Q_i + q_k^\sigma - m_k < p_k + q_k^\sigma - m_k = q_{k+1} - p_k$,
	while at $i=i^+$, its value $\geqslant p_k + q_k^\sigma - m_k = q_{k+1} -p_k$. 
	Consequently,
	\[\max\{ L_{k, j}^\sigma - L_{k+1, j}^\sigma \mid j < i\} 
	=\left\{ \begin{aligned}
	&L_{k, i-1}^\sigma - L_{k+1, i-1}^\sigma,	& i \leqslant i'',\\
	&q_{k+1} - p_k,	& i'' < i \leqslant i^+,\\
	&L_{k, i-1}^\sigma - L_{k+1, i-1}^\sigma,	& i> i^+.
	\end{aligned}\right.\]
	
	 \eqref{fulfill-2} becomes $L_{k+1, i} \stackrel ?= L_{k+1, i}^\sigma + L_{k, i-1}^\sigma - L_{k+1, i-1}^\sigma$ when $i \leqslant i''$ and $i > i^+$,
	so is equivalent to $L^\sigma_{k+1, i-1} \stackrel ?= L_{k, i-1}$ according to \eqref{type}.
	This then follows by comparing \eqref{detail-type-of-S_k} and \eqref{detail-type-of-S_k+1^sigma} directly.
	If $i'' < i \leqslant i^+$, \eqref{fulfill-2} becomes
	\[L_{k+1, i} \stackrel ?= L_{k+1, i}^\sigma + q_{k+1} -p_k.\]
	In \eqref{detail-type-of-S_k+1} we have calculated
	\[L_{k+1, i} = \left\{\begin{aligned}
	&Q_i + P_{i-1},	& i''< i < i',\\
	&q_k + q_{k+1},	& i' \leqslant i \leqslant i^+,
	\end{aligned}\right.\]
	and in \eqref{detail-type-of-S_k+1^sigma} we have calculated
	\[L_{k+1, i}^\sigma = \left\{\begin{aligned}
	&Q_i + P_{i-1} + p_k -q_{k+1},	& i''< i < i',\\
	&m_k,	& i' \leqslant i \leqslant i^+.
	\end{aligned}\right.\]
	Now \eqref{fulfill-2} for $i'' < i \leqslant i^+$ also follows.
	We verified \eqref{swap-for-nu-fill} thoroughly.

\subsection{Further conclusions}
	Here we provide two further conclusions of the previous calculations.
	The first one gives a simpler formula for overlaps of adjacent skew columns arising from $\underline p \in \MR$,
	which implies the necessity in Theorem \ref{non-vanishing}.
	The second one gives a deeper formula for types of skew columns arising from $\underline p \in \MR$ after Trapa's operation.
	It will be useful for the sufficiency in Theorem \ref{non-vanishing}.
	
\begin{lem}\label{overlap-p}
	$\overlap(S_k, S_{k+1}) = \min\{p_k, q_{k+1} \} + \min \{ q_k, p_{k+1} \}$.
\end{lem}
\begin{rmk}
	This is consistent with theorem 3.6 in \cite{Du}.
\end{rmk}
\begin{proof}
	Denote $P_i, Q_i$ to be the number of $+, -$ filling into first $i$ components of extreme column of $S^{k-1} = \bigsqcup_{j<k} S_j$, as in subsection \ref{type-for-p}.
	Let $i^- = \min\{ i \mid P_i \geqslant q_k\}$,  $i^+ = \min\{ i \mid Q_i \geqslant p_k\}$, and assume $i^- \leqslant i^+$.
	
	If $i^- = i^+$ then $S_k$ does not break, and the formula is easy to observe.
	If $i^- < i^+$, then $S_k$ breaks after $i^-$th component.
	We have
	\[\overlap(S_k^{i^-}, S_{k+1}^{i^-}) =
	\min\{ Q_{i^-}, q_{k+1} \} + \min\{q_k, p_{k+1} \} ,\]
	where notations are same as in the proof of Lemma \ref{overlap}. 
	Using a similiar argument as for \eqref{reccursive-overlap}, we would obtain
	\[\overlap(S_k, S_{k+1}) =
	\overlap(S_k^{i^- }, S_{k+1}^{i^-}) +
	\min\{ p_k - Q_{i^- }, q_{k+1} - \min\{Q_{i^-}, q_{k+1} \}\},\]
	for which we need the following observation:
\begin{itemize}
	\item $p_k > Q_{i^-}$, so the $i^- +1, \cdots, i^+$th components of $S_k$ will be filled by $+$, and that of $S_{k+1}$ will be filled by $-$;
	\item the $i^++1, \cdots, (k+1)$th components of $S_{k+1}$ contribute nothing to overlap.
\end{itemize}
	Then we deduce further
	\[\begin{aligned}
	\overlap(S_k, S_{k+1}) = &
	\min\{q_k, p_{k+1}\} +  \min\{p_k - Q_{i^-} + \min\{ Q_{i^-}, q_{k+1} \}, q_{k+1} \} \\
	= &
	\min\{q_k, p_{k+1} \} + \min\{p_k, q_{k+1} \}.
	\end{aligned}\]
	since $q_{k+1} \geqslant p_k > Q_{i^-}$.
\end{proof}

	Recall $S_k' \sqcup S_{k+1}'$ is the partition obtained from $S_k \sqcup S_{k+1} = S_k^\sigma \sqcup S_{k+1}^\sigma$ by Trapa's operation.
	
\begin{lem}\label{operation-swap}
	Denote  the type of $\nu$-filling on $S_j$, $S_j^\sigma$, $S_j'$ by $\nu_{j; i}, \nu_{j; i}^\sigma, \nu_{j; i}'$ respectively.
	Then
	\[\nu_{k; i}' = \max\{\nu_{k; i}, \nu_{k; i}^\sigma \},\quad
	\nu_{k+1; i}' = \min\{ \nu_{k+1; i}, \nu_{k+1; i}^\sigma \}.\]	
\end{lem}
\begin{proof}
	According to \eqref{fulfill-1}
	\[L^\sigma_{k+1,i}= L_{k+1, i} - \min\{ L_{k+1, j} - L_{k,j} \mid j \geqslant i\},\]
	we know
	\[\begin{aligned}
	\nu_{k+1; 0}^\sigma - \nu_{k+1; i}^\sigma =&
	\nu_{k+1; 0} - \nu_{k+1; i} - \min\{ \nu_{k+1; 0} - \nu_{k+1; j} - \nu_{k;0} + \nu_{k; j} \mid j \geqslant i \} \\
	= &
	\nu_{k; 0} - \nu_{k+1; i} - \min\{ \nu_{k; j} - \nu_{k+1; j} \mid j \geqslant i\},
	\end{aligned}\]
	and hence 	
\begin{align}\label{fulfill-type-1}
	\nu_{k+1; i}^\sigma = \nu_{k+1; i} + \min\{ \nu_{k; j} - \nu_{k+1; j} \mid j \geqslant i\}.
\end{align}
	Since $m_k \leqslant m_{k+1}$, we know by \eqref{eq: operation-sub-k+1} that
	\[\nu_{k+1; i}' = \nu_{k+1; i} + \min\{0, \nu_{k; j} - \nu_{k+1; j} \mid j \geqslant i\}.\]
	Combine them and we obtain
	\[\nu_{k+1; i}' = \min\{\nu_{k+1; i}, \nu_{k+1; i}^\sigma\}.\]
	
	According to \eqref{type}
	\[ L^\sigma_{k+1, i} + L^\sigma_{k, i-1} = L_{k+1, i} + L_{k, i-1}\]
	and \eqref{fulfill-1} we know
	\[L^\sigma_{k, i}= L_{k, i} + \min\{ L_{k+1, j} - L_{k,j} \mid j \geqslant i+1\}.\]
	Then
	\[\begin{aligned}
	\nu_{k; 0}^\sigma - \nu_{k; i}^\sigma = &
	\nu_{k; 0} - \nu_{k; i} + \min\{ \nu_{k+1; 0} - \nu_{k+1; j} - \nu_{k;0} + \nu_{k; j} \mid j \geqslant i+1\}\\
	= &
	\nu_{k+1; 0} - \nu_{k; i} + \min\{ \nu_{k; j} - \nu_{k+1; j} \mid j \geqslant i+1\},
	\end{aligned}\]
	and hence
\begin{equation}\label{fulfill-type-2}
	\nu_{k; i}^\sigma = \nu_{k; i} - \min\{ \nu_{k; j} - \nu_{k+1; j} \mid j \geqslant i+1\}.
\end{equation}
	We also know by \eqref{eq: operation-sub-k} that
	\[\nu_{k; i}' = \nu_{k; i} - \min\{0, \nu_{k; j} - \nu_{k+1; j} \mid j \geqslant i+1\}.\]
	Combine them and we obtain
	\[\nu_{k; i}' = \max\{\nu_{k; i}, \nu_{k; i}^\sigma\}.\]
	Now the conclusion follows.
\end{proof}

%% file: Proof-2.tex
\section{Proof of sufficiency}\label{sufficiency}
	Denote $\mathrm P(\nu_1, \cdots, \nu_r; \underline p)$ to be the following proposition:
	if $\underline p = (p_1, \cdots, p_r) \in \MR$ lies in all $\mathrm B^\sigma(i)$ and $\mathrm C^\sigma(i, i+1)$, then $A_{\q_{\underline p}}(\lambda_{\underline p})$ is non-zero.
	This is the sufficiency part of Theorem \ref{non-vanishing}.
	According to the conclusion in section \ref{swap-tableau}, $\mathrm P(\nu_1, \cdots, \nu_r; \underline p)$ is equivalent to $\mathrm P(\nu_1^\sigma, \cdots, \nu_r^\sigma; \underline p^\sigma)$, $\forall \sigma \in \Sigma_r$,
	so it suffices to prove $\mathrm P(\nu_1, \cdots, \nu_r; \underline p)$ for an ``appropriate'' arrangement $\nu = (\nu_1, \cdots, \nu_r)$.
	
	Take $\nu_r \in \mathrm{Seg}_\psi$ that minimizes $\ed(\nu_{r})$ and then $\ed(\nu_r) - \bg(\nu_r)$;
	it is also determined by the property that $\forall i<r$, either $\nu_i > \nu_r$, or $\nu_i \subset \nu_r$.
	If $\forall i<r$, $\nu_i > \nu_r$, then we would prove $\mathrm P(\nu_1, \cdots, \nu_{r-1}; \underline p^{r-1}) \Rightarrow \mathrm P(\nu_1, \cdots, \nu_{r}; \underline p)$, 
	where $\underline p^{r-1} = (p_1, \cdots, p_{r-1})$ consists of first $r-1$ components of $\underline p$.
	This is Lemma \ref{lem: easy-induction}; it's an easy case.
	
	If $\nu_r$ contains other segments in $\mathrm{Seg}(\psi)$, then take $\nu_{r-1}$ that minimizes  $\bg(\nu_{r-1})$, and then $\bg(\nu_{r-1}) - \ed(\nu_{r-1})$;
	it is also determined by the property that $\forall i<r-1$, either $\nu_i > \nu_{r-1}$, or $\nu_i \supset \nu_{r-1}$.
	In this case, we would prove $\mathrm P(\nu_1, \cdots, \nu_{r-1}; \underline p^{r-1}) \wedge \mathrm P(\nu_1^\sigma, \cdots, \nu_{r-1}^\sigma; \underline p^{\sigma, r-1})\Rightarrow \mathrm P(\nu_1, \cdots, \nu_{r}; \underline p)$, 
	where $\sigma = (r-1, r) \in \Sigma_r$, and $\underline p^{\sigma, r-1}$ consists of first $r-1$ components of $\underline p^\sigma$.
	This is Lemma \ref{lem: general-induction}.
	
	$\mathrm P(\nu_1, \cdots, \nu_r; \underline p)$ follows from the above inductions;
	its starding point, case $r=2$, is exactly the sufficient part of Lemma \ref{non-vanish-2}.
	Here we scketch
	
\begin{proof}[Proof of Lemma \ref{non-vanish-2}]
	It suffices to show for $\underline p \in \mathrm B(1) \cap \mathrm B(2)$ that, 
	$A_{\q_{\underline p}}(\lambda_{\underline p}) \not = 0$ iff $\underline p \in \mathrm C(1, 2)$;
	recall $\MR \to \Pi_\psi(G) \sqcup \{0\}$ in subsection \ref{subsec: transition} that
	those $\underline p$ outside $\mathrm B(1) \cap \mathrm B(2) = \mathcal D(\psi)$ are sent to zero representation.
	According to Trapa's criterion, $A_{\q_{\underline p}}(\lambda_{\underline p}) \not = 0$ iff $\overlap(S_1, S_2) \geqslant \sing(\nu_1, \nu_2)$;
	$S_1 \sqcup S_2$ is the Young diagram constructed from $\underline p$ as in subsection \ref{construct-tableau}.
	Lemma \ref{overlap-p} has calculated $\overlap(S_1, S_2) = \min\{p_1, q_2\} + \min\{q_1, p_2\}$,
	so the $\overlap \geq \sing$ here is exactly $\mathrm C(1, 2)$.
\end{proof}

\subsection{The easy case}
	In this subsection we prove the following lemma. 
\begin{lem}\label{lem: easy-induction}
	If $\forall i<r$, $\nu_i >\nu_r$, then $\mathrm P(\nu_1, \cdots, \nu_{r-1}; \underline p^{r-1}) \Rightarrow \mathrm P(\nu_1, \cdots, \nu_{r}; \underline p)$. 
	Recall $\underline p^{r-1}$ consists of the first $r-1$ components of $\underline p$.
\end{lem}

	Denote the compatible partition of $\nu$-quasitableau constructed from $\underline p^\sigma \in \MR^\sigma$ by $S^\sigma = \bigsqcup S_k^\sigma$,
	and the type of $\nu$-filling on $S_k^\sigma$ by $\nu_{k; i}^\sigma$.
	We have seen in section \ref{swap-tableau} that the shapes of diagram $S^\sigma$, $\sigma \in \Sigma_r$ coincide.
	The following general result is also necessary for latter subsections.

\begin{prop}\label{type-of-last}
	Suppose the $\nu$-quasitableau on $S=\bigsqcup_{k \leqslant r} S_k$ is equivalent to a non-zero $\nu$-antitableu with compatible partition $S= \bigsqcup_{k \leqslant r} S_k'$,
	then the type of $\nu$-filling on $S_r'$ is give by  
	\[\nu_{r; i}' = \min\{ \nu_{r; i}^\sigma \mid \sigma \in \Sigma_r\} .\]
\end{prop}
\begin{proof}
	Prove by induction on $r$. 
	Case $r=2$ is exactly Lemma \ref{operation-swap}, and then consider $r>2$.
	Since $\underline p$ is assumed to parametrize a non-zero representation, so is $\underline p^{r-1}$.
	Now the $\nu$-quasitableau on $S^{r-1} = \bigsqcup_{k<r} S_k$ is equivalent to  a non-zero $\nu$-antitableau with compatible partition $\bigsqcup_{k<r} T_k.$
	By induction hypothesis on case $r-1$, the type of $\nu$-filling on $T_{r-1}$ is given by
	\[\mu_{r-1; i} = \min\{ \nu_{r-1; i}^\sigma \mid \sigma \in \Sigma_{r-1} \}.\]
	\textbf{We may assume} $\forall i<r,$ either $\nu_i> \nu_r$ or $\nu_i \subset \nu_r$,
	so there is an injection $\Sigma_{r-1} \to \Sigma_r$ with image
	\[\{ \sigma \in \Sigma_r \mid \sigma(r) = r\};\]
	identify these two sets.
	
	Apply Trapa's operation to the $\nu$-filling on $T_{r-1} \sqcup S_r$, 
	and get a new partition $T_{r-1}^{(1)} \sqcup S_r^{(1)}$. 
	The segment filling into $S_r^{(1)}$ has minimal beginnings and ends, 
	so further adjustment on $\left(\bigsqcup_{k<r-1} T_k\right) \sqcup T_{r-1}^{(1)} \sqcup S_r^{(1)}$ according to Trapa's algorithm would never change $S_r^{(1)}$ any more.
	Consequently, this $S_r^{(1)}$ and its $\nu$-filling are consistent with the final $S_r'$,
	and then we konw the type of $\nu$-filling on $S_r'$ by $T_{r-1} \sqcup S_r \sim T_{r-1}^{(1)} \sqcup S_r'$ and \eqref{eq: operation-sub-k+1} as follows:
	\[\begin{aligned}
	\nu_{r;i}' = &
	\min\{ \nu_{r; i}, \nu_{r; i} - \nu_{r; j} + \mu_{r-1; j} \mid j \geqslant i \}\\
	= &
	\min\{\nu_{r; i}, \nu_{r; i} - \nu_{r; j} + \nu_{r-1; j}^\sigma \mid \sigma \in \Sigma_{r-1},  j \geqslant i \}. 
	\end{aligned}\]
	
	For fixed $\sigma \in \Sigma_{r-1}$, if $\nu_{r-1}^\sigma > \nu_r^\sigma = \nu_r$, then by the fact that $\underline p^\sigma$ parametrizes non-zero representation and Trapa's algorithm we know 
	\[\overlap(S_{r-1}^\sigma, S_r^\sigma) \geqslant \sing(\nu_{r-1}^\sigma, \nu_r^\sigma).\]
	It follows further by Lemma \ref{overlap-geq-sing} that
	\[\nu_{r-1; j}^\sigma \geqslant \nu_{r; j}^\sigma = \nu_{r; j},\]
	so
	\[\min\{\nu_{r; i}, \nu_{r; i} - \nu_{r; j} + \nu_{r-1; j}^\sigma \mid  j \geqslant i \} = \nu_{r; i}.\]
	If $\nu_{r-1}^\sigma \subset \nu_r^\sigma = \nu_r$, then by \eqref{fulfill-type-1},
	\[\min\{\nu_{r; i}, \nu_{r; i} - \nu_{r; j} + \nu_{r-1; j}^\sigma \mid  j \geqslant i \}
	=
	\min\{\nu_{r; i}, \nu_{r; i}^{\sigma \circ (r-1, r)}\}. \]
	Now we obtain
	\[\begin{aligned}
	\nu_{r;i}' = &
	\min\{\nu_{r; i}, \nu_{r; i} - \nu_{r; j} + \nu_{r-1; j}^\sigma \mid  j \geqslant i, \sigma \in \Sigma_{r-1}\}\\
	=&
	\min\{ \nu_{r;i}, \nu_{r; i}^\sigma \mid \sigma(r-1) = r\}.
	\end{aligned}\]
	
	The conclusion would follows by the following two observations:
	\begin{itemize}
	\item if $\sigma, \tau \in \Sigma_r$, such that $\sigma(r) = \tau(r)$, then $\nu_{r; i}^\sigma = \nu_{r; i}^\tau$.
	\item if $\sigma \in \Sigma_r$ has $\sigma^{-1}(r) < r$, then $\forall \sigma^{-1}(r) < k < r$, $\nu_k^\sigma$ is contained in $\nu_r$ rather than precedes it, and there is $\tau = \sigma \circ ( \sigma^{-1}(r), \cdots, r-1) \in \Sigma$ such that $\tau(r) = \sigma(r)$, and $\tau(r-1) = r$. 
	\end{itemize}
	These observations imply that $\nu_{r; i}$ and $\nu_{r; i}^\sigma$ with $\sigma(r-1) = r$ exhaust $\{\nu_{r; i}^\sigma \mid \sigma \in \Sigma\}$.
\end{proof}

	\textbf{Now assume} $\forall i<r$, $\nu_i > \nu_r$, and we will show $\mathrm P(\nu_1, \cdots, \nu_{r-1}; \underline p^{r-1}) \Rightarrow \mathrm P(\nu_1, \cdots, \nu_{r}; \underline p)$.
	If $\underline p \in \MR$ lies in all $\mathrm B^\sigma(i)$ and $\mathrm C^\sigma(i, i+1)$, then so does $\underline p^{r-1}$.
	Since we have $\mathrm P(\nu_1, \cdots, \nu_{r-1}; \underline p^{r-1})$, 
	the $\nu$-quasitableau on $S^{r-1} = \bigsqcup_{k< r} S_k$ is equivalent to a $\nu$-antitableau with compatible partition $S^{r-1} = \bigsqcup_{k<r} T_k$.
	According to Proposition \ref{type-of-last}, the type of $\nu$-filling on $T_{r-1}$ is given by
	\[\mu_{r-1; i} = \min\{ \nu_{r-1; i}^\sigma \mid \sigma \in \Sigma_{r-1} \}.\]
	
	By the assumption that $\forall i< r$, $\nu_i > \nu_r$, we have $\forall \sigma \in \Sigma_{r-1}$, $\sigma(r) = r$; that is, $\Sigma_{r-1} = \Sigma_r$.
	More importantly, the segment $\mu_{r-1}$ filling $T_{r-1}$ precedes $\nu_r$,
	so it suffices that $T_{r-1} \sqcup S_r$ is not equivalent to $0$.
	According to Trapa's operation and Lemma \ref{overlap-geq-sing}, this means $\mu_{r-1; i} \stackrel ?\geqslant \nu_{r; i}$.
	
	For each $\sigma \in \Sigma_{r-1}$, since $\underline p \in \mathrm C^\sigma(r-1, r)$,
	$\overlap(S_{r-1}^\sigma , S_r^\sigma) \geqslant \sing(\nu_{r-1}^\sigma, \nu_r^\sigma)$ according to Lemma \ref{overlap-p}.
	Then by Lemma \ref{overlap-geq-sing}, $\nu_{r-1; i}^\sigma \geqslant \nu_{r;i}$ since
	we have assumed $\nu_{r-1}^\sigma > \nu_r^\sigma = \nu_r$.
	Consequently, $\mu_{r-1; i} =\min\{\nu_{r-1; i}^\sigma \} \geqslant \nu_{r; i}$.
	Hence, we obtain $\mathrm P(\nu_1, \cdots, \nu_{r-1}; \underline p^{r-1}) \Rightarrow \mathrm P(\nu_1, \cdots, \nu_{r}; \underline p)$, and
	Lemma \ref{lem: easy-induction} is finished.

\subsection{The remaining case}\label{general-case}
	In the rest of this section, we aim to prove the following.
\begin{lem}\label{lem: general-induction}
	Suppose
	\begin{itemize}
	\item $\forall i< r+1$, either $\nu_i> \nu_{r+1}$ or $\nu_i \subset \nu_{r+1}$;
	\item $\forall i<r$, either $\nu_i > \nu_r$ or $\nu_i \supset \nu_r$.
	\end{itemize}
	Then for  $\sigma = (r, r+1)\in \Sigma_r$, $\underline p = (p_1, \cdots, p_{r+1}) \in \MR$, there is
	$\mathrm P(\nu_1, \cdots, \nu_r; \underline p^r) \wedge 
	\mathrm P(\nu_1^\sigma, \cdots, \nu_r^\sigma; \underline p^{\sigma, r}) \Rightarrow 
	\mathrm P(\nu_1, \cdots, \nu_{r+1}; \underline p)$.
	Recall $\underline p^r$, $\underline p^{\sigma, r}$ consist of first $r$ components of $\underline p$, $\underline p^\sigma$.
\end{lem}

	\textbf{From now on we may assume $\nu= (\nu_1, \cdots, \nu_{r+1})$ satisfies the coondition of this lemma}.
	Denote by $S^\sigma = \bigsqcup S_k^\sigma$ the compatible partition of $\nu$-quasitableau given by $\underline p^\sigma \in \MR^\sigma$,
	and by $\nu_{k; i}^\sigma$ the type of $\nu$-filling on $S_k^\sigma$.
	Our proof mainly consists of two steps.
	\begin{enumerate}[fullwidth, itemindent=2em, label=\textbf{Step \arabic*.}]
	\item Suppose $\underline p^r$ parametrize a non-zero representation 
	(alternatively, $\bigsqcup_{k\leqslant r} S_k$ is equivalent to a non-zero partition $\bigsqcup_{k \leqslant r} T_k$).
	Find an equivalent condition for $\underline p$ parametrizing a non-zero representation 
	(alternatively, $\left(\bigsqcup_{k \leqslant r} T_k\right) \sqcup S_{r+1}$ is not equivalent to $0$).
	It comes out to be an upper bound for $\nu_{r+1; i}$ (see Proposition \ref{equiv-for-non-zero}).
	\item Suppose $\underline p^r, \underline p^{\sigma, r}$ both parametrize non-zero representations.
	Prove the type $\nu_{r+1; i}$ arising from $\underline p$ satisfies the above upper bound condition (see Proposition \ref{upper-bound-control});
	then $\underline p$ also parametrizes a non-zero representation.
	\end{enumerate}
	After then, Lemma \ref{lem: general-induction} follows easily.
	If $\underline p \in \MR$ lies in all $\mathrm B^\sigma(i), \mathrm C^\sigma(i, i+1)$, then so do $\underline p^r$ and $\underline p^{\sigma, r}$.
	By $\mathrm P(\nu_1, \cdots, \nu_r; \underline p^r)$ and $\mathrm P(\nu_1^\sigma, \cdots, \nu_r^\sigma; \underline p^{\sigma, r})$, we know $\underline p^r$ and $\underline p^{\sigma, r}$ both parametrize non-zero representations,
	then so do $\underline p$ according to the \textbf{Step 2} above.
	Hence we get $\mathrm P(\nu_1, \cdots, \nu_{r+1}; \underline p)$.

\subsection{An equivalent condition for non-zero condition}
	Let $T= \bigsqcup_{k<r} T_k$ be a compatible partition  of a $\mu$-antitableau,
	and $\mu_k$ be the segment filling into $T_k$.
	Due to Lemma \ref{anti-by-type}, the types of $T_k$ satisfy $\mu_{k; i} \geqslant \mu_{k+1; i}$.
	In particular, $\mu_1 \geqslant \cdots \geqslant \mu_{r-1}$.
\begin{prop}\label{equiv-for-non-zero}
	If there is a skew column $S_r$ filled by segment $\nu_r$ such that 
\begin{itemize}
	\item $T\sqcup S_r$ is a Young diagram, and
	\item $\mu_1, \cdots, \mu_h > \nu_r$, $\mu_{h+1}, \cdots, \mu_{r-1} \subset \nu_r$, 
\end{itemize}
	then the $\nu$-filling on $T \sqcup S_r$ is not equivalent to $0$ if and only if
\begin{subequations}\label{upper-bound}
	\begin{align}
	\nu_{r; i} \leqslant &
	\min\left\{ 
	\sum_{s=1}^{r-h-1} \mu_{r-s; j_{s-1}} - \mu_{r-s; j_s} + \mu_{h; j_{r-h-1}} \Bigm | 
	i = j_0 > \cdots > j_{r-h-1}
	\right \},
	\label{eq: upper-bound-a}\\
	\nu_{r; i} - \nu_{r; 0} \leqslant&  
	\min\left\{ 
	\sum_{s=1}^{t} \mu_{r-s; j_{s-1}} - \mu_{r-s; j_s} \Bigm | 
	i = j_0 > \cdots > j_t =0, 1 \leqslant t \leqslant r-h-1
	\right \}.
	\label{eq: upper-bound-b}
	\end{align}
\end{subequations}
\end{prop}
\begin{proof}
	Prove by induction on $r-h \geqslant 1$.
	If $r= h+1$, then \eqref{upper-bound} means $\nu_{r; i} \leqslant \mu_{h; i}$.
	The equivalence follows from Lemma \ref{overlap-geq-sing} then.
	
	If $r> h+1$, $T\sqcup S_r \nsim 0$ implies the following.
	\begin{subequations}\label{eq: equiv-non-zero}
\begin{enumerate}[fullwidth, itemindent=2em, label=(\alph*)]
	\item $T_{r-1} \sqcup S_r$ is converted into non-zero $T_{r-1}^{(1)} \sqcup S_r'$ via Trapa's operation, 
	which is equivalent to 
	\begin{align}\label{eq: equiv-non-zero-a}
	\nu_{r; i} - \nu_{r; 0} \stackrel?\leqslant \mu_{r-1; i} - \mu_{r-1; 0}
	\end{align}
	according to Lemma \ref{overlap-geq-sing}.
	\item $\left( \bigsqcup_{k<r-1} T_k \right) \sqcup T_{r-1}^{(1)} \nsim 0$;
	denote the segment filling into $T_{r-1}^{(1)}$ by $\mu_{r-1}^{(1)}$, 
	then due to induction hypothesis, this is equivalent to
	\[\begin{aligned}
	\mu_{r-1; j}^{(1)} \stackrel?\leqslant &
	\min\left\{ 
	\sum_{s=1}^{r-h-2} \mu_{r-1-s; j_{s-1}} - \mu_{r-1-s; j_s} + \mu_{h; j_{r-h-2}}\Bigm | 
	j = j_0 > \cdots > j_{r-h-2}
	\right \},\\
	\mu_{r-1; j}^{(1)} - \mu_{r-1; 0}^{(1)} \stackrel?\leqslant&  
	\min\left\{ 
	\sum_{s=1}^{t} \mu_{r-1-s; j_{s-1}} - \mu_{r-1-s; j_s} \Bigm | 
	j = j_0 > \cdots > j_t =0, 1 \leqslant t \leqslant r-h-2
	\right \}.
	\end{aligned}\]
	By \eqref{eq: operation-sub-k}, we know
	\[\mu_{r-1; j}^{(1)} = \max\{ \mu_{r-1; j},
	\mu_{r-1; j} - \mu_{r-1; i} + \nu_{r; i} \mid i>j \},\]
	and $\mu_{r-1; 0}^{(1)} = \nu_{r; 0}$,
	so this is equivalent further to 
	\begin{align}
	&\mu_{r-1; j} \stackrel?\leqslant 
		\label{eq: equiv-non-zero-b-1}\\
	&\min\left\{ 
	\sum_{s=1}^{r-h-2} \mu_{r-1-s; j_{s-1}} - \mu_{r-1-s; j_s} + \mu_{h; j_{r-h-2}}\Bigm | 
	j = j_0 > \cdots > j_{r-h-2}
	\right \},	\notag\\
	\mu_{r-1; j} - &\nu_{r; 0} \stackrel?\leqslant
		\label{eq: equiv-non-zero-b-2}\\
	&  \min\left\{ 
	\sum_{s=1}^{t} \mu_{r-1-s; j_{s-1}} - \mu_{r-1-s; j_s} \Bigm | 
	j = j_0 > \cdots > j_t =0, 1 \leqslant t \leqslant r-h-2
	\right \},	\notag\\
	&\nu_{r; i} \stackrel?\leqslant 
		\label{eq: equiv-non-zero-b-3}	\\
	&\min\left\{ 
	\sum_{s=0}^{r-h-2} \mu_{r-1-s; j_{s-1}} - \mu_{r-1-s; j_s} + \mu_{h; j_{r-h-1}} \Bigm | 
	i = j_{-1} > \cdots > j_{r-h-2}
	\right \},	\notag\\
	\nu_{r; i} - &\nu_{r; 0} \stackrel?\leqslant
		\label{eq: equiv-non-zero-b-4}	\\
	& \min\left\{ 
	\sum_{s=0}^{t} \mu_{r-1-s; j_{s-1}} - \mu_{r-1-s; j_s} \Bigm | 
	i = j_{-1} > \cdots > j_t =0, 1\leqslant t \leqslant r-h-2
	\right \}.
	\notag
	\end{align}
\end{enumerate}
	\end{subequations}
	The condition \eqref{upper-bound} is exactly \eqref{eq: equiv-non-zero-a} combined with \eqref{eq: equiv-non-zero-b-3}, \eqref{eq: equiv-non-zero-b-4}: 
	in \eqref{eq: equiv-non-zero-b-4}, $1\leqslant t \leqslant r-h-2$ can be replaced by $0 \leqslant t \leqslant r-h-2, j_0>0$; 
	then \eqref{eq: equiv-non-zero-a} adds the case $j_0=0$ to it.
	
	Moreover, \eqref{eq: equiv-non-zero-b-1}, \eqref{eq: equiv-non-zero-b-2} hold automatically:
\begin{itemize}
	\item $\forall j = j_0 > \cdots > j_{r-h-2}$, there is
	\[\mu_{r-1; j} \leqslant \mu_{r-2; j_0} 
	+ \sum_{s=1}^{r-h-2}  - \mu_{r-1-s; j_s} + \mu_{r-1-s-1; j_s},\]
	since $\mu_{r-1-s; j_s} \leqslant \mu_{r-1-s-1; j_s}, \forall s=0, \cdots, r-h-2;$ 
	\item $\forall j=j_0 > \cdots > j_t=0$ with $1 \leqslant t \leqslant r-h-2$, there is
	\[\mu_{r-1; j} - \nu_{r; 0} \leqslant \mu_{r-2; j_0} 
	+ \sum_{s=1}^{t-1}  - \mu_{r-1-s; j_s} + \mu_{r-1-s-1; j_s}
	- \mu_{r-1-t;0},\]
	since $\mu_{r-1-s; j_s} \leqslant \mu_{r-1-s-1; j_s}, \forall s=0, \cdots, t-1,$
	and $\nu_{r;0} \geqslant \mu_{r-1-t; 0}$ due to $\nu_r \supset \mu_{r-1-t}$. 
\end{itemize}
	This shows \eqref{upper-bound}  is equivalent to \eqref{eq: equiv-non-zero}, which are necessary conditions for $T \sqcup S_r\nsim 0$.
	
	To show the sufficiency, we have to show the above  \eqref{eq: equiv-non-zero} guarantee $T\sqcup S_r \nsim 0$.
	Apply Trapa's operation to $T_k \sqcup T_{k+1}^{(1)}$ and get $T_k^{(1)} \sqcup T_{k+1}'$ for $k = r-2, \cdots, h+1$ sequencially.
	Denote the segments filling into $T_k^{(1)}$, $T_k'$ and $S_r'$ by $\mu_k^{(1)}$, $\mu_k'$ and $\nu_r'$.
	Then by the assumption on segments, $\mu_1 \geqslant \mu_h \geqslant \mu_{h+1}^{(1)} \geqslant \mu_{h+2}' \geqslant \cdots \geqslant \mu_{r-1}' \geqslant \nu_r'$.
	In other words, $\left( \bigsqcup_{k \leqslant h} T_k\right) \sqcup
	T_{h+1}^{(1)} \sqcup
	\left ( \bigsqcup_{h+1 < k < r-1} T_k' \right)$
	is the compatible partition of $\nu$-antitebleau that $\left( \bigsqcup_{k<r-1} T_k \right) \sqcup T_{r-1}^{(1)}$ is equivalent to.
	If we want to get $T\sqcup S_r \nsim 0$, it suffices that $T_{r-1}' \sqcup S_r' \nsim 0$, 
	since there is no need to adjust the $\nu$-filling on $\left( \bigsqcup_{k \leqslant h} T_k\right) \sqcup
	T_{h+1}^{(1)} \sqcup
	\left ( \bigsqcup_{h+1 < k < r-1} T_k' \right) \sqcup
	S_r'$ anymore.
	According to Lemma \ref{overlap-geq-sing}, we just need
	\[\mu_{r-1; i}' \stackrel? \geqslant \nu_{r; i}'.\]
	
	By \eqref{operation-sub}, 
	\[\begin{aligned}
	\mu_{r-1; i}^{(1)} &= \mu_{r-1; i} - \min\{ 0, \mu_{r-1; j} - \nu_{r; j} \mid j> i\},\\
	\mu_{r-1; i}' &= \mu_{r-1; i}^{(1)} + \min \{0, \mu_{r-2; j} - \mu_{r-1; j}^{(1)} \mid j \geqslant i\}.
	\end{aligned}\]
	We can show $\mu_{r-1; i}' \geqslant \mu_{r-1; i}$ as follows.
	By the first equation above, $\mu_{r-1; i}^{(1)} \geqslant \mu_{r-1; i}$, and $\forall j \geqslant i$,
	\[\begin{aligned}
	\mu_{r-1; i}^{(1)} - \mu_{r-1; j}^{(1)} + \mu_{r-2; j} =&
	\mu_{r-1; i} - \mu_{r-1; j} + \mu_{r-2; j} +\\
	&-
	\min\{0, \mu_{r-1; k} - \nu_{r; k} \mid k>i\} +
	\min\{0, \mu_{r-1; k} - \nu_{r; k} \mid k>j\}.
	\end{aligned}\]
	Since $\mu_{r-1; j} \leqslant \mu_{r-2; j}$ (by the condition on $T$), and 
	$\min\{0, \mu_{r-1; k} - \nu_{r; k} \mid k>i\}
	\leqslant
	\min\{0, \mu_{r-1; k} - \nu_{r; k} \mid k>j\}$,
	it follows
	\[\mu_{r-1; i}^{(1)} - \mu_{r-1; j}^{(1)} + \mu_{r-2; j} \geqslant \mu_{r-1; i}.\]
	Consequently,
	\[\mu_{r-1; i}' = \min\{ \mu_{r-1; i}^{(1)} , \mu_{r-1; i}^{(1)} - \mu_{r-1; j}^{(1)} + \mu_{r-2; j} \mid j\geqslant i \} \geqslant \mu_{r-1; i}.\]
	
	On the other hand, by \eqref{eq: operation-sub-k+1}, we have
	\[\nu_{r; i}' = \nu_{r; i} + \min\{ 0, \mu_{r-1; j} - \nu_{r; j} \mid j \geqslant i \} \leqslant \mu_{r-1; i}.\]
	Then we obtain that $\mu_{r-1; i}' \geqslant \mu_{r-1; i} \geqslant \nu_{r; i}'$.
\end{proof}
	
	We finished \textbf{Step 1} in subsection \ref{general-case}.

\subsection{How the partitions are changed by a small segment?}
	Let $\sigma = (r, r+1) \in \Sigma_{r+1}$.
	Assume $\underline p^r = (p_1, \cdots, p_r)$ and $\underline p^{\sigma, r} = ( p_1, \cdots, p_{r-1}, p_r^\sigma)$ parametrize non-zero representations, we shall prove that $\underline p$ also does;
	this is \textbf{Step 2} in subsection \ref{general-case}.
	
	Let's rephrase it in terms of tableaus.
	Denote the compatible partition of $\nu$-quasitableau constructed from $\underline p^\sigma \in \MR^\sigma$ by $S^\sigma = \bigsqcup_{k=1}^{r+1} S_k^\sigma$,
	and the type of $\nu$-filling on $S_k^\sigma$ by $\nu_{k; i}^\sigma$.
	Then $\bigsqcup_{k<r} S_k^\sigma = \bigsqcup_{k<r} S_k$ is the partition given by $\underline p^{r-1}$.
	It's harmless to assume it is adjusted to a compatible partition $T = \bigsqcup_{k<r} T_k$ of antitableau.
	Denote the segment filling into $T_k$ by $\mu_k$, its type by $\mu_{k; i}$,
	and we know from Lemma \ref{anti-by-type}, that 
	$\mu_1 \geqslant \cdots \geqslant \mu_{r-1}$, 
	and $\mu_{k; i} \geqslant \mu_{k+1; i}$.
	
\begin{prop}\label{upper-bound-control}
	If the $\nu$-quasitableaus over $T \sqcup S_r$, $T\sqcup S_r^\sigma$ and $S_r \sqcup S_{r+1} = S_r^\sigma \sqcup S_{r+1}^\sigma$ are all equivalent to non-zero ones, then so is $T\sqcup S_r \sqcup S_{r+1}$.
\end{prop}	
\begin{proof}
	According to Proposition \ref{equiv-for-non-zero}, $T\sqcup S_r^\sigma \nsim 0$ means $\forall j$, 
	\[\begin{aligned}
	\nu_{r; j}^\sigma \leqslant &
	\min\left\{ 
	\sum_{s=1}^{r-h-1} \mu_{r-s; j_{s-1}} - \mu_{r-s; j_s} + \mu_{h; j_{r-h-1}} \Bigm | 
	j = j_0 > \cdots > j_{r-h-1}
	\right \},\\
	\nu_{r; j}^\sigma - \nu_{r; 0}^\sigma \leqslant&  
	\min\left\{ 
	\sum_{s=1}^{t} \mu_{r-s; j_{s-1}} - \mu_{r-s; j_s} \Bigm | 
	j = j_0 > \cdots > j_t =0, 1 \leqslant t \leqslant r-h-1
	\right \},
	\end{aligned}\]
	where $h\leqslant r-1$ satisfies $\mu_1, \cdots, \mu_h > \nu_r^\sigma$, $\mu_{h+1}, \cdots, \mu_{r-1} \subset \nu_r^\sigma$. 
	By \eqref{fulfill-type-2}
	\[\nu_{r; j}^\sigma = \nu_{r; j} - \min\{ \nu_{r; i} - \nu_{r+1; i} \mid i > j\},\]
	so
\begin{subequations}\label{upper-bound-have}
	\begin{align}
	&\nu_{r+1; i} 
	\leqslant \label{upper-bound-have-1}\\
	&\min\left\{ \nu_{r; i} - \nu_{r; j_0} +
	\sum_{s=1}^{r-h-1} \mu_{r-s; j_{s-1}} - \mu_{r-s; j_s} + \mu_{h; j_{r-h-1}} \Bigm | 
	i> j_0 > \cdots > j_{r-h-1}
	\right \},
	\notag\\
	\nu_{r+1; i}- &\nu_{r+1; 0} 
	\leqslant \label{upper-bound-have-2}\\
	&\min\left\{ \nu_{r; i} - \nu_{r; j_0} +
	\sum_{s=1}^{t} \mu_{r-s; j_{s-1}} - \mu_{r-s; j_s} \Bigm | 
	i > j_0 > \cdots > j_t =0, 0 \leqslant t \leqslant r-h-1
	\right \}.
	\notag
	\end{align}
\end{subequations}
	Here \eqref{upper-bound-have-2} also contains the equality $\nu_{r+1; i} - \nu_{r+1; 0} \leqslant \nu_{r; i} - \nu_{r; 0}$ following from $S_r \sqcup S_{r+1} \nsim 0$ and Lemma \ref{overlap-geq-sing}.
	
	Now apply Trapa' algorithm to $T\sqcup S_r \nsim 0$, and 
	denote the resulted partition by $\bigsqcup_{k\leqslant r} T_k'$, 
	the segment filling into $T_k'$ by $\mu_k'$ .
	According to Proposition \ref{equiv-for-non-zero} again, $\left( \bigsqcup_{k\leqslant r} T_k' \right) \sqcup S_{r+1} \stackrel? \nsim 0$ is equivalent to 
\begin{subequations}\label{upper-bound-want}
	\begin{align}
	\nu_{r+1; i} \stackrel?\leqslant &
	\min\left\{ 
	\sum_{s=0}^{r-h-1} \mu_{r-s; j_{s-1}}' - \mu_{r-s; j_s}' + \mu_{h; j_{r-h-1}}' \Bigm | 
	i = j_{-1} > \cdots > j_{r-h-1}
	\right \},\\
	\nu_{r+1; i} - \nu_{r+1; 0} \stackrel?\leqslant&  
	\min\left\{ 
	\sum_{s=0}^{t} \mu_{r-s; j_{s-1}}' - \mu_{r-s; j_s}' \Bigm | 
	i = j_{-1} > \cdots > j_t =0, 1 \leqslant t \leqslant r-h-1
	\right \}.
	\end{align}
\end{subequations}
	Consequently, what we need is \eqref{upper-bound-have} $\Longrightarrow$ \eqref{upper-bound-want},
	which follows from the next Proposition.
\end{proof}

\begin{prop}\label{control}
\begin{subequations}
	$\forall i= j_{-1} > \cdots > j_{r-h-1}$, there is
	\begin{align}\label{control-1}
	&\sum_{s=0}^{r-h-1} \mu_{r-s; j_{s-1}}' - \mu_{r-s; j_s}' + \mu_{h; j_{r-h-1}}' \\
	\geqslant &
	\min\left\{ \nu_{r; i} - \nu_{r; j_0'} +
	\sum_{s=1}^{r-h-1} \mu_{r-s; j_{s-1}'} - \mu_{r-s; j_s'} + \mu_{h; j_{r-h-1}} \Bigm | 
	i> j_0' > \cdots > j_{r-h-1}
	\right \}.	\notag
	\end{align}
	$\forall i= j_{-1} > \cdots > j_t=0$ with $0 \leqslant t \leqslant r-h-1$, there is
	\begin{align}\label{control-2}
	&\sum_{s=0}^t \mu_{r-s; j_{s-1}}' - \mu_{r-s; j_s}' \\
	\geqslant &
	\min\left\{ \nu_{r; i} - \nu_{r; j_0'} +
	\sum_{s=1}^{t'} \mu_{r-s; j_{s-1}'} - \mu_{r-s; j_s'} \Bigm | 
	i > j_0' > \cdots > j'_{t'} =0, 0 \leqslant t' \leqslant r-h-1
	\right \}.	\notag
	\end{align}	
\end{subequations}
\end{prop}
	To prove it, we have to study Trapa's algorithm on $T\sqcup S_r$ in detail.
	Take $g\leqslant r-1$ that $\mu_1, \cdots, \mu_g > \nu_r$, $\mu_{g+1}, \cdots, \mu_{r-1} \supset \nu_r$.
	For convenience, let $T_r^{(1)}:= S_r$. 
	Apply Trapa's operation to $T_k \sqcup T_{k+1}^{(1)}$ and get $T_k^{(1)} \sqcup T_{k+1}'$ for $k = r-1, \cdots, g+1$ sequencially.
	Denote the segment filling into $T_k^{(1)}$, $T_k'$ by $\mu_k^{(1)}$, $\mu_k'$ .	
	Then by the assumption on segments, $\mu_1 \geqslant \cdots \geqslant \mu_g \geqslant 
	\mu_{g+1}^{(1)} \geqslant \mu_{g+2}' \geqslant \cdots \geqslant \mu_r'$.
	In other words, $\left( \bigsqcup_{k \leqslant g} T_k\right) \sqcup
	T_{g+1}^{(1)} \sqcup
	\left ( \bigsqcup_{g+1 < k \leqslant r} T_k' \right)$
	is the compatible partition of $\nu$-antitebleau that $T  \sqcup S_r = \left( \bigsqcup_{k<r} T_k \right) \sqcup T_{r}^{(1)}$ is equivalent to.

	From \eqref{operation-sup} we know
\begin{equation}\label{process}
	\begin{aligned}
	\mu_{k; i}^{(1)} =&
	\mu_{k; i} - \min\{0, \mu_{k; j} - \mu_{k+1; j}^{(1)} \mid j \leqslant i \},\\
	\mu_{k+1; i}' = &
	\mu_{k+1; i}^{(1)} + \min\{0, \mu_{k; j} - \mu_{k+1; j}^{(1)} \mid j < i \}.
	\end{aligned}
\end{equation}
	Proposition \ref{control} would follow easily from \eqref{process} and the following critical inequalities.
\begin{lem}\label{key}
	$\forall j_{s-1} > j_s > j_{s+1}$,
	\[\begin{aligned}
	& \mu_{r-s; j_{s-1}}' - \mu_{r-s; j_s}' + \mu_{r-s-1; j_s}^{(1)} - \mu_{r-s-1; j_{s+1}}^{(1)} \\
	\geqslant 
	& \min\{ \mu_{r-s; j_{s-1} }^{(1)} - \mu_{r-s; j_s'}^{(1)} + \mu_{r-s-1; j_s'} - \mu_{r-s-1; j_{s+1}} \mid
	j_{s-1} >  j_s' > j_{s+1} \}.
	\end{aligned}\]
\end{lem}

\begin{proof}[Proof of Proposition \ref{control}]
	By assumptions on $\nu_r$, $\nu_{r+1}$, we have $g \geqslant h$. 
	Then the left hand side of \eqref{control-1} would be 
	\[\begin{aligned}
	&\left( \sum_{s=0}^{r-g-2} \mu_{r-s; j_{s-1}}' - \mu_{r-s; j_s}' \right)
	+\left( \mu_{g+1; j_{r-g-2}}^{(1)} - \mu_{g+1; j_{r-g-1}}^{(1)} \right)	\\
	+&\left(\sum_{s=r-g}^{r-h-1} \mu_{r-s; j_{s-1}} - \mu_{r-s; j_s} \right)
	+\mu_{h; j_{r-h-1}},
	\end{aligned}\]
	so \eqref{control-1} follows easily from Lemma \ref{key} by induction.
	Note that in our notations, $\mu_r^{(1)}$ and $\nu_r$ is the same segment filling into $T_r^{(1)} = S_r$.

	If $r-g-1 \leqslant t \leqslant r-h-1$, then the left hand side of \eqref{control-2} would be
	\[\left( \sum_{s=0}^{r-g-2} \mu_{r-s; j_{s-1}}' - \mu_{r-s; j_s}' \right)
	+\left( \mu_{g+1; j_{r-g-2}}^{(1)} - \mu_{g+1; j_{r-g-1}}^{(1)} \right)
	+\left(\sum_{s=r-g}^{t} \mu_{r-s; j_{s-1}} - \mu_{r-s; j_s} \right) ,
	\]
	so \eqref{control-2} also follows easily from Lemma \ref{key} by induction.
	If $t< r-g-1$,  then the left hand side of \eqref{control-2} would be
	\[\sum_{s=0}^{t} \mu_{r-s; j_{s-1}}' - \mu_{r-s; j_s}' ,\]
	with $r-t> g+1$, and we have to replace $\mu_{r-t, j_{t-1}}' - \mu_{r-t, j_t}'$ by $\mu_{r-t, j_{t-1}}^{(1)} - \mu_{r-t, j_t}^{(1)}$ together with some acceptable term. 
	Notice that $\mu_{r-t; j_t}' = \mu_{r-t; 0}' = \mu_{r-t; 0}^{(1)} = \nu_{r; 0} \leqslant \mu_{r-t -1; 0}$,
	so according to \eqref{process},
	\[\begin{aligned}
	\mu_{r-t; j_{t-1}}' - \mu_{r-t; j_t}' =&
	\mu_{r-t; j_{t-1}}^{(1)} 
	+ \min\{0, \mu_{r-t-1; j_t'} - \mu_{r-t; j_t'}^{(1)} \mid j_t'< j_{t-1}\}
	- \mu_{r-t; 0}^{(1)} \\
	= & 
	\min_{j_{t-1} > j_t' >0}	\{ \mu_{r-t; j_{t-1}}^{(1)} - \mu_{r-t; 0}^{(1)},
	\mu_{r-t; j_{t-1}}^{(1)} - \mu_{r-t; j_t'}^{(1)} + \mu_{r-t-1; j_t'} - \mu_{r-t; 0}^{(1)}\}\\
	\geqslant &
	\min_{ j_{t-1} > j_t' >0}	\{ \mu_{r-t; j_{t-1}}^{(1)} - \mu_{r-t; 0}^{(1)},
	\mu_{r-t; j_{t-1}}^{(1)} - \mu_{r-t; j_t'}^{(1)} + \mu_{r-t-1; j_t'} - \mu_{r-t-1; 0}\},	
	\end{aligned}\]
	and then \eqref{control-2} follows also from Lemma \ref{key} by induction.
\end{proof}

\begin{proof}[Proof of Lemma \ref{key}]
	For convenience, let
	\[\Delta_{r-s; i, j} = \min\{0, \mu_{r-s-1; l} - \mu_{r-s; l}^{(1)} \mid i \leqslant l \leqslant j\}.\]
	Then \eqref{process} can be simplified to 
	\[\mu_{r-s-1; i}^{(1)} =
	\mu_{r-s-1; i} - \Delta_{r-s;0, i},\quad
	\mu_{r-s; i}' = 
	\mu_{r-s; i}^{(1)} +\Delta_{r-s;0, i-1}.\]
	Hence,
	\[\begin{aligned}
	&\mu_{r-s; j_{s-1}}' - \mu_{r-s; j_s}' + \mu_{r-s-1; j_s}^{(1)} - \mu_{r-s-1; j_{s+1}}^{(1)}	\\
	=
	& \mu_{r-s; j_{s-1}}^{(1)} - \mu_{r-s; j_s}^{(1)} + \mu_{r-s-1; j_s} - \mu_{r-s-1; j_{s+1}}	\\
	&+ \Delta_{r-s; 0, j_{s-1} -1} - \Delta_{r-s; 0, j_s -1} - \Delta_{r-s; 0, j_s} + \Delta_{r-s; 0, j_{s+1}},
	\end{aligned}\]
	where by $j_{s-1} > j_s > j_{s+1}$,
	\[\Delta_{r-s; 0, j_{s-1} -1} \leqslant \Delta_{r-s; 0, j_s -1} \leqslant \Delta_{r-s; 0, j_s} \leqslant \Delta_{r-s; 0, j_{s+1}}.\]
	If they are all the same, then the conclusion holds trivially.
	
	If $\Delta_{r-s; 0, j_{s-1} -1}< \Delta_{r-s; 0, j_{s+1}}$, then 
	\[\Delta_{r-s; 0, j_{s-1} -1} = \Delta_{r-s; j_{s+1} + 1, j_{s-1} -1}< 0.\]
	Combine it with $\Delta_{r-s; 0, j_s -1} \leqslant  \Delta_{r-s; 0, j_{s+1}}$, and 
	\[\mu_{r-s-1; j_s}^{(1)} \geqslant \mu_{r-s; j_s}^{(1)}\]
	following from first equation of \eqref{process}, we deduce
	\[\begin{aligned}
	&\mu_{r-s; j_{s-1}}' - \mu_{r-s; j_s}' + \mu_{r-s-1; j_s}^{(1)} - \mu_{r-s-1; j_{s+1}}^{(1)} \\
	=
	& \mu_{r-s; j_{s-1}}^{(1)} - \mu_{r-s; j_s}^{(1)} + \mu_{r-s-1; j_s}^{(1)} - \mu_{r-s-1; j_{s+1}} 
	+ \Delta_{r-s; 0, j_{s-1} -1} - \Delta_{r-s; 0, j_s -1} + \Delta_{r-s; 0, j_{s+1}}\\
	\geqslant 
	& \mu_{r-s; j_{s-1}}^{(1)} - \mu_{r-s-1; j_{s+1}} + \Delta_{r-s; j_{s+1} +1, j_{s-1} -1} \\
	=
	& \mu_{r-s; j_{s-1}}^{(1)} 
	+\min\{ -\mu_{r-s; j_s'}^{(1)} + \mu_{r-s-1; j_s'} \mid j_{s-1} > j_s' > j_{s+1} \} 
	- \mu_{r-s-1; j_{s+1}} 
	\end{aligned}\]
	Now the conclusion follows.	
\end{proof}